\documentclass[11pt]{article}
\usepackage[total={7in, 8in}]{geometry}
\usepackage{graphicx}
\usepackage[margin=0.5in]{caption}
\usepackage{comment}
\usepackage{amsmath, amsthm, latexsym, amssymb, color,cite,enumerate, physics, framed,dsfont,caption,subcaption,verbatim, empheq, cancel}
\usepackage{enumitem}
\usepackage{mathtools}
\usepackage{appendix}
\usepackage{braket}
\usepackage{bbm}

\usepackage{etoolbox}
\newcommand{\addQEDstyle}[2]{\AtBeginEnvironment{#1}{\pushQED{\qed}
\renewcommand{\qedsymbol}{#2}}\AtEndEnvironment{#1}{\popQED}}

\newtheorem{theorem}{Theorem} [section]
\newtheorem{lemma}[theorem]{Lemma}
\newtheorem{definition}[theorem]{Definition}
\newtheorem{corollary}[theorem]{Corollary}
\newtheorem{proposition}[theorem]{Proposition}

\newtheoremstyle{example}
{5pt}
{5pt}
{}
{}
{\bf}
{:}
{.5em}
{}

\theoremstyle{example}
\newtheorem{example}{Example}
\addQEDstyle{example}{$\triangle$}
\theoremstyle{exercise}

\newtheorem{remark}{Remark}

\newcommand*{\myproofname}{Proof}

\renewcommand\Re{\operatorname{Re}}

\newcommand\Gl{\operatorname{Gl}} 

\newcommand\supp{\operatorname{Supp}}

\renewcommand\det{\operatorname{det}}

\newcommand{\Aut}{\mathrm{Aut}}
\newcommand{\Proj}{\operatorname{Proj}}
\newcommand{\Tor}{\operatorname{Tor}}
\newcommand{\E}{\mathrm{E}}
\newcommand{\Cov}{\mathrm{Cov}}
\usepackage{hyperref}
\usepackage{tensor}
\usepackage{xcolor}
\hypersetup{
	colorlinks,
	linkcolor={black!50!black},
	citecolor={blue!50!black},
	urlcolor={blue!80!black}
}

\author{Evan Randles\thanks{Corresponding author: erandles@colby.edu}\\
\normalsize  Department of Mathematics\\
\normalsize Colby College\\
\and
Yutong Yan\thanks{yyan9@nd.edu}\\
\normalsize  Department of Mathematics\\
\normalsize University of Notre Dame
}

\title{The predicable dance of random walk \\ \large{local limit theorems on finitely-generated abelian groups}}
\date{}
\begin{document}
\maketitle
\abstract{In random walk theory, it is customary to assume that a given walk is irreducible and/or aperiodic. While these prevailing assumptions make particularly tractable the analysis of random walks and help to highlight their diffusive nature, they eliminate a natural phenomenon: the \textit{dance}. This dance can be seen, for example, in the so-called simple random walk on the integers where a random walker moves back and forth between even and odd integers. It can also be seen in the random walk on the integer lattice that takes only those steps available to a knight on a chess board. In this work, we develop a general Fourier-analytic method to describe random walks on finitely-generated abelian groups, making no assumptions concerning aperiodicity or irreducibility. Our main result is a local central limit theorems that describes the large-time behavior of the transition probabilities for any random walk whose driving measure has finite second moments. Generalizing the ubiquitous Gaussian approximation, our asymptotic is given as a product of two functions, one that describes the diffusion and another that describes the dance. Interestingly, our ``dance function" is gotten as a Haar integral over a subgroup of the group's Pontryagin dual. Our results recapture many long-standing results, especially the local limit theorems made famous by F. Spitzer's classic text, \textit{Principles of Random Walk}.}\\

\noindent{\small\bf Keywords:} random walks on abelian groups, local limit theorems\\

\noindent{\small\bf Mathematics Subject Classification:} Primary 60F05, 60G50, 60J10; Secondary 60E10, 60B15.

\section{Introduction}

In this article, we develop a Fourier-analytic method to study random walks on finitely-generated abelian groups. Our approach generalizes and departs from more standard ones that assume walks be irreducible and/or aperiodic \cite{KestenSpitzer1965,Spitzer,Woess2000}. While some approaches do consider periodic walks, large-time asymptotic behavior is often obtained by replacing a periodic walk by a lazy aperiodic one for which the analysis is far more tractable. Irreducible walks are less frequently considered as, in the words of F. Spitzer, they are said to simply be \textit{posed on the wrong group}\cite{Spitzer}. By contrast, our perspective here is that, on a group $G$, one is handed a probability $p$ that drives a random walk for which aperiodicity or irreducibility are not a priori obvious. From this perspective and in the context of finitely-generated abelian groups, our method uses $p$'s Fourier transform/characteristic function to describe the set of points reachable by the random walk at each step. This description is captured by a ``dance function" gotten as a Haar integral over a subgroup of $G$'s Pontryagin dual. The dance function always exists and is useful whether or not a random walk is irreducible/aperiodic. For those probabilities satisfying the mild assumption of having finite second moments, we obtain large-time asymptotics for the transition probabilities of random walks in the form of local (central) limit theorems. \\

\noindent The subject of local limit theorems is well established and we refer the reader to \cite{McDonald2005} for a historical account (see also the introduction of \cite{DiaconisHough2021}). Classical results for random walks on $\mathbb{Z}^d$ can be found in \cite{Spitzer, Woess2000, LawlerLimic2010}. There has also been much recent progress obtaining local limit theorems on certain non-abelian groups, much of it focusing on the Heisenberg group \cite{Alexopoulos2002, Breuillard2005, DiaconisHough2021}. For irreducible and aperiodic walks on $\mathbb{Z}^d$, the classical local limit theorem shows that the transition probabilities are well approximated by scaled Gaussian densities.  Local limit theorems have also been proven under some weaker assumptions.  For instance, G. Lawler and V. Limic \cite{LawlerLimic2010} establish local limit theorems for irreducible, symmetric, and finite-range walks on $\mathbb{Z}^d$ (with sharp Gaussian-type error). While, F. Spitzer obtains local limit theorems only for irreducible and aperiodic walks on $\mathbb{Z}^d$ in \cite{Spitzer}, he does treat several related results for his so-called genuinely $d$-dimensional walks; local limit theorems for these walks were later obtained in \cite{RSC17}. A more complete (and precise) discussion of these results can be found in Section \ref{sec:Main}. Our main result, Theorem \ref{thm:MainLLT}, gives an exhaustive account of local limit theorems on finitely-generated abelian groups assuming only a finite second-moment condition for $p$. This result extends/recaptures the known results on $\mathbb{Z}^d$, save for obtaining sharp forms of error which will be treated in a forthcoming paper. The novelty of our local limit theorems is that attractors are always given as a product of two functions, one (the dance function) describing the support of the walk and the other (Gaussian or uniform) describing the diffusion.\\

\noindent We now set the stage and and introduce some basic concepts and terminology. Throughout this article, $G$ will denote a finitely-generated abelian group. A probability distribution on $G$ is, by definition, a non-negative function on $G$ with $\sum_{x\in G}p(x)=1$. The set of such probability distributions on $G$ will be denoted by $\mathcal{M}(G)$. It is well known that every $p\in\mathcal{M}(G)$ generates a Markov process on $G$ called the \textit{random walk on $G$ driven by $p$} \cite{LawlerLimic2010,Woess2000}. The random walk has transition kernels given by 
\begin{equation}\label{eq:Markov}
    k_n(x,y)=p^{(n)}(y-x)
\end{equation}
for $x,y\in G$ and $n\in\mathbb{N}_+=\{1,2,\dots,\}$ where $p^{(n)}\in\mathcal{M}(G)$ denotes the $n$th convolution power of $p$ and is defined iteratively by setting $p^{(1)}=p$ and, for $n\geq 2$,
\begin{equation}\label{eq:ConvDef}
    p^{(n)}(x)=\sum_{y\in G}p^{(n-1)}(x-y)p(y)
\end{equation}
for $x\in G$. Through the identity \eqref{eq:Markov}, all information about the random walk is captured by $p$'s convolution powers. Given a symmetric finite generating set $\mathcal{G}=\{x_1,x_2,\dots,x_N\}$ of $G$, the word norm associated to $\mathcal{G}$ on $G$ is the function assigning $x\in G$ to the natural number
\begin{equation*}
\abs{x}_\mathcal{G}=\min\left\{n_1+n_2+\cdots+n_N:x=n_1x_1+n_2x_2+\cdots+n_N x_N\,\,\mbox{for}\,\,n_1,n_2,\dots,n_N\in\mathbb{N}\right\}.
\end{equation*}
It is well known that the word norms/metrics associated to any two finite symmetric generating sets are (bi-Lipschitz) equivalent \cite{Woess2000}. With this observation, we say that a probability distribution $p\in\mathcal{M}(G)$ has \textit{finite second moments} if
\begin{equation*}
\sum_{x\in G}\abs{x}_{\mathcal{G}}^2\, p(x)<\infty 
\end{equation*}
for any (and hence every) finite symmetric generating set $\mathcal{G}$ of $G$. The set of all probability distributions with finite second moments will be denoted by $\mathcal{M}_2(G)$. We remark that all finitely-supported probability distributions are members of $\mathcal{M}_2(G)$;  these correspond to finite-range random walks on $G$. We recall that the random walk on $G$ driven by $p\in\mathcal{M}(G)$ is said to be \textit{irreducible} if, for every $x\in G$, there is an $n$ for which $p^{(n)}(x)>0$. We say that a random walk is \textit{periodic of period }$s$ if $s=\gcd\{n\in\mathbb{N}_+:p^{(n)}(0)>0\}$; of course, for an irreducible walk, this is always finite and well defined in the sense that it is independent the walk's starting point. An \textit{aperiodic walk} is a walk with period $s=1$. Motivated by \cite{Spitzer}, a random walk on $\mathbb{Z}^d$ driven by $p\in\mathcal{M}(\mathbb{Z}^d)$ is said to be \textit{genuinely $d$ dimensional} if $p$ cannot be supported in any proper affine hyperplane of $\mathbb{R}^d$ (see Definition \ref{def:dimension}).\\

\noindent As we discussed above, the most we will ask of a distribution $p\in\mathcal{M}(G)$ is that it have finite second moments and many of our results pertain to the general class $\mathcal{M}(G)$ (See Section \ref{sec:W&D}). In particular, our main results require no hypotheses concerning irreducibility or aperiodicity of the associated random walk. Correspondingly, we ask no special relation between $\supp(p)$ and any generating set of $G$. As the convolution formula \eqref{eq:ConvDef} dictates, when $p$ is not aperiodic/irreducible, the support of $p^{(n)}$ (and so the location of a random walker started at $0$) cycles through certain proper cosets of $G$ as $n$ increases.  We shall refer to this cycling loosely as the ``dance" of random walk. As a canonical example (we give many others below), simple random walk on $\mathbb{Z}$ is supported on odd integers at odd steps and at even integers on even steps and so it ``dances" through cosets of even integers while spreading out in space. It is the point of this article to show that, on any finitely-generated abelian group $G$ and for any $p\in\mathcal{M}(G)$, the Fourier transform ``sees" this dance and gives a way to easily capture it. \\

\noindent To illustrate the assertions above, we shall focus the remainder of this introduction on the case in which $G$ is a finite abelian group (and so, necessarily, $\mathcal{M}_2(G)=\mathcal{M}(G)$). We shall make use of the basic tools of harmonic analysis to understand the convolution powers of $p$ and hence the asymptotic behavior of random walk. Denoting by $\widehat{G}$ the Pontryagin dual of $G$ (with characters written as $\chi_\xi(\cdot):G\to \mathbb{S}:=\{z\in\mathbb{C}:\abs{z}=1\}$ for $\xi\in\widehat{G}$), we have
\begin{equation}\label{eq:FiniteFTConvIdentity}
p^{(n)}(x)=\frac{1}{\abs{G}}\sum_{\xi\in \widehat{G}}\widehat{p}(\xi)^n\chi_{\xi}(-x)
\end{equation}
for $x\in G$ and $n\in\mathbb{N}$ where $\widehat{p}$ denotes $p$'s Fourier transform/characteristic function defined by
\begin{equation*}
\widehat{p}(\xi)=\sum_{x\in G}p(x)\chi_{\xi}(x)
\end{equation*}
for $\xi\in\widehat{G}$. Since $p$ is a probability distribution, 
\begin{equation*}
    \abs{\widehat{p}(\xi)}\leq\sum_{x\in G}\abs{p(x)\chi_{\xi}(x)}=\sum_{x\in G}p(x)=1.
\end{equation*}
for all $\xi\in \widehat{G}$. Through the Fourier-Convolution identity \eqref{eq:FiniteFTConvIdentity}, it is easy to see that the large-$n$ behavior of random walk depends essentially on the set of points $\xi\in\widehat{G}$ for which $\abs{\widehat{p}(\xi)}=1$. Let's denote this set by
\begin{equation*}
\Omega(p)=\{\xi\in\widehat{G}:\abs{\widehat{p}(\xi)}=1\}=\{\xi\in\widehat{G}:\widehat{p}(\xi)\in \mathbb{S}\}
\end{equation*}
and observe that
\begin{eqnarray}\label{eq:IntroSplit}
    p^{(n)}(x)&=&\frac{1}{\abs{G}}\sum_{\xi\in\Omega(p)}\widehat{p}(\xi)^n\chi_{\xi}(-x)+\frac{1}{\abs{G}}\sum_{\xi\in \widehat{G}\setminus\Omega(p)}\widehat{p}(\xi)^n\chi_{\xi}(-x)\\\nonumber
    &=&\frac{1}{\abs{G}}\sum_{\xi\in\Omega(p)}\widehat{p}(\xi)^n\chi_{\xi}(-x)+O(\rho^n)
\end{eqnarray}
uniformly for $x\in G$ as $n\to\infty$ where
\begin{equation*}
    0\leq \rho=\max\left\{\abs{\widehat{p}(\xi)}:\xi\in\widehat{G}\setminus \Omega(p)\right\}<1;
\end{equation*}
we will take this constant to be zero whenever the complement of $\Omega(p)$ in $\widehat{G}$ is empty. If we define
\begin{equation}\label{eq:IntroLLT1}
    \Theta_p(n,x)=\sum_{\xi\in\Omega(p)}\widehat{p}(\xi)^n\chi_{\xi}(-x),
\end{equation}
we arrive at the following result.
\begin{proposition}\label{prop:IntroLLT}
    Let $G$ be a finite abelian group, $p$ a probability distribution on $G$, and define $\Theta_p$ as above. Then, for some $0\leq\rho<1$, 
    \begin{equation*}
p^{(n)}(x)=\frac{\Theta_p(n,x)}{\abs{G}}+O(\rho^n)
    \end{equation*}
    uniformly for $x\in G$ as $n\to\infty$. 
\end{proposition}

\noindent We illustrate this proposition in a particularly simple situation, a random walk on the clock $\mathbb{Z}_{12}=\mathbb{Z}/12\mathbb{Z}$. Several other examples, including cases where $G$ is infinite, are explored in Section \ref{sec:Examples}.
\begin{example}\label{ex:Intro}
Consider $p\in\mathcal{M}(\mathbb{Z}_{12})$ defined by
\begin{equation*}
p(x)=\begin{cases}
1/2 &\mbox{if}\quad x=-1,2\\
0 &\mbox{else}
\end{cases}
\end{equation*}
for $x\in\mathbb{Z}_{12}$. Calculating
\begin{equation*}
\widehat{p}(\xi)=\frac{1}{2}\left(e^{-2\pi i\xi/12}+e^{4\pi i\xi/12}\right)=\frac{e^{-\pi i\xi/6}}{2}\left(1+e^{i\pi\xi/2}\right)
\end{equation*}
we find easily that $\Omega(p)=\{0,4,8\}\subseteq\widehat{\mathbb{Z}_{12}}=\mathbb{Z}_{12}$, $\rho=\max\{\abs{\widehat{p}(\xi)}:\xi\in \mathbb{Z}_{12}\setminus \Omega(p)\}=1/\sqrt{2}$, and 
\begin{equation*}
\Theta_p(n,x)=\sum_{\xi=0,4,8}\widehat{p}(\xi)^ne^{-2\pi i x\xi/12}=1+e^{-2\pi i(x+n)/3}+e^{-4\pi i(x+n)/3}=\begin{cases}
3 &\mbox{if}\quad 3\vert(x+n)\\
0 &\mbox{otherwise}
\end{cases}
\end{equation*}
for $n\in\mathbb{N}$ and $x\in\mathbb{Z}_{12}$. With this, Proposition  \ref{prop:IntroLLT} gives
\begin{equation*}
p^{(n)}(x)=\frac{\Theta_p(n,x)}{12}+O(2^{-n/2})=\begin{cases} \frac{1}{4} &\mbox{if}\quad 3\vert(n+x)\\ 0 & \mbox{otherwise}\end{cases}+O(2^{-n/2})
\end{equation*}
uniformly for $x\in\mathbb{Z}_{12}$ as $n\to\infty$. This result is illustrated in Figure \ref{fig:IntroExample}. In looking at the figure, we observe that points reachable by the random walk are concentrated on four equally spaced elements of $\mathbb{Z}_{12}$ that rotate counterclockwise as $n$ increases. In particular, we see that this random walk is irreducible but not aperiodic. In fact, knowing $\Theta_p$ is enough to decide all of this, as we shall see.

\begin{table}[!h]
  \centering
  \begin{tabular}{  | c | c | c | }
    \hline
    $n=1$ & $n=2$ & $n=3$ \\ \hline
    
    \begin{minipage}{.25\textwidth}
      \includegraphics[width=\linewidth]{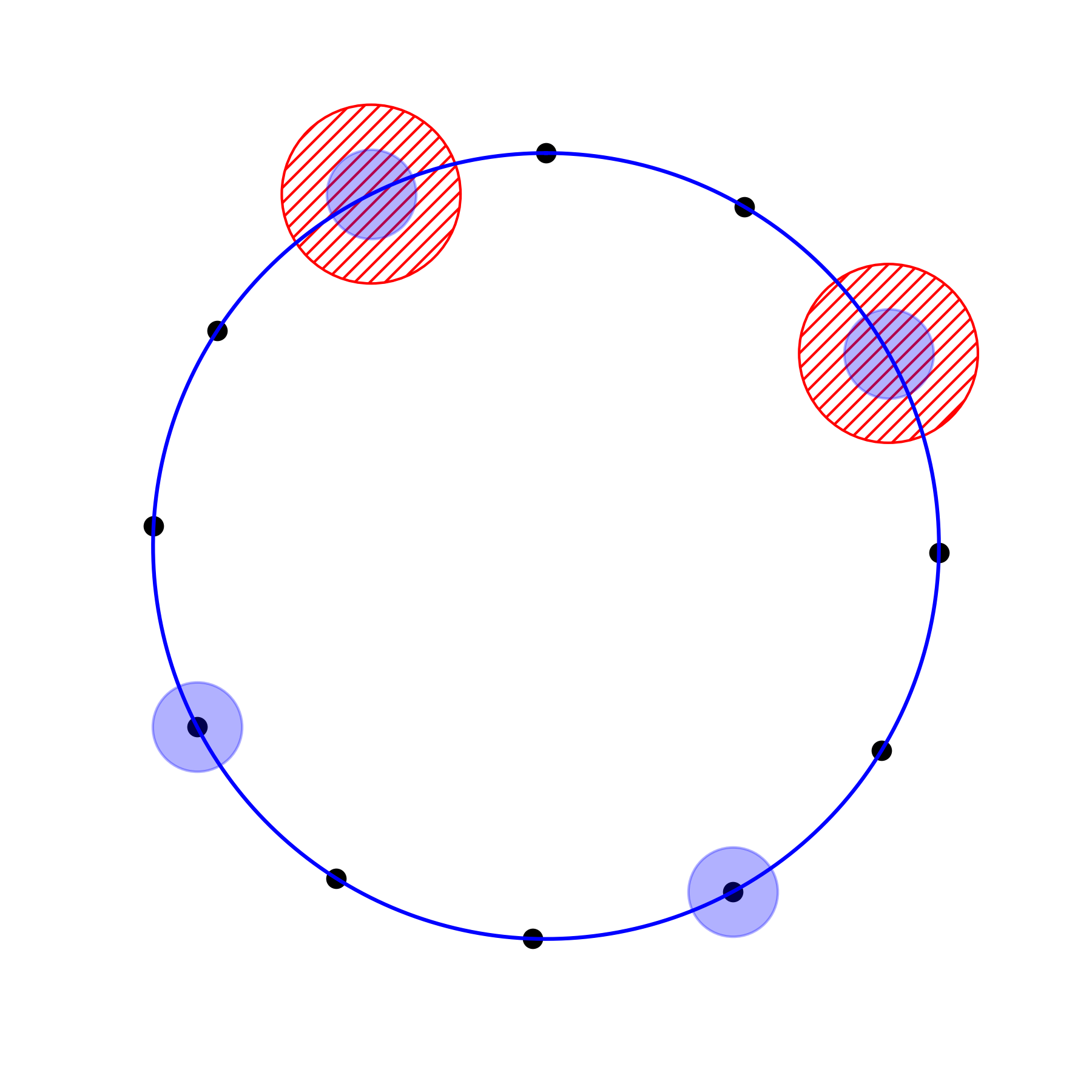}
    \end{minipage}
    &
    \begin{minipage}{.25\textwidth}
      \includegraphics[width=\linewidth]{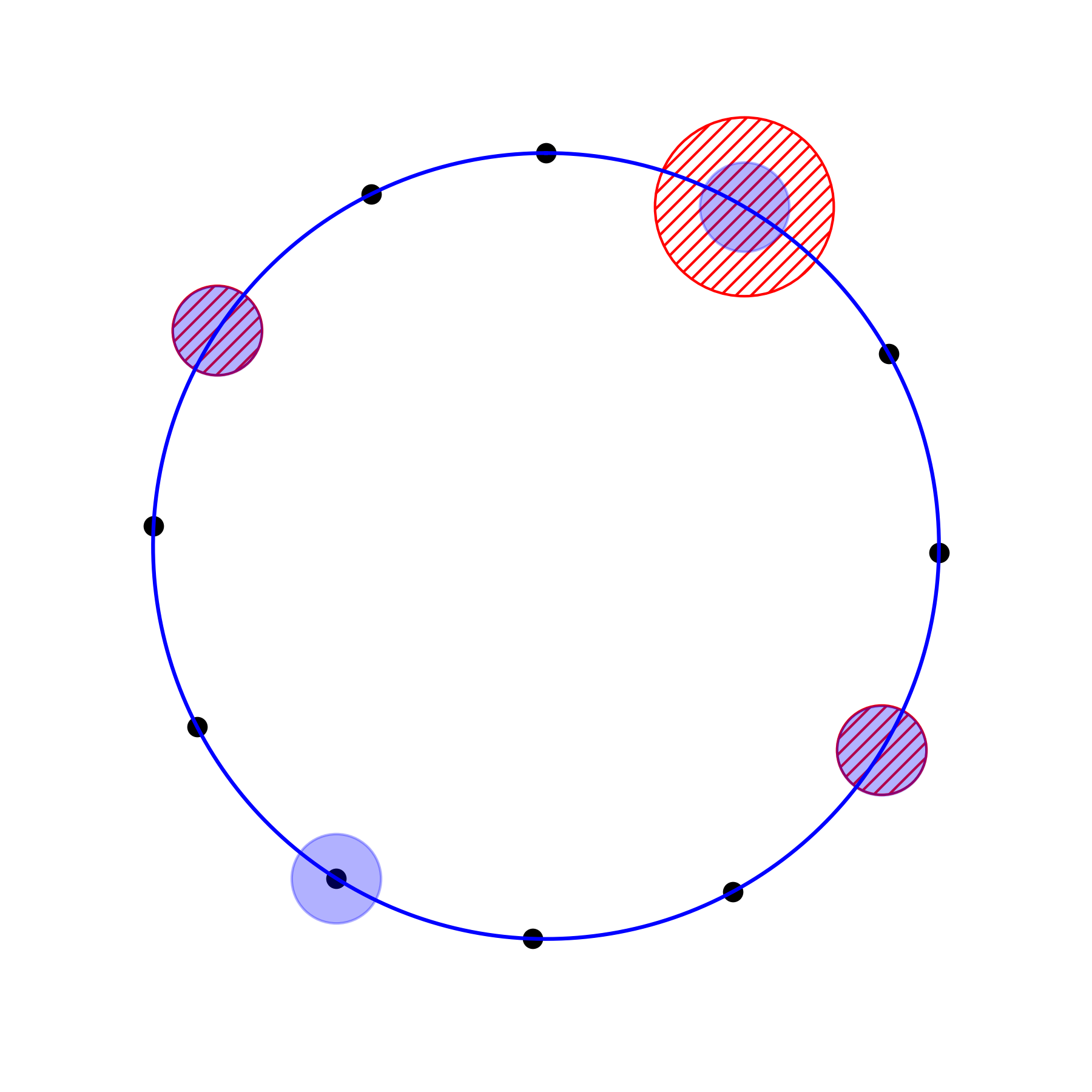}
    \end{minipage}
	&
      \begin{minipage}{.25\textwidth}
      \includegraphics[width=\linewidth]{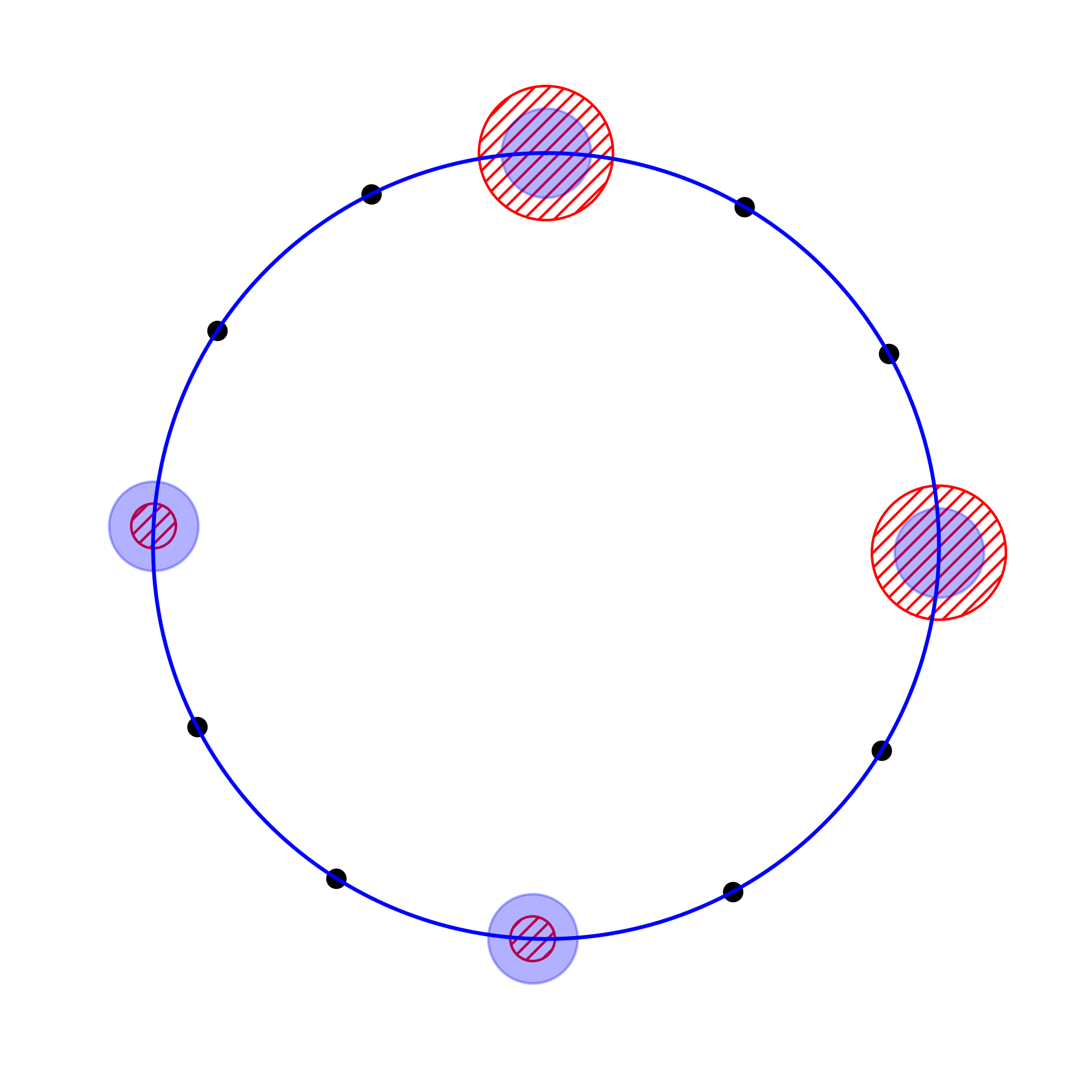}
    \end{minipage}
        \\ \hline
       $n=4$ & $n=5$ & $n=6$ \\ \hline 
    \begin{minipage}{.25\textwidth}
      \includegraphics[width=\linewidth]{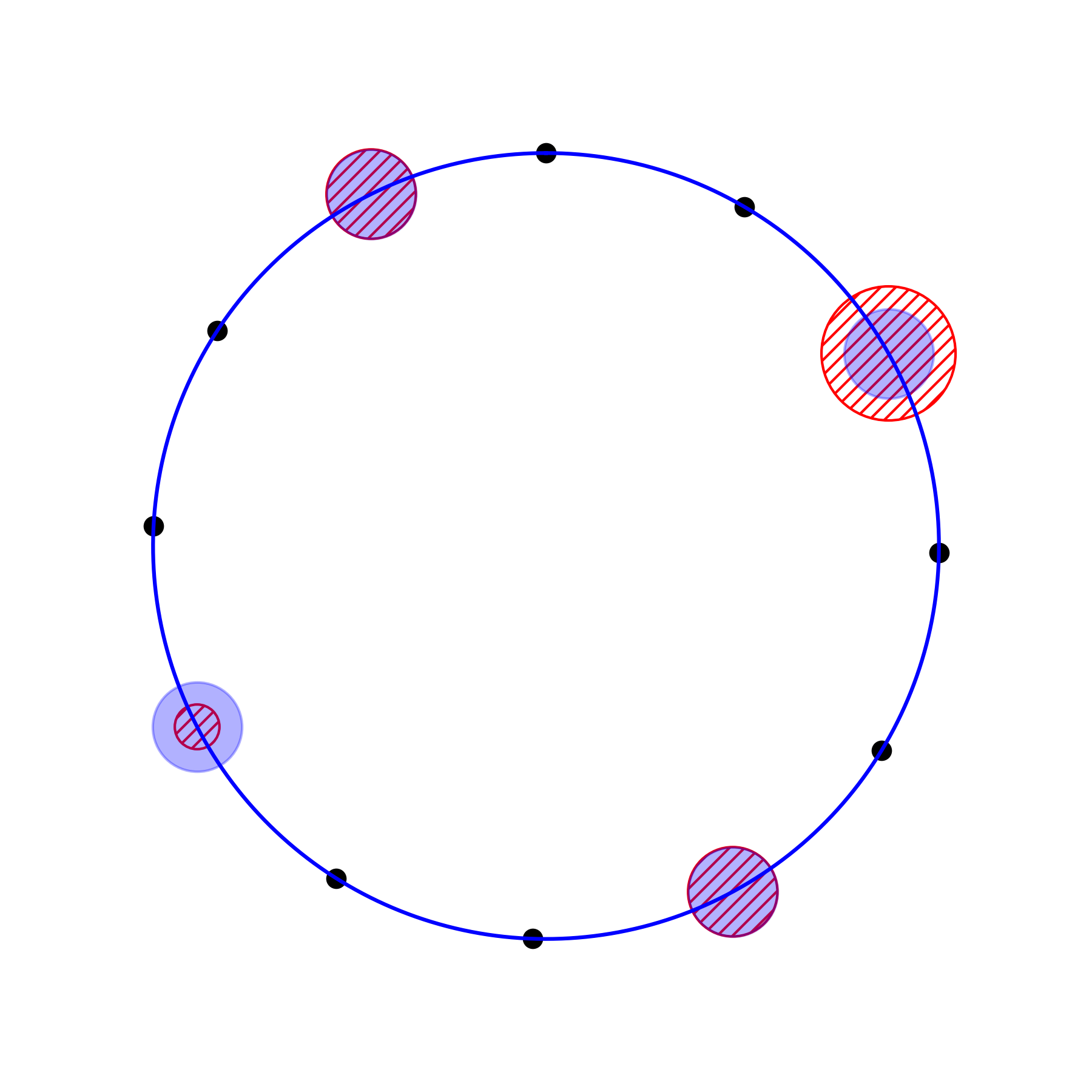}
    \end{minipage}
	&
    \begin{minipage}{.25\textwidth}
      \includegraphics[width=\linewidth]{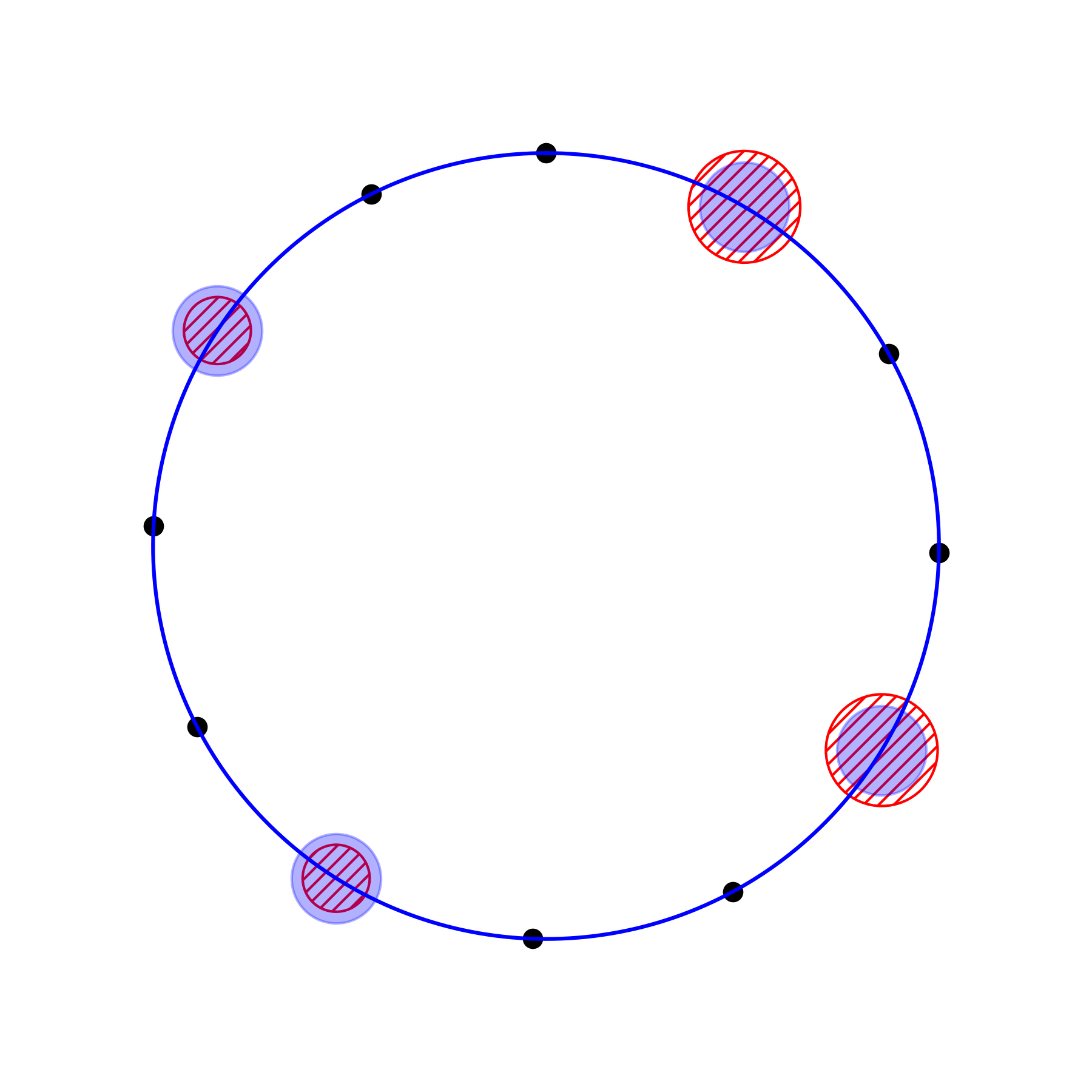}
    \end{minipage}
	&
      \begin{minipage}{.25\textwidth}
      \includegraphics[width=\linewidth]{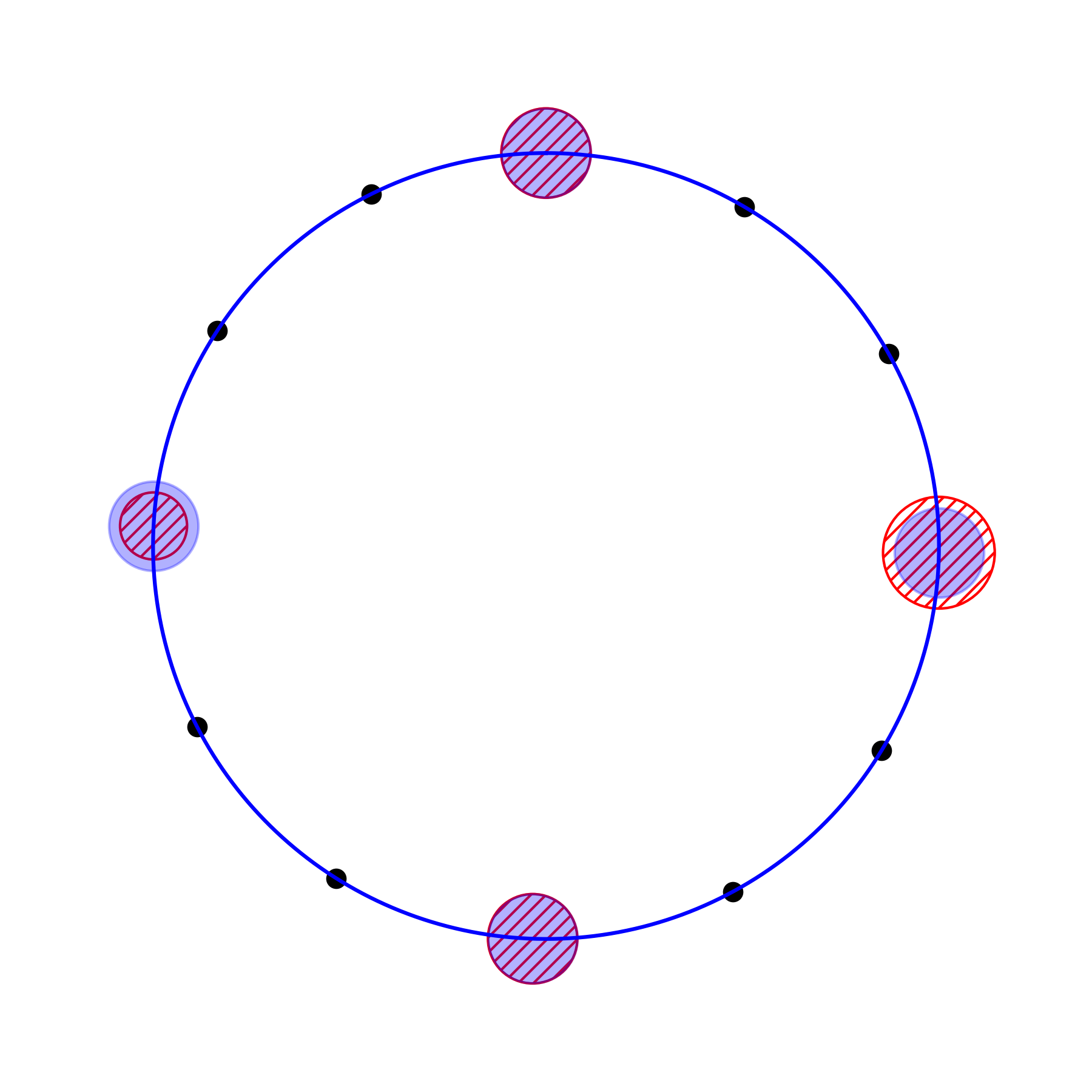}
    \end{minipage}
    \\ \hline
  \end{tabular}
  \captionof{figure}{An illustration of the walk on $\mathbb{Z}_{12}$ driven by $p\in\mathcal{M}(\mathbb{Z}_{12})$ with $p(-1)=p(2)=1/2$. The red (hashed) disks indicate the probabilities $p^{(n)}(x)$ and the blue disks indicate the values of the attractor $\Theta_p(n,x)/12$ for $x\in\mathbb{Z}_{12}$ and $1\leq n\leq 6$. The area of the disks are proportional to the represented values.}\label{fig:IntroExample}
\end{table}

\end{example}

\noindent As it has been pointed out many times in many places \cite{Diaconis1988, Greenhalgh1989,  Woess2000}, a random walk on $G$ which is not aperiodic (or irreducible) must cycle through cosets of a proper subgroup in $G$. Evidenced by the preceding example, we suspect that $\Theta_p$ is simply an indicator function of cosets and hence $\Theta_p$ captures the random walk's ``dance''. This is correct and, to make it precise, let's quickly treat a lemma.
\begin{lemma}\label{lem:equiv_gen}
Let $G$ be an abelian group and $S$ a non-empty subset of $G$. Then, for any $x,x'\in S$, the subgroups generated by $S-x$ and $S-x'$ coincide.
\end{lemma}
\begin{proof}
Let $H$ be a subgroup of $G$ containing $S-x$. Observe that $x-x'=-(x'-x)\in H$. Thus, for any $y=s-x'\in S-x'$, $y=(s-x)+(x-x')\in H$ so that $S-x'\subseteq H$. As this argument is symmetric, this guarantees that all subgroups of $G$ containing $S-x$ must contain $S-x'$ and vice versa. Hence, the subgroups generated by these sets coincide.
\end{proof}

\noindent Using the preceding lemma, we define $G_p$ to be the unique subgroup of $G$ generated by $\supp(p)-x_0$, i.e., 
\begin{equation}\label{eq:GpDef}
G_p:=\langle \supp(p)-x_0\rangle,
\end{equation}
for any (and every) $x_0\in \supp(p)$. Using some basic techniques of harmonic analysis, we prove in the next section that $\Omega(p)$ is a subgroup of $\widehat{G}$ (Proposition \ref{prop:SamePhase}) , $G_p$ is its annihilator subgroup (Proposition \ref{prop:ThetaIsIndicator}), and, for any $x_0\in\supp(p)$,
\begin{equation*}
\Theta_p(x,n)=\abs{\Omega(p)}\mathds{1}_{G_p}(x-nx_0)=\abs{\Omega(p)}\mathds{1}_{G_p+nx_0}(x)
\end{equation*}
for all $x\in G$ and $n\in \mathbb{N}$; here $\abs{\Omega(p)}$ is the number of elements in $\Omega(p)$. Making use of Proposition \ref{prop:IntroLLT}, we see that 
\begin{equation*}
    1=\limsup_n \sum_{x\in G}p^{(n)}(x)=\limsup_n \sum_{x\in G}\frac{\abs{\Omega(p)}\mathds{1}_{G_p}(x-nx_0)}{\abs{G}}=\frac{\abs{\Omega(p)} \abs{G_p}}{\abs{G}}
\end{equation*}
so that $\abs{\Omega(p)}=\abs{G}/\abs{G_p}$. With this, the proposition can be restated as follows.
\begin{proposition}\label{prop:IntroLLTinIndicator}
Let $G$ be a finite abelian group and $p$ a probability distribution on $G$. Let $G_p$ be as above. Then, there is a constant $0\leq\rho<1$ such that, for any $x_0\in \supp(p)$,
\begin{equation*}
    p^{(n)}(x)=\frac{\mathds{1}_{G_p}(x-nx_0)}{\abs{G_p}}+O(\rho^n)
\end{equation*}
uniformly in $x\in G$ as $n\to \infty$.
\end{proposition}

\begin{example}[Example \ref{ex:Intro} Revisited]
For $p\in\mathcal{M}(\mathbb{Z}_{12})$ given in Example \ref{ex:Intro}, we see that $\Omega(p)=\{0,4,8\}=4\mathbb{Z}_{12}$. Its annihilator subgroup, $\Omega(p)^\dagger$, is precisely $G_p=3\mathbb{Z}_{12}=\langle \{-1,2\}-x_0\}\rangle $ where $x_0=-1$. Of course, $3\vert(x+n)$ precisely when $x\in 3\mathbb{Z}_{12}-n=G_p+nx_0$.
\end{example}

\noindent This article proceeds as follows. In Section \ref{sec:W&D}, we develop the basic Fourier-analytic tools to introduce $\Theta_p$ in the context of a finitely-generated abelian group $G$. We show that $\Theta_p$ informs the support of $p^{(n)}$ and can be expressed as an indicator function of cosets. In the special case that our random walk is irreducible and periodic of period $s$, we use $\Theta_p$ to show these cosets make up the entire quotient group. All results in Section \ref{sec:W&D} pertain to the entire class of $p\in\mathcal{M}(G)$. In Section \ref{sec:Main}, we state our main result (Theorem \ref{thm:MainLLT}) and connect it to the classical local limit theorems on $\mathbb{Z}^d$. We also present several corollaries. Section \ref{sec:Proof} is dedicated the the proof of Theorem \ref{thm:MainLLT} where we highlight, in particular, the natural appearance of $\Theta_p$ as a Haar integral. In Section \ref{sec:Examples}, we present several examples; additional examples which making use of $\Theta_p$ in the context of $\mathbb{Z}^d$ can be found in Section 7 of \cite{RSC17}.

\section{Walking and Dancing}\label{sec:W&D}
In what follows, we shall take $G$ to be a finitely-generated abelian group (taken with the discrete topology) with operation $+$ and identity $0$. The Pontryagin dual of $G$ will be denoted by $\widehat{G}$ and we write its operation as $+$ and identity $0$; necessarily $\widehat{G}$ is compact\cite[Theorem 1.2.5]{Rudin1967}. Consistent with the introduction, for $\xi\in\widehat{G}$, we shall denote by $\chi_\xi$ the corresponding continuous homomorphism of $G$ into the unit circle $\mathbb{S}$, i.e., $\chi_\xi$ is a continuous complex-valued function with $\abs{\chi_\xi(x)}=1$ for all $x\in G$ and
\begin{equation*}
\chi_\xi(x+y)=\chi_\xi(x)\chi_\xi(y)
\end{equation*}
for all $x,y\in G$. Further, for any $\xi,\xi'\in\widehat{G}$, the sum $\xi+\xi'$ is characterized by
\begin{equation*}
\chi_{\xi+\xi'}(x)=\chi_{\xi}(x)\chi_{\xi'}(x)
\end{equation*}
for all $x\in G$. In this way, our convention sees the dual/character group $\widehat{G}$ as that which parameterizes the (necessarily) continuous homomorphisms from $G$ to $\mathbb{S}$, i.e., $\chi_{\cdot}:\widehat{G}\mapsto\operatorname{Hom}(G,\mathbb{S})$. In this notation, for each $p\in\mathcal{M}(G)$, we define its Fourier transform $\widehat{p}:\widehat{G}\to\mathbb{C}$ by
\begin{equation*}
\widehat{p}(\xi)=\sum_{x\in G}p(x)\chi_{\xi}(x)
\end{equation*}
for $\xi\in \widehat{G}$. As the Fourier transform converts convolution into multiplication, we obtain the following key identity after noting that every absolutely summable function on $G$ can be recovered from its Fourier transform.

\begin{proposition}\label{prop:FTConvIdentity}
Let $p\in\mathcal{M}(G)$. Then, for all $n\geq \mathbb{N}_+$ and $x\in G$, 
\begin{equation}\label{eq:FTConvIdentity}
p^{(n)}(x)=\int_{\widehat{G}}\widehat{p}(\xi)^n \chi_\xi(-x)\,d\mu(\xi)
\end{equation}
where $\mu$ denotes the Haar measure on $\widehat{G}$ with $\mu(\widehat{G})=1$.
\end{proposition}

\noindent For $p\in\mathcal{M}(G)$, we define
\begin{equation}\label{eq:OmegaDef}
\Omega(p)=\{\xi\in\widehat{G}: \widehat{p}(\xi)\in \mathbb{S}\}.
\end{equation}
Just as we saw in the introduction, this simple set characterizes the ``dance" of random walk precisely. The following proposition plays a key role in this.

\begin{proposition}\label{prop:SamePhase}
For $p\in\mathcal{M}(G)$,  $\Omega(p)$ is a closed subgroup of $\widehat{G}$ and therefore a compact abelian group. For each $\xi\in \Omega(p)$ we have
\begin{equation*}
\widehat{p}(\xi)=\chi_\xi(x)
\end{equation*}
for all $x\in \supp(p)$. In particular, if $0\in\supp(p)$, then $\widehat{p}(\xi)=1$ for all $\xi\in \Omega(p)$.
\end{proposition}
\begin{proof}
Because $\sum_x p(x)=1$, the condition that $\xi\in\Omega(p)$, i.e., 
\begin{equation*}
\abs{\sum_{x\in G}p(x)\chi_\xi(x)}=1,
\end{equation*} 
guarantees that every term in the sum
\begin{equation*}
\widehat{p}(\xi)=\sum_{x\in \supp(p)}p(x)\chi_\xi(x)
\end{equation*}
must have the same phase. This ensures at once that $\widehat{p}(\xi)=\chi_\xi(x)$ for all $x\in\supp(p)$.

We now show that $\Omega(p)$ is a closed subgroup of $\widehat{G}$. First, it is clear that $0\in\Omega(p)$. Observe that, if $\xi,\xi'\in\Omega(p)$,
\begin{eqnarray*}
\widehat{p}(\xi-\xi')&=&\sum_{x\in G}p(x)\chi_{\xi}(x)\chi_{-\xi'}(x)\\
&=&\sum_{x\in \supp(p)}p(x)\chi_{\xi}(x)\overline{\chi_{\xi'}(x)}\\
&=&\sum_{x\in\supp(p)}p(x)\widehat{p}(\xi)\widehat{p}(\xi')^{-1}\\
&=&\widehat{p}(\xi)\widehat{p}(\xi')^{-1}
\end{eqnarray*}
and so $\xi-\xi'\in\Omega(p)$. We remark that this argument also shows $\widehat{p}$, when restricted to $\Omega(p)$, is a group homomorphism from $\Omega(p)$ to $\mathbb{S}$. Finally, $\Omega(p)$ is closed because $\widehat{p}$ is continuous.
\end{proof}

\noindent As we will see, $\Omega(p)$ is a central object in this article and its structure is essential to our understanding of the support of random walk, i.e., the dance. Historically, it appears that much early work on random walk drew conclusions from sets closely related to $\Omega(p)$ but somehow miss a full appreciation of $\Omega(p)$ itself. For example, F. Spitzer works heavily with the set $\{\xi\in\widehat{G}:\widehat{p}(\xi)=1\}$, drawing conclusions concerning aperiodicity from it (c.f., Theorem 2.7.1 of \cite{Spitzer}). For a complex measure $\mu$ on an LCA group $G$, B. Schreiber recognized the importance of the corresponding set $E_\mu=\{\xi\in\widehat{G}:\abs{\widehat{\mu}(\xi)}=1\}$ \cite{Schreiber1970}. While, in the complex-valued setting, this set is not generally a subgroup of $\widehat{G}$, Schreiber shows it is key to understanding which measures have bounded convolution powers in the total variation norm, a property which has important consequences in the study of numerical solutions to partial differential equations.\\

\noindent Because $\Omega(p)$ is a compact abelian group thanks to Proposition \ref{prop:SamePhase}, it has a Haar measure $\omega_p$ which is uniquely specified by a yet-to-be-determined\footnote{In all cases, $\omega_p(\Omega(p))$ will be an integer. When $\Omega(p)$ is finite, $\omega_p$ is counting measure. See Remark \ref{rmk:Normalization}.} normalization $\omega_p(\Omega(p))>0$. In general, the nature of $\Omega(p)$ and $\omega_p$ are quite dissimilar from $\widehat{G}$ and its Haar measure $\mu$. Still, every character of $\Omega(p)$ extends to a character of $\widehat{G}$ and we shall denote these by $\chi_x(\xi)=\chi_{\xi}(x)$ by virtue of Pontryagin's theorem. With this,  for every $n\in\mathbb{N}$ and $x\in G$, we define
\begin{equation}\label{eq:ThetaDef}
\Theta_p(n,x)=\int_{\Omega(p)}\widehat{p}(\xi)^n\chi_\xi(-x)d\omega_p(\xi)
\end{equation}
and set
\begin{equation*}
\Theta_p(x)=\Theta_p(0,x)=\int_{\Omega(p)}\chi_{\xi}(-x)\,d\omega_p(\xi).
\end{equation*} 
As we saw the introduction, this function carries information about the random walk and will naturally appear in our arguments. Precisely, we will find that $\Theta_p$ describes the support of $p^{(n)}$ and hence the ``dance" of random walk; correspondingly, we refer to it as the \textit{dance function}. Toward this understanding is the following basic result.

\begin{proposition}\label{prop:ThetaCapturesSupport}
Let $p\in\mathcal{M}(G)$. For every $n\in\mathbb{N}_+$,
\begin{equation*}
\supp\left(p^{(n)}\right)\subseteq \supp(\Theta_p(n,\cdot)).
\end{equation*}
In other words, we have: If $\Theta_p(n,x)=0$, then the random walk driven by $p$ will not visit $x$ at step $n$.
\end{proposition}

\begin{proof}
We argue by induction. For $n=1$, let $x\in\supp(p)$ and observe that, by virtue of Proposition \ref{prop:SamePhase}, $\chi_{\xi}(x)=\widehat{p}(\xi)$ for all $\xi\in\Omega(p)$. Consequently,
\begin{equation*}
\Theta_p(1,x)=\int_{\Omega(p)}\widehat{p}(\xi)\chi_{\xi}(-x)\,d\omega_p(\xi)=\int_{\Omega(p)}\chi_\xi(x)\chi_{\xi}(-x)\,d\omega_p(\xi)=\int_{\Omega(p)}1\,d\omega_p(\xi)=\omega_p(\Omega(p))>0
\end{equation*}
where we have used the homomorphism property of characters. Thus $\supp(p)\subseteq\supp(\Theta(1,\cdot))$ and we have verified the base case. 

Assume now that the assertion is true for $n\geq 1$ and suppose that $p^{(n+1)}(x)>0$ for $x\in G$. In this case, 
\begin{equation*}
\sum_{y\in G}p^{(n)}(x-y)p(y)>0
\end{equation*}
and therefore $x-y\in\supp(p^{(n)})$ for some $y\in\supp(p)$. For this $y$, our induction hypothesis gives us $\Theta_p(n,x-y)>0$ and so
\begin{equation*}
\Theta_p(n+1,x)=\int_{\Omega(p)}\widehat{p}(\xi)^{n}\widehat{p}(\xi)\chi_{\xi}(-x)\,d\omega_p(\xi)=\int_{\Omega(p)}\widehat{p}(\xi)^n\chi_{\xi}(y-x)d\omega_p(\xi)=\Theta_p(n,x-y)>0
\end{equation*}
where we have used the fact that $\widehat{p}(\xi)=\chi_{\xi}(y)$ for all $\xi\in\Omega(p)$ thanks to Proposition \ref{prop:SamePhase}. Hence $\supp(p^{(n+1)})\subseteq\supp(\Theta_p(n+1,\cdot))$ and our proof is complete.
\end{proof}

\noindent As we discussed in the introduction, we can also recognize $\Theta_p$ as an indicator function.
\begin{proposition}\label{prop:ThetaIsIndicator}
    Let $p\in \mathcal{M}(G)$ and let $G_p$ be the subgroup as given in \eqref{eq:GpDef}. Then, $G_p$ is the annihilator subgroup of $\Omega(p)$, i.e., 
    \begin{equation*}
    G_p=\Omega(p)^{\dagger}:=\{x\in G:\chi_{\xi}(x)=1\quad\mbox{for all}\quad \xi\in\Omega(p)\}
    \end{equation*}
and, for any $x_0\in\supp(p)$,
    \begin{equation*}
        \Theta_p(n,x)=\Theta_p(x-nx_0)=\omega_p(\Omega(p))\mathds{1}_{G_p}(x-nx_0)
    \end{equation*}
    for all $n\in\mathbb{N}_+$ and $x\in G$.
\end{proposition}

\begin{proof}
Throughout the proof, we shall suppress our notation's dependence on $p$ and write $\Omega=\Omega(p)$, $\omega=\omega_p$, and $\Theta=\Theta_p$. If $y=x-x_0\in\supp(p)-x_0$ for $x_0\in\supp(p)$, then 
\begin{equation*}
    \chi_{\xi}(y)=\chi_{\xi}(x-x_0)=\chi_{\xi}(x)\overline{\chi_{\xi}(x_0)}=\widehat{p}(\xi)\overline{\widehat{p}(\xi)}=1
\end{equation*}
whenever $\xi\in\Omega$ by virtue of Proposition \ref{prop:SamePhase}. Thus, $\Omega^{\dagger}$ is a subgroup containing $\supp(p)-x_0$ and so $G_p\subseteq \Omega^{\dagger}.$ Observe that, if $\xi\in G_p^{\dagger}$ (i.e., $\chi_{\xi}(x)=1$ for all $x\in G_p$), then
\begin{eqnarray*}
    \widehat{p}(\xi) &=& \sum_{x \in \supp(p)}p(x)\chi_{\xi}(x) 
    = \chi_{\xi}(x_0)\sum_{x \in \supp(p)} p(x)\chi_{\xi}(x-x_0) \in \mathbb{S}
\end{eqnarray*}
since $\supp(p)-x_0\in G_p$. In other words, we have $G_p^{\dagger}\subseteq\Omega$.  Using the definition of $\dagger$ and making use of \cite[Proposition 4.38]{Folland1995}, it follows that
\begin{equation*}
\Omega^{\dagger}\subseteq\left(G_p^{\dagger}\right)^{\dagger}=G_p
\end{equation*}
since $G_p$ is necessarily closed (as $G$ carries the discrete topology). Thus, $\Omega^{\dagger}=G_p$.

Given $x_0\in\supp(p)$, the identity
\begin{equation*}
    \Theta(n,x)=\Theta(x-nx_0)
\end{equation*}
follows by virtually the same argument used in the previous proposition (making use of Proposition \ref{prop:SamePhase} and the homomorphism property of characters).  Thus, to complete the proof, it suffices to prove that
\begin{equation}\label{eq:IndicateTheta}
\Theta(x)=\omega(\Omega)\mathds{1}_{\Omega^{\dagger}}(x).
\end{equation}
If $x\in \Omega^{\dagger}$, $\chi_{\xi}(-x)=\overline{\chi_{\xi}(x)}=1$ for all $\xi\in \Omega$ and so
\begin{equation*}
\Theta(x)=\int_{\Omega}1\,d\omega=\omega(\Omega)=\omega(\Omega)\mathds{1}_{\Omega^{\dagger}}(x).
\end{equation*}
Observe that, for any $x\in G$ and $\xi_0\in \Omega$, the invariance of the Haar measure $\omega$ and the homomorphism property of characters guarantee that
\begin{equation*}
\Theta(x)=\int_{\Omega}\chi_{\xi_0+\xi}(-x)\,d\omega(\xi)=\chi_{\xi_0}(-x)\int_{\Omega}\chi_{\xi}(-x)\,d\omega(\xi)
=\chi_{\xi_0}(-x)\Theta(x).
\end{equation*}
In other words, for any $x\in G$ and $\xi_0\in\Omega$, we have
\begin{equation*}
\Theta(x)\left(1-\overline{\chi_{\xi_0}(x)}\right)=0.
\end{equation*}
Now, in the case that $x\notin\Omega^{\dagger}$, there must be some $\xi_0\in\Omega$ with $\chi_{\xi_0}(x)\neq 1$. With this, the above identity ensures that $\Theta(x)=0$ whenever $x\notin\Omega^{\dagger}$ and hence \eqref{eq:IndicateTheta} is true for all $x\in G$.

\end{proof}

\noindent As we discussed in the introductory section, it is customary in random walk theory to study random walks that are (or are modified to be) both irreducible and aperiodic. It is under these hypotheses that local limit theorems are often established, absent of a prefactor $\Theta_p$. The following proposition shows why this is the case.

\begin{proposition}\label{prop:AperAndIrreduc}
Let $p\in\mathcal{M}(G)$ and assume that $p$ drives a random walk on $G$ that is both aperiodic and irreducible. Then $\Omega(p)=\{0\}$, $\Theta_p\equiv 1$, and $G_p=G$.
\end{proposition}
\noindent While this proposition is a corollary of Proposition \ref{prop:IrreducAndPerS} below (for $s=1$),  we have decided to give a distinct proof here that completely determines $\Omega(p)$ under the given hypotheses.
\begin{proof}
In view of the preceding proposition, we need only to show that $\Omega(p)=\{0\}$. To this end, let $\xi\in\Omega(p)$. Using the identity $(\widehat{p})^n=\widehat{p^{(n)}}$ for $n\in\mathbb{N}_+$, we see easily that $\xi\in\Omega(p^{(n)})$ for all $n\in\mathbb{N}_+$.  Using our hypotheses that the random walk driven by $p$ is both aperiodic and irreducible, for each $x\in G$, it is easy to find $n=n_x\in\mathbb{N}_+$ for which $p^{(n)}(0)$ and $p^{(n)}(x)$ are both positive, i.e., $\{0,x\}\subseteq\supp(p^{(n)})$. By virtue of Proposition \ref{prop:SamePhase} (applied to $p^{(n)}$), we have
\begin{equation*}
\widehat{p^{(n)}}(\xi)=\chi_\xi(x)=\chi_\xi(0)=1
\end{equation*}
since $\xi\in\Omega(p^{(n)})$. As $x$ was arbitrary, we can conclude that $\chi_\xi(x)=1$ for all $x\in G$ and so, because characters separate points, it can only be the case that $\xi=0$.
\end{proof}

\noindent In the slightly more general case that $p\in\mathcal{M}(G)$ drives an irreducible walk on $G$ with period $s\in\mathbb{N}_+$, we can make use of Proposition \ref{prop:ThetaCapturesSupport} and \ref{prop:ThetaIsIndicator} to obtain useful information about $G_p$ and $G/G_p$. To this end, let's first treat a corollary of Propositions \ref{prop:ThetaCapturesSupport} and \ref{prop:ThetaIsIndicator}.

\begin{lemma}\label{lem:PeriodicReturnsToSupport}
Let $p\in\mathcal{M}(G)$ and let $G_p$ be as given in \eqref{eq:GpDef}. If $p$ drives an irreducible walk on $G$ of period $s$, then, for any $x_0\in\supp(p)$, $sx_0\in G_p$.
\end{lemma}

\begin{proof}
We take $s$ to be the period of $p$, i.e., $s=\gcd R$ where
\begin{equation*}
R=\{n\in\mathbb{N}_+:p^{(n)}(0)>0\}
\end{equation*}
is the set of return times to zero, and fix $x_0\in\supp(p)$. It follows that there is a (finite) collection $n_1,n_2,\dots,n_K\in R$ with $s=\gcd(n_1,n_2,\dots,n_K)$. Appealing to Propositions \ref{prop:ThetaCapturesSupport} and \ref{prop:ThetaIsIndicator}, we find
\begin{equation*}
0<\Theta(n_k,0)=\omega_p(\Omega(p))\mathds{1}_{G_p}(0-n_kx_0)
\end{equation*}
so that $\pm n_kx_0\in G_p$ for every $k=1,2,\dots,K$ (since $G_p$ is a group). Thanks to Bezout's identity, there are integers $a_1,a_2,\dots,a_K$ for which $s=a_1n_1+a_2n_2+\cdots+a_Kn_K$ and consequently
\begin{equation*}
sx_0=a_1n_1x_0+a_2n_2x_0+\cdots a_Kn_Kx_0\in G_p,
\end{equation*}
as desired.
\end{proof}

\begin{proposition}\label{prop:IrreducAndPerS}
Let $p\in\mathcal{M}(G)$, define $G_p$ as above, and take $x_0\in\supp(p)$. If $p$ drives a random walk on $G$ that is irreducible and periodic of period $s$, then the cosets
\begin{equation*}
G_p, G_p+x_0,\,G_p+2x_0,\dots,\mbox{and }\,\,G_p+(s-1)x_0
\end{equation*}
form a partition of $G$. In particular, 
\begin{equation*}
G/G_p=\{G_p,G_p+x_0,G_p+2x_0,\dots,G_p+(s-1)x_0\}
\end{equation*}
and $[G:G_p]=s$.
\end{proposition}

\begin{proof}
Let $p\in\mathcal{M}(G)$ be irreducible and periodic of period $s$. For each $x\in G$, irreducibility gives $n\in\mathbb{N}$ for which $x\in \supp(p^{(n)})$ and, by virtue of Propositions \ref{prop:ThetaCapturesSupport} and \ref{prop:ThetaIsIndicator}, we have
\begin{equation*}
0<\Theta(n,x)=\omega_p(\Omega(p))\mathds{1}_{G_p}(x-nx_0)=\omega_p(\Omega(p))\mathds{1}_{G_p+nx_0}(x)
\end{equation*}
so that $x\in G_p+nx_0$. Hence
\begin{equation*}
G=\bigcup_{n=1}^\infty (G_p+nx_0).
\end{equation*}
Making an appeal to Lemma \ref{lem:PeriodicReturnsToSupport}, $sx_0\in G_p$ and so
\begin{equation*}
G=\bigcup_{k=1}^s (G_p+kx_0)=\bigcup_{k=0}^{s-1}(G_p+kx_0).
\end{equation*}
To complete the proof of the proposition, it remains to show that the cosets $G_p,\,G_p+x_0,\dots,G_p+(s-1)x_0$ are all distinct. To do this, we shall make use of the following identity which says that these cosets are precisely the periodic classes of the random walk: For $k\in\mathbb{N}$,
\begin{equation}\label{eq:CosetIden}
G_p+kx_0=\{x\in G:p^{(ns+k)}(y)>0\,\,\mbox{ for some } n\in\mathbb{N}\}.
\end{equation}
Using basic techniques of Markov/random walks theory, the verification of this identity is straightforward \cite{Chung1960,Woess2000}. For completeness, we have given a proof in Section \ref{sec:PeriodicClasses} of the appendix. 

Armed with \eqref{eq:CosetIden}, suppose that $y\in (G_p+kx_0)\cap (G_p+jx_0)$ for $1\leq j\leq k\leq s$ and select natural numbers $n$ and $m$ for which $
p^{(ns+j)}(y)>0$  and $p^{(ms+k)}(y)>0.$ Since $p$ is irreducible, let's also choose $l$ for which $p^{(l)}(-y)>0$. With this, we see that
\begin{equation*}
p^{(ns+j+l)}(0)\geq p^{(l)}(-y)p^{(ns+j)}(y)>0\hspace{1cm}\mbox{and}\hspace{1cm}
p^{(ms+k+l)}(0)\geq p^{(l)}(-y)p^{(ms+k)}(y)>0.
\end{equation*}
By the $s$-periodicity of $p$, it follows that $s$ divides both $j+l$ and $k+l$ and so it must divide $k-j=k+l-(j+l)$. Of course, $0\leq k-j<s$ and so $k=j$. Thus, the cosets are distinct and our proof is complete.
\end{proof}

\noindent For irreducible walks with period $s$, Proposition \ref{prop:IrreducAndPerS} can be used to inform on the structure of $\Omega(p)$ and $\Theta_p$. This information is captured by the following corollary, a results we will soon use to obtain asymptotics for time averages of $p^{(n)}$ (Theorem \ref{thm:TimeAverageLLT}).

\begin{corollary}\label{cor:ThetaAverage}
Let $p\in\mathcal{M}(G)$. If $p$ drives an irreducible and periodic walk on $G$ of period $s$, then $\Omega(p)$ is a finite abelian group of order $s$. Further, if $\omega_p$ is taken to be counting measure, then
\begin{equation*}
\frac{1}{s}\sum_{k=0}^{s-1} \Theta_p(n+k,x)=1
\end{equation*}
for every $n\in\mathbb{N}_+$ and $x\in G$.
\end{corollary}
\begin{proof}
By virtue of Proposition \ref{prop:IrreducAndPerS}, we see that $G/G_p$ is a finite abelian group of order $s$ and, since such groups are self-dual, $\widehat{G/G_p}$ is of order $s$. Recalling that $G_p^\dagger=\Omega(p)$ as was shown in the proof of Proposition \ref{prop:ThetaIsIndicator}, an appeal to Theorem 4.39 of \cite{Folland1995} guarantees that $\widehat{G/G_p}$ and $G_p^\dagger=\Omega(p)$ are isomorphic. Consequently, $\Omega(p)$ is a finite abelian group of order $s$. 

With $\omega_p$ chosen as counting measure, $\omega_p(\Omega(p))=\abs{\Omega(p)}=s$ and so and appeal to Proposition \ref{prop:ThetaIsIndicator} gives
\begin{equation*}
\Theta_p(m,x)=s\cdot\mathds{1}_{G_p+mx_0}(x)
\end{equation*}
for all $x\in G$ and $m\in\mathbb{N}_+$ where $x_0\in\supp(p)$. By virtue of Proposition \ref{prop:IrreducAndPerS},
 for any $n\in\mathbb{N}_+$, $\{G_p+(n+k)x_0:k=0,1,\dots,s-1\}$ forms a partition of $G$. Thus, for every $x\in G$,
\begin{equation*}
\frac{1}{s}\sum_{k=0}^{s-1}\Theta_p(n+k,x)=\sum_{k=0}^{s-1} \mathds{1}_{G_p+(n+k)}(x)=\mathds{1}_G(x)=1.
\end{equation*}
\end{proof}

\noindent We end this section with a proposition that describes how $\Omega(p)$ and $\Theta_p$ are transformed under isomorphism. 

\begin{proposition}\label{prop:ThetaUnderIso}
Let $G$ and $H$ be isomorphic finitely-generated abelian groups and $T:G\to H$ an isomorphism. Denote by $\widehat{T}:\widehat{H}\to\widehat{G}$ the natural isomorphism between their Pontryagin duals, i.e., this is the isomorphism for which
\begin{equation*}
\chi_{\widehat{T}(\zeta)}(x)=\chi_{\zeta}(T(x))
\end{equation*}
for all $x\in G$ and $\zeta\in \widehat{H}$. Finally, let $p\in \mathcal{M}(G)$ and consider its pushforward $q=T_*(p)\in\mathcal{M}(H)$. Then $\widehat{T}$ restricts to an isomorphism between $\Omega(q)$ and $\Omega(p)$. Further, if the measures $\omega_p$ and $\omega_{q}$ are compatibly normalized so that $\omega_p(\Omega(p))=\omega_{q}(\Omega(q))$, then
\begin{equation*}
    \Theta_{p}(n,x)=\Theta_{q}(n,T(x))
\end{equation*}
for all $n\in\mathbb{N}$ and $x\in G$.
\end{proposition}
\begin{proof}
    Observe that
    \begin{equation*}
        \widehat{q}(\eta)=\sum_{y\in H}q(y)\chi_{\eta}(y)
        =\sum_{x\in G}q(T(x))\chi_{\eta}(T(x))
        =\sum_{x\in G}p(x)\chi_{\widehat{T}(\eta)}(x)=\widehat{p}(\widehat{T}(\eta))
    \end{equation*}
    for all $\eta\in\widehat{H}$. Since $\widehat{T}$ is a bijection, it follows that $\widehat{T}$ maps the subgroup $\Omega(q)$ of $\widehat{H}$ bijectively onto $\Omega(p)$ and so we conclude that $\widehat{T}\vert_{\Omega(q)}:\Omega(q)\to\Omega(p)$ is an isomorphism which, by an abuse of notation, we henceforth denote by $\widehat{T}$. This situation is illustrated by Figure \ref{fig:ConnDiag}.
    \begin{figure}[h!]
\begin{center}
\includegraphics[width=10cm]{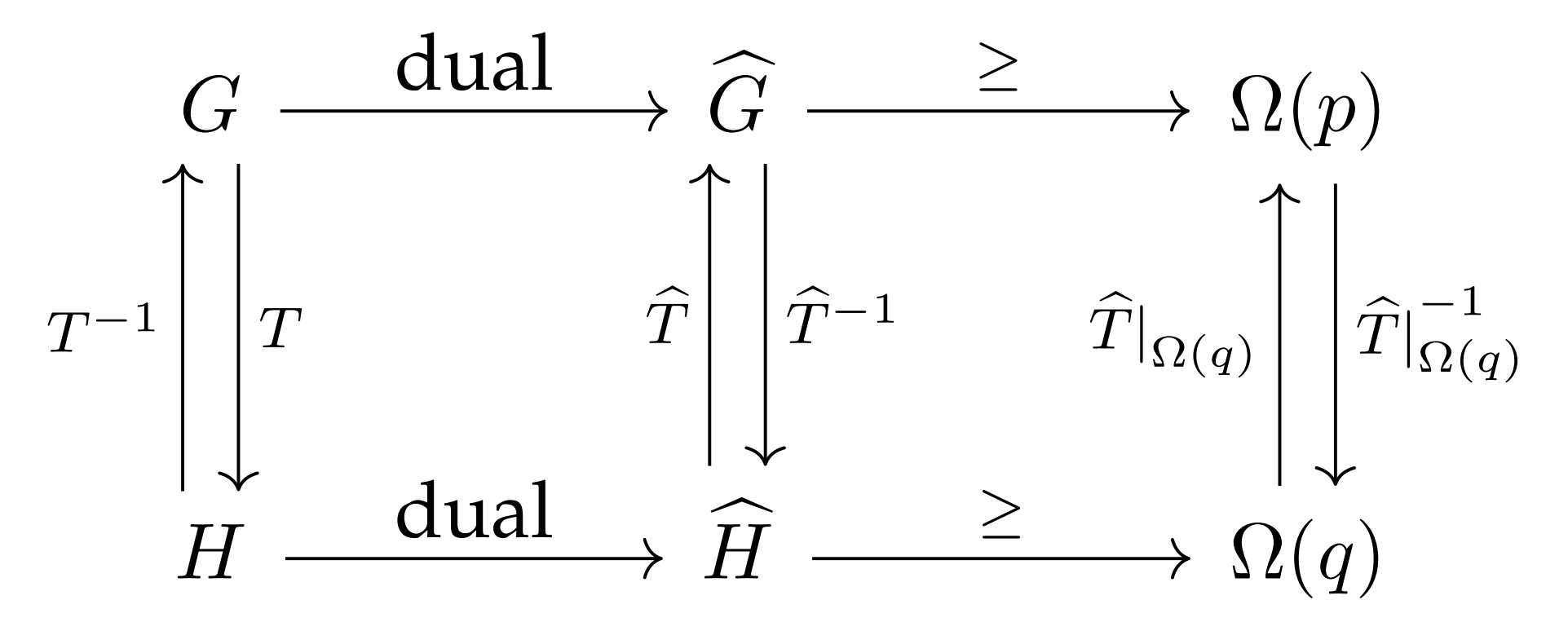}
\caption{The isomorphism between $\Omega(q)$ and $\Omega(p)$}\label{fig:ConnDiag}
\end{center}
\end{figure}

\noindent Let $\omega_{p}$ and $\omega_{q}$ be compatibly normalized Haar measures on $\Omega(p)$ and $\Omega(q)$, respectively. Since $\widehat{T}$ is an isomorphism between $\Omega(q)$ and $\Omega(p)$, the pushforward measure $\widehat{T}_*(\omega_{q})$ is necessarily a Haar measure on $\Omega(p)$ and hence, by the uniqueness of Haar measure, there must be some constant $C>0$ for which
\begin{equation*}
\widehat{T}_*(\omega_{q})=C\cdot\omega_{p}.
\end{equation*}
By our hypothesis that $\omega_p$ and $\omega_{q}$ are compatibly normalized, we see that
\begin{equation*}
    \omega_p(\Omega(p))=\omega_{q}(\Omega(q))=\omega_{q}(\widehat{T}^{-1}(\Omega(p))=\widehat{T}_*(\omega_{q})(\Omega(p))=C\omega_p(\Omega(p))
\end{equation*}
and, from this we conclude that $C=1$ and hence
\begin{equation*}
    \widehat{T}_*(\omega_{q})=\omega_p.
\end{equation*}
We can now complete the proof. Appealing to proposition \ref{prop:ThetaIsIndicator} (upon taking $x_0\in\supp(p)$ so $T(x_0)\in\supp(q)$), we obtain
    \begin{eqnarray*}
        \Theta_{q}(n,T(x))&=&\Theta_{q}(0,T(x-nx_0))\\
&=&\int_{\Omega(q)}\chi_{\eta}(T(x-nx_0))\,d\omega_{q}(\eta)\\
&=&\int_{\widehat{T}^{-1}(\Omega(p))}\chi_{\widehat{T}(\eta)}(x-nx_0)\,d\omega_{q}(\eta)\\
&=&\int_{\Omega(p)}\chi_\xi(x-nx_0)\,d(\widehat{T}_*\omega_{q})(\xi)\\
&=&\int_{\Omega(p)}\chi_{\xi}(x-n x_0)d\omega_p(\xi)\\
&=&\Theta_p(n,x)
    \end{eqnarray*}
    for all $n\in\mathbb{N}$ and $x\in G$.
\end{proof}

\section{Limit theorems on finitely-generated abelian groups}\label{sec:Main}

In this section, we present our main results.  Our central result is a generalized local limit theorem which naturally extends the classical local limit theorem from $\mathbb{Z}^k$ appearing, for instance in \cite[Chapter 2]{Spitzer}, \cite[Chapter 3]{Woess2000}, and \cite[Chapter 2]{LawlerLimic2010} while completely removing any and all assumptions concerning irreducibility/aperiodicity. As a natural extension of Proposition \ref{prop:IntroLLT} to the finitely-generated setting, the dance function $\Theta_p$ takes center stage.

\begin{theorem}\label{thm:MainLLT}
Let $p\in\mathcal{M}_2(G)$ where $G$ is a finitely-generated abelian group and denote by $\Tor(G)$ the torsion subgroup of $G$. Let $G_p$ be the subgroup of $G$ defined by \eqref{eq:GpDef}, set $d=\rank(G_p)$, and let $\Theta_p$ be given by $\eqref{eq:ThetaDef}$ where the Haar measure $\omega_p$ is normalized in a way described below (see Remark \ref{rmk:Normalization}). We have exactly two cases:
\begin{enumerate}
\item If $d=0$, then there is $0\leq \rho<1$ for which
\begin{equation}\label{eq:MainLLTFinite}
p^{(n)}(x)=\frac{\Theta_p(x,n)}{\abs{\Tor(G)}}+O(\rho^n)
\end{equation}
uniformly for $x\in G$ as $n\to\infty$.
\item If $d\geq 1$, then there is a surjective homomorphism $\varphi:G\to \mathbb{Z}^d$ so that the pushforward $\varphi_*(p)\in\mathcal{M}(\mathbb{Z}^d)$ drives a genuinely $d$-dimensional random walk on $\mathbb{Z}^d$ with mean $\mu=\E[\varphi_*(p)]\in\mathbb{R}^d$ and positive-definite $d\times d$ covariance matrix $\Gamma=\Cov[\varphi_*(p)]$. In these terms, we have
\begin{equation}\label{eq:MainLLTInfinite}
p^{(n)}(x)=\frac{\Theta_p(n,x)}{\abs{\Tor(G)}}K_{\varphi_*(p)}^n(\varphi(x)-n\mu)+o(n^{-d/2})
\end{equation}
uniformly for $x\in G$ as $n\to\infty$ where $K_{\varphi_*(p)}$ is the Gaussian density/heat kernel given by
\begin{equation*}
K^t_{\varphi_*(p)}(y)=\frac{1}{(2\pi t)^{d/2}\sqrt{\det(\Gamma)}}\exp\left(-\frac{y\cdot \Gamma y}{2t}\right)
\end{equation*}
for $t>0$ and $y\in \mathbb{R}^d$.
\end{enumerate}
\end{theorem}

\begin{remark}
Its not hard to see that $d\geq 1$ exactly when $G_p$ is infinite and this happens exactly when $\supp(p)-x_0$ contains an element of infinite order. Thus the theorem's rank condition can be easily rephrased in terms of $G_p$'s cardinality and, equivalently, the existence of infinite-order elements of $\supp(p)-x_0$. 
\end{remark}
\begin{remark}\label{rmk:Normalization}
In view of of Proposition \ref{prop:ThetaIsIndicator}, the normalization of Haar measure $\omega_p$ is gotten by specifying the constant $\omega_p(\Omega(p))=\Theta_p(0,0)$. In the course of proving Theorem \ref{thm:MainLLT}, we will determine this constant precisely and find that, in particular, it is always an integer (see \eqref{eq:MainLLTNormalization}). In the special case that $\Omega(p)$ is finite, $\omega_p$ is necessarily counting measure; this is Proposition \ref{prop:CountingMeasure}. As we discussed in Section \ref{sec:W&D}, $\omega_p(\Omega(p))=\abs{\Omega(p)}=s$ whenever $p$ is irreducible and periodic of period $s$.
\end{remark}

\begin{remark}\label{rmk:LLT_with_Indicators}
As in the introduction, the local limits above can be restated in terms of indicator functions thanks to Proposition \ref{prop:ThetaIsIndicator}. In particular, in the case that $G_p$ is finite (i.e., $d=0$), \eqref{eq:MainLLTFinite} is exactly as it appears in Proposition \ref{prop:IntroLLTinIndicator}. In the case that $d\geq 1$, we obtain
\begin{equation*}
p^{(n)}(x)=\frac{\omega_p(\Omega(p))}{\abs{\Tor(G)}}\mathds{1}_{G_p}(x-nx_0)K_{\varphi_*(p)}^n(\varphi(x)-n\mu)+o(n^{-d/2})
\end{equation*}
uniformly for $x\in G$ as $n\to\infty$.
\end{remark}

\begin{remark}\label{rmk:LLTIndicatorError}
In view of Propositions \ref{prop:ThetaCapturesSupport} and \ref{prop:ThetaIsIndicator}, both $p^{(n)}$ and $\Theta_p$ vanish outside of $G_p+nx_0$. Consequently, the errors in Theorem \ref{thm:MainLLT} can be replaced by $O(\rho^n\Theta_p(n,x))=O(\rho^n\mathds{1}_{G_p}(x-nx_0))$ in \eqref{eq:MainLLT1} and $o(n^{-d/2}\Theta_p(n,x))=o(n^{-d/2}\mathds{1}_{G_p}(x-nx_0))$ in \eqref{eq:MainLLT2}.
\end{remark}

\noindent Before moving onto a discussion of previous results (and placing Theorem \ref{thm:MainLLT} in context), we state two corollaries; the first concerns the ``support'' of random walk and the second concerns irreducibility and aperiodicity. In Proposition \ref{prop:ThetaCapturesSupport}, we saw that $\Theta_p$ gives us information about where a random walk cannot visit. The following related corollary of Theorem \ref{thm:MainLLT} tells us that, under certain mild conditions, we can also use $\Theta_p$ to determine where the random walk can visit.

\begin{corollary}\label{cor:RWSupport}
Given $p\in\mathcal{M}_2(G)$, let $G_p$ be as defined by \eqref{eq:GpDef} and set $d=\rank(G_p)$.
\begin{enumerate}
\item If $d=0$, then 
\begin{equation*}
\supp(p^{(n)})=\supp\left(\Theta_p(n,\cdot)\right)
\end{equation*}
for sufficiently large $n$.
\item  If $d\geq 1$ and $p$ is symmetric or, more generally, $\mu=\E[\varphi_*(p)]=0$ where $\varphi$ is the epimorphism appearing in the statement of Theorem \ref{thm:MainLLT}, then the following holds: If, for $x\in G$, there is a sequence $\{n_k\}$ for which $\Theta_p(n_k,x)>0$ for all $k$, then $p^{(n_k)}(x)>0$ for all sufficiently large $k$. In particular, if
\begin{equation*}
\limsup_n \Theta_p(n,x)>0,\hspace{1cm}\mbox{then}\hspace{1cm}\limsup_n n^{d/2}p^{(n)}(x)>0.
\end{equation*}
\end{enumerate} 
\end{corollary}
\noindent As the corollary's proof follows immediately from Theorem \ref{thm:MainLLT}, we leave the details to the interested reader. We remark that, the corollary can be extended to handle the case in which $\mu$ is non-zero and, in that case, it is possible to say that  the probability, for sufficiently large $n$, is non-zero for those $x$'s for which $\varphi(x)$ is near $n\mu$. A related statement can be found as Theorem 7.7 in \cite{RSC17}.

\begin{corollary}\label{cor:AperAndIrreducPartialConverse}
Given $p\in\mathcal{M}_2(G)$, assume that $\Omega(p)=\{0\}$ (or, equivalently, $G_p=G$). We have
\begin{enumerate}
\item If $G=G_p$ is finite, then $p$ drives an irreducible and aperiodic walk on $G$.
\item If $G=G_p$ is infinite and $p$ is symmetric (or more generally $\mu=\E[\varphi_*(p)]=0$ where $\varphi$ is the epimorphism appearing in the statement of Theorem \ref{thm:MainLLT}), then $p$ drives an irreducible and aperiodic walk on $G$. 
\end{enumerate}
\end{corollary}

\noindent The above corollary is a partial converse to Proposition \ref{prop:AperAndIrreduc} and its proof follows directly from the statement of Corollary \ref{cor:RWSupport}. It is worth noting that the random walk on $\mathbb{Z}$ driven by $p(1)=p(2)=1/2$ has $\Omega(p)=\{0\}$ and $G_p=\mathbb{Z}$, yet it never visits the negative integers and is therefore not irreducible. Thus, when $G$ is infinite, it is not enough to have $\Omega(p)=\{0\}$ to conclude a walk is irreducible (hence the corollary's mean-zero assumption). \\

\noindent As we discussed in the introduction, the vast majority of known local limit theorems on finitely-generated abelian groups take as basic assumptions that a given probability distribution $p$ is irreducible and aperiodic.  These results, save for varying forms of error (\textit{\'{a} la} Gaussian or Berry-Esseen), are captured by Theorem \ref{thm:MainLLT} upon noting that $\Theta_p\equiv 1$ when $p$ is irreducible and aperiodic thanks to Proposition \ref{prop:AperAndIrreduc}. To discuss other known local limit theorems in the context of $\mathbb{Z}^k$, it is helpful to first state a definition motivated by F. Spitzer.

\begin{definition}[Dimension of Sets and Random Walks]\label{def:dimension}
In what follows $k$ is a positive natural number.
\begin{enumerate}
\item A non-empty subset $S$ of $\mathbb{R}^k$ is said to be $d$ dimensional (with $d\in\mathbb{N}$, $0\leq d\leq k$) if it can be contained in some $d$-dimensional affine subspace of $\mathbb{R}^k$ but no proper subspace thereof. 
\item Given $p\in\mathcal{M}(\mathbb{Z}^k)$, we say that the random walk driven by $p$ is $d$ dimensional, if $\supp(p)\subseteq \mathbb{Z}^k\subseteq\mathbb{R}^k$ is a $d$ dimensional set. In the case that this dimension matches that of the ambient space, i.e., $d=k$, we say that it is genuinely $k$ (or $d$) dimensional.
\end{enumerate}
\end{definition}

\noindent Of the many known local limit theorems on $G=\mathbb{Z}^k$, a recent such result appears as Theorem 7.6 of \cite{RSC17}. The result there gives a local limit theorem for any probability distribution $p$ on $\mathbb{Z}^k$ which is genuinely $k$ dimensional and has finite second moments. Notably, Theorem 7.6 contains the function $\Theta_p$ where, under the hypotheses of genuine $k$ dimensionality, $\Omega(p)$ is necessarily finite and correspondingly $\omega_p$ is simply counting measure. To the authors' knowledge, \cite{RSC17} represents the first instance in which $\Theta_p$ (or its equivalent) appears explicitly in the literature. Our theorem above naturally extends the results of \cite{RSC17} to allow for distributions which are not genuinely $k$ dimensional and, in the case of a genuinely $k$-dimensional $p\in\mathcal{M}(\mathbb{Z}^k)$, it recaptures Theorem 7.6 of \cite{RSC17} giving $d=k$ and $\varphi$ to be the identity automorphism on $\mathbb{Z}^k=\mathbb{Z}^{d}$. \\

\noindent As noted in \cite{RSC17}, Theorem 7.6 therein extends the local limit theorems of the classic reference \cite{Spitzer} of F. Spitzer. These theorems, presented in \cite{Spitzer} as Theorem 7.9 and its ``strong form'' $(2)$ on Page 77, assume that a given walk is ``strongly aperiodic'', a term\footnote{It should be noted that Spitzer's nomenclature contrasts that of Markov chain theory. In his language, strongly aperiodic implies aperiodic which implies genuinely $d$ dimensional. Spitzer's aperiodic walks all have the property that, for some natural number $s$, $p^{(n)}(0)>0$ exactly when $n\vert s$; in other words, his aperiodic walks are periodic of period $s$. While Spitzer's non-standard definitions are somewhat complicated, it appears he made them to build in the condition of irreducibility, saying a walk otherwise was simply ``posed on the wrong group". Still, our theory in this article gives evidence that the various notions considered in \cite{Spitzer} were not completely necessary (some being overly complicated) and we suspect that this, at least partially, is why the group structure of $\Omega(p)$ was not noted nor exploited in \cite{Spitzer}.} that means the walk is both irreducible and aperiodic. Thus, Spitzer's hypotheses give $\Theta_p\equiv 1$ according to Proposition \ref{prop:AperAndIrreduc}, a conclusion that also follows from \cite[Proposition 7.8]{Spitzer}. The more recent text \cite{LawlerLimic2010} of G. Lawler and V. Limic also extends the local limit theorems of \cite{Spitzer} in certain directions. The focus in \cite{LawlerLimic2010} is on finite-range symmetric random walks on $\mathbb{Z}^k$ and, in particular, the authors establish local limit theorems for such walks giving various degrees of Gaussian-type error depending on the finiteness of higher-order moments; their error is stronger than the ``little-o" error present in this article. Related to this are the recent works \cite{Coe25} and \cite{CF24} which establish local limit theorems with Gaussian-type error for the convolution powers of complex-valued functions on $\mathbb{Z}$.   In contrast to the local limit theorems of \cite{Spitzer}, \cite{LawlerLimic2010} allows for certain periodic walks, however, their prevalent assumption that the walk be symmetric ensures that the only periodicity seen is $s=2$ (the authors use the term \textit{bipartite}). For these walks, \cite[Theorem 2.1.3]{LawlerLimic2010} gives a Gaussian pointwise description of the average $(p^{(n)}(x)+p^{(n+1)}(x))/2$ rather than $p^{(n)}(x)$. The following theorem shows that this pointwise description of averages can be deduced directly from Theorem \ref{thm:MainLLT} whenever an irreducible walk is periodic of period $s\geq 1$; in particular, the theorem recaptures the local limit theorems of \cite{LawlerLimic2010} (sans their sharp error) for irreducible, symmetric, bipartite walks on $G=\mathbb{Z}^d$ where $\varphi$ is the identity. \\

\begin{theorem}\label{thm:TimeAverageLLT}
Let $G$ be a finitely-generated abelian group and suppose that $p\in\mathcal{M}_2(G)$ drives a random walk on $G$ which is irreducible and periodic of period $s\geq 1$. 
\begin{enumerate}
\item If $G$ is finite, there exits $0\leq \rho<1$ for which
\begin{equation*}
\frac{p^{(n)}(x)+p^{(n+1)}(x)+\cdots+p^{(n+(s-1))}(x)}{s}=\frac{1}{\abs{G}}+O(\rho^n)
\end{equation*}
uniformly for $x\in G$ as $n\to\infty$. 
\item In the case that $G$ is infinite, $d=\rank(G_p)=\rank(G)\geq 1$ and so we take $\varphi:G\to\mathbb{Z}^d\cong G/\Tor(G)$ to be the epimorphism guarateed by Theorem \ref{thm:MainLLT} (which is the identity if $G=\mathbb{Z}^d$). If the pushforward $\varphi_*(p)\in\mathcal{M}(\mathbb{Z}^d)$ has zero mean (which holds, in particular, whenever the random walk is symmetric), then
\begin{equation*}
\frac{p^{(n)}(x)+p^{(n+1)}(x)+\cdots+p^{(n+(s-1))}(x)}{s}=\frac{1}{\Tor(G)}K_{\varphi_*(p)}^n(\varphi(x))+o(n^{-d/2})
\end{equation*}
uniformly for $x\in G$ as $n\to\infty$; here $K_{\varphi_*(p)}$ is that given by Theorem \ref{thm:MainLLT}. 
\end{enumerate}
\end{theorem}

\begin{proof}
As the first assertion follows straightforwardly from Corollary \ref{cor:ThetaAverage} and Theorem \ref{thm:MainLLT} (or Proposition \ref{prop:IntroLLT}), we shall focus on the second assertion. First, it is a basic fact that $\rank(H)=\rank(G)$ whenever $H$ is a finite-index subgroup of $G$. Consequently, $d=\rank(G_p)=\rank(G)$ by virtue of Proposition \ref{prop:IrreducAndPerS}. Let's take $\varphi:G\to\mathbb{Z}^d$ to be the homomorphism guaranteed by Theorem \ref{thm:MainLLT}  and set $K=K_{\varphi_*(p)}$ with $\Gamma=\Gamma_{\varphi_*(p)}$ and, by hypothesis, $\mu=\E(\varphi_*(p))=0$. Observe that, for fixed $k\in\mathbb{N}_+$,
\begin{equation*}
n^{d/2}\abs{\frac{e^{-\frac{y\cdot\Gamma^{-1}y}{2(n+k)}}}{(n+k)^{d/2}}-\frac{e^{-\frac{y\cdot\Gamma^{-1}y}{2n}}}{n^{d/2}}}\leq \left(\left(1-\frac{n^{d/2}}{(n+k)^{d/2}}\right)+\left(1-e^{-\frac{k(y\cdot\Gamma^{-1}y)}{2n(n+k)}}\right)\right)e^{-\frac{y\cdot\Gamma^{-1}y}{2(n+k)}}=o(1)
\end{equation*}
uniformly for $y\in\mathbb{Z}^d$ as $n\to\infty$ and, consequently, for every $k=1,2\dots,s-1$,
\begin{equation*}
K^{n+k}(\varphi(x))=K^n(\varphi(x))+o(n^{-d/2})
\end{equation*}
uniformly for $x\in G$ as $n\to\infty$.  With this estimate, we appeal to Theorem \ref{thm:MainLLT} (with $\mu=0$) and Corollary \ref{cor:ThetaAverage} to obtain
\begin{eqnarray*}
\frac{p^{(n)}(x)+p^{(n+1)}(x)+\cdots+p^{(n+(s-1))}(x)}{s}&=&\frac{1}{\Tor(G)}\frac{1}{s}\sum_{k=0}^{s-1}\Theta_p(n+k,x)K^{n+k}(\varphi(x))+o(n^{-d/2})\\
&=&\frac{1}{\Tor(G)}\left(\frac{1}{s}\sum_{k=0}^{s-1}\Theta_p(n+k,x)\right)K^n(\varphi(x))+o(n^{-d/2})\\
&=&\frac{1}{\Tor(G)}K^n(\varphi(x))+o(n^{-d/2})
\end{eqnarray*}
uniformly for $x\in G$ as $n\to\infty$.
\end{proof}

\noindent Before we conclude this section, let's turn our attention briefly to non-Gaussian behavior. Random walks on finite/compact non-abelian groups have been studied extensively, especially given their applications to card shuffling. References are too extensive to list here, so we point to the notes/surveys \cite{Diaconis1988,LSC2004}. As aperiodic and irreducible walks on finite/compact groups converge to the uniform measure $U=U_G$, it's helpful to measure this convergence by means of the total variation norm; this is given by
\begin{equation*}
\|\mu-\nu\|=\sup_{A\subseteq G}\abs{\mu(A)-\nu(A)}=\frac{1}{2}\sum_{x\in G}\abs{\mu(x)-\nu(x)}
\end{equation*}
for measures $\mu,\nu$ on $G$; the latter equality holding when $G$ is countable. In these terms, for an irreducible and aperiodic walk $p$ on a finite (or compact) group $G$, it is well known that $\|p^{(n)}-U\|$ converges exponentially to $0$ as $n\to\infty$ and the precise constants that quantify such exponential convergence give rise to cut-off phenomena. For these results, in addition to the above references, we point the reader to \cite{Kloss1959, DiaconisShahshahani1981, AldousDiaconis1986,AldousDiaconis1987}. In the case of finitely-generated abelian groups, the following theorem gives a compatible result without making any assumptions on irreducibility/aperiodicity.

\begin{theorem}\label{thm:ConvergenceToUniform}
Let $G$ be a finitely-generated abelian group. For $p\in\mathcal{M}(G)$, suppose that $G_p$ is finite, let $p_A$ denote the pushforward measure of $p$ onto $A=\Tor(G)$ and set
\begin{equation*}
\rho=\max\left\{\abs{\widehat{p_A}(\alpha)}:\alpha\in \widehat{A}\setminus\Omega(p_A)\right\}\in [0,1).
\end{equation*}
Then, for any $x_0\in\supp(p)$,
\begin{equation*}
\|p^{(n)}-\tau_{nx_0}(U_{G_p})\|\leq \frac{\abs{G_p}-1}{2} \rho^n 
\end{equation*}
for $n\in\mathbb{N}$; here, $U_{G_p}$ denotes the uniform distribution on $G_p$ and $\tau_y$ is the translation operator $\tau_y(f)(x)=f(x-y)$ for $y\in G$.
\end{theorem}

\section{Proof of Theorem \ref{thm:MainLLT}}\label{sec:Proof}

\noindent This section is dedicated to proving Theorem \ref{thm:MainLLT}. For reasons that will become clear from structure theorems, we shall begin our study by focusing on certain random walks on the group
\begin{equation*}
H=A\times \mathbb{Z}^{d_B}\times\mathbb{Z}^{d_C}
\end{equation*}
where $A$ is a finite abelian group and $d_B$ and $d_C$ are non-negative integers, one of which we assume to be positive. For our general theory, the difficult case happens, in some sense, when both $d_B$ and $d_C$ are positive and treating this correctly is the primary aim of the present development.  For simplicity, we shall write elements of $H$ as $h=(a,b,c)$. We remark that the Pontryagin dual of $H$ is
\begin{equation*}
\widehat{H}=\widehat{A}\times\mathbb{T}^{d_B}\times\mathbb{T}^{d_C}
\end{equation*}
with elements expressed as $(\alpha,\beta,\gamma)$ and the corresponding homomorphisms from $H$ to $\mathbb{S}$ by
\begin{equation*}
(a,b,c)\mapsto \chi_\alpha(a)e^{i\beta\cdot b}e^{i\gamma\cdot c}
\end{equation*}
for each $(\alpha,\beta,\gamma)\in \widehat{H}$. We shall assume that a given $p\in\mathcal{M}_2(H)$ satisfies one or both of the following properties.
\begin{enumerate}[label=(P\arabic*)]
\item\label{property:1} There is some element $c_0\in \mathbb{Z}^{d_C}$ for which
\begin{equation*}
\supp(p)\subseteq A\times \mathbb{Z}^{d_B}\times\{c_0\}.
\end{equation*}
\item\label{property:2} The dimension $d_B\geq 1$ and the support's projection on $\mathbb{Z}^{d_B}$,
\begin{equation*}
\Proj_{\mathbb{Z}^{d_B}}(\supp(p))=\{b\in\mathbb{Z}^{d_B}:p(a,b,c)>0\,\,\mbox{for some }(a,b,c)\in H\},
\end{equation*}
is a $d_B$-dimensional set in the sense of Definition \ref{def:dimension}.
\end{enumerate}
We remark that, in the case that $d_C=0$, we will take Property \ref{property:1} to be automatic.\\

\noindent Assuming that $p$ satisfies \ref{property:1}, let's make some observations. First, define $p_{AB}\in\mathcal{M}(A\times\mathbb{Z}^{d_B})$ by
\begin{equation*}
p_{AB}(a,b)=p(a,b,c_0)
\end{equation*}
for $(a,b)\in A\times\mathbb{Z}^{d_B}$. With this, observe that
\begin{equation*}
\widehat{p}(\alpha,\beta,\gamma)=\sum_{(a,b,c)\in H}p(a,b,c) \chi_{\alpha}(a)e^{i\beta \cdot b}e^{i\gamma\cdot c}=e^{i\gamma\cdot c_0}\sum_{(a,b)\in A\times\mathbb{Z}^{d_B}}p_{AB}(a,b)\chi_{\alpha}(a)e^{i \beta\cdot b}=e^{i\gamma \cdot c_0}\widehat{p_{AB}}(\alpha,\beta)
\end{equation*}
for $(\alpha,\beta,\gamma)\in\widehat{H}$. From this, we immediately obtain the following result.

\begin{lemma}\label{lem:p_to_pAB}
If $p$ satisfies Property \ref{property:1}, then
\begin{equation*}
\Omega(p)=\Omega(p_{AB})\times\mathbb{T}^{d_C}
\end{equation*}
and
\begin{equation*}
\widehat{p}(\alpha,\beta,\gamma)=\widehat{p_{AB}}(\alpha,\beta)e^{i\gamma\cdot c_0}
\end{equation*}
for all $(\alpha,\beta,\gamma)\in \widehat{H}$.
\end{lemma}

\noindent In the non-trivial case that $d_B\geq 1$, let's now define $p_B\in\mathcal{M}(\mathbb{Z}^{d_B})$ by
\begin{equation}\label{eq:pBdef}
p_B(b)=\sum_{a\in A}p_{AB}(a,b)=\sum_{a\in A}p(a,b,c_0)
\end{equation}
for $b\in\mathbb{Z}^{d_B}$. We have
\begin{lemma}\label{lem:FTpB}
Let $p$ satisfy \ref{property:1} and assume $d_B\geq 1$. For any $(\alpha_0,\beta_0)\in\Omega(p_{AB})$,
\begin{equation*}
\widehat{p_{AB}}(\alpha_0,\beta_0+\eta)=\widehat{p_{AB}}(\alpha_0,\beta_0)\widehat{p_B}(\eta)
\end{equation*}
for all $\eta\in\mathbb{T}^{d_B}$.
\end{lemma}
\begin{proof}
Fix $(\alpha_0,\beta_0)\in\Omega(p_{AB})$ and observe that $\chi_{\alpha_0}(a)e^{i\beta_0 \cdot b}=\widehat{p_{AB}}(\alpha_0,\beta_0)$ for all $(a,b)\in\supp(p_{AB})$ by virtue of Proposition \ref{prop:SamePhase}. Consequently, for each $\eta\in \mathbb{T}^{d_B}$,
\begin{eqnarray*}
\widehat{p_{AB}}(\alpha_0,\beta_0+\eta)&=&\sum_{(a,b)\in A\times \mathbb{Z}^{d_B}}p_{AB}(a,b)\chi_{\alpha_0}(a)e^{i(\beta_0+\eta)\cdot b}\\
&=&\sum_{(a,b)\in \supp(p_{AB})}p_{AB}(a,b)\chi_{\alpha_0}(a)e^{i\beta_0 \cdot b}e^{i\eta\cdot b}\\
&=&\sum_{(a,b)\in\supp(p_{AB})}p_{AB}(a,b)\widehat{p_{AB}}(\alpha_0,\beta_0)e^{i\eta \cdot b}\\
&=&\widehat{p_{AB}}(\alpha_0,\beta_0)\sum_{(a,b)\in A\times\mathbb{Z}^{d_B}}p_{AB}(a,b)e^{i\eta\cdot b}\\
&=&\widehat{p_{AB}}(\alpha_0,\beta_0)\sum_{b\in\mathbb{Z}^{d_B}}\sum_{a\in A}p_{AB}(a,b)e^{i\eta\cdot b}\\
&=&\widehat{p_{AB}}(\alpha_0,\beta_0)\sum_{b\in\mathbb{Z}^{d_B}}p_B(b)e^{i\eta\cdot b}\\
&=&\widehat{p_{AB}}(\alpha_0,\beta_0)\widehat{p_B}(\eta)
\end{eqnarray*}
thanks to the Fubini-Tonelli theorem.
\end{proof}

\noindent The following lemma is fairly standard and proofs can be found (essentially) in \cite[Proposition 2.7]{Spitzer}, \cite[Lemma 2.3.3]{LawlerLimic2010}, and \cite[Proposition 7.5]{RSC17}. Our proof here is given to highlight exactly how Property \ref{property:2} is used; the use of its equivalent hypothesis, the genuine $d_B$ dimensionality of $p_B$, is not made particularly clear in Proposition 2.7 of \cite{Spitzer}.

\begin{lemma}\label{lem:MeanVarLog}
Let $p\in\mathcal{M}_2(H)$ satisfy Properties \ref{property:1} and \ref{property:2}. Then $p_B$ has finite second moments (i.e., $p_B\in\mathcal{M}_2(\mathbb{Z}^{d_B})$) and drives a genuinely $d_B$-dimensional random walk on $\mathbb{Z}^{d_B}$ with mean $\mu_B\in\mathbb{R}^{d_B}$ and $d_B\times d_B$ covariance matrix $\Gamma_B$ given, respectively, by
\begin{equation*}
\{\mu_B\}_k=\sum_{b\in \mathbb{Z}^{d_B}}b_k p_B(b)
\end{equation*}
and
\begin{equation*}
\{\Gamma_B\}_{k,l}=\left(\sum_{b\in\mathbb{Z}^{d_B}}b_kb_lp_B(b)\right)-\{\mu_B\}_k\{\mu_B\}_l=\sum_{b\in\mathbb{Z}^{d_B}}(b_k-\{\mu_B\}_k)(b_l-\{\mu_B\}_l)p_B(b)
\end{equation*}
for $k,l=1,2,\dots,d_B$. Furthermore, $\Gamma_B$ is symmetric positive-definite and
\begin{equation}\label{eq:LogExpansion}
\operatorname{Log}(\widehat{p_B}(\eta))=i\mu_B\cdot\eta-\frac{\eta\cdot (\Gamma_B\eta)}{2}+o(\abs{\eta}^2)
\end{equation}
as $\eta\to 0$; here $\operatorname{Log}$ denotes the principal branch of the logarithm.
\end{lemma}

\begin{proof}
Consider the finite symmetric generating set $\mathcal{H}$ of $H$ defined by
\begin{equation*}
\mathcal{H}=\{(a,0,0):a\in A\}\cup \{(0,\pm e_j,0):j=1,2,\dots,d_B\}\cup\{(0,0,\pm e_j):j=1,2,\dots,d_C\}.
\end{equation*}
Using this set of generators, the norm $\abs{\cdot}_\mathcal{H}$ on $H$ satisfies $\abs{b}\leq \abs{h}_{\mathcal{H}}$ for every $h=(a,b,c)\in H$; here $\abs{b}$ denotes the Euclidean (or taxicab) norm on $\mathbb{Z}^d$. Thus, by the definition of $p_B$ and Tonelli's theorem, we have
\begin{equation*}
\sum_{b\in\mathbb{Z}^{d_B}}\abs{b}^2p_B(b)=\sum_{(a,b,c)\in H}\abs{b}^2p(a,b,c)\leq\sum_{h\in H}\abs{h}_{\mathcal{H}}^2 \,p(h)<\infty
\end{equation*}
since $p\in\mathcal{M}_2(H)$. Correspondingly, $p_B\in\mathcal{M}_2(\mathbb{Z}^{d_B})$. Using the definition of $p_B$, we also observe that
\begin{equation*}
\supp(p_B)=\Proj_{\mathbb{Z}^{d_B}}(\supp(p))
\end{equation*}
and so Property \ref{property:2} is equivalent to the condition that $p_B$ cannot be supported on any affine subspace of $\mathbb{R}^{d_B}$; this is precisely what it means for the random walk driven by $p_B$ to be genuinely $d_B$ dimensional.

Define $f(\eta)=\operatorname{Log}(\widehat{p_B}(\eta))$ and observe that, because $p_B\in\mathcal{M}_2(\mathbb{Z}^{d_B})$ and $\widehat{p_B}(0)=1$, $f$ is $C^2$ on a neighborhood of $\eta=0$ with
\begin{equation*}
f(\eta)=0+\sum_{k}\partial_j f(0)\eta_k+\frac{1}{2}\sum_{k,l}\partial_{k}\partial_l f(0)\eta_k\eta_l+o(\abs{\eta}^2)
\end{equation*}
as $\eta\to 0$ where the summations are taken for $k,l=1,2,\dots,d_B$. By direct calculation, we find that
\begin{equation*}
\partial_k f(0)=\frac{1}{\widehat{p_B}(0)}\sum_{b\in\mathbb{Z}^{d_B}}i b_kp_B(b)=i\sum_{b\in\mathbb{Z}^{d_B}}b_k p_B(b)=i\{\mu_B\}_k
\end{equation*}
for $k=1,2,\dots,d_B$. Similarly,
\begin{equation*}
   \partial_k\partial_l f(0)=\{\mu_B\}_k\{\mu_B\}_l- \sum_{b\in \mathbb{Z}^{d_B}}b_kb_lp_B(b)=-\{\Gamma_B\}_{kl}
\end{equation*}
for $k,l=1,2,\dots,d_B$. Thus,
\begin{equation*}
\operatorname{Log}(\widehat{p_B}(\eta))=f(\eta)=i\mu_B\cdot \eta-\frac{\eta\cdot \Gamma_B\eta}{2}+o(\abs{\eta}^2)
\end{equation*}
as $\eta\to 0$. As $\Gamma_B$ is clearly symmetric, it remains to prove it is positive definite and so, equivalently, that its quadratic form $\eta\mapsto\eta\cdot\Gamma_B\eta$ is positive definite. To this end, observe that
\begin{eqnarray*}
    \eta \cdot \Gamma_B\eta &=& \sum_{k,l} \eta_k \left (
    \sum_{b \in \mathbb{Z}^{d_B}}(b_k - \mu_k)(b_l-\mu_l)p_B(b)
    \right )
    \eta_l \\
    &=& \sum_{b \in \mathbb{Z}^{d_B}} \sum_{k,l} \eta_k (b_k-\mu_k) \eta_l (b_l-\mu_l) p_B(b) \\
    &=& \sum_{b\in \mathbb{Z}^{d_B}} \left [ \eta \cdot (b-\mu) \right ]^2 p_B(b)\\
    &=&\sum_{b\in\supp(p_B)}\left[\eta\cdot(b-\mu)\right]^2 p_B(b)
\end{eqnarray*}
where, for simplicity, we have written $\mu=\mu_B$. If this expression were zero for some $\eta\neq 0$,  $\eta$ would be orthogonal to the set of vectors $y=b-\mu\in\mathbb{R}^d$ as $b$ ranges over $\supp(p_B)$. In particular, this would imply that $\supp(p_B)-\mu$ is contained in the subspace perpendicular to $\eta\neq 0$, in contradiction to our hypothesis \ref{property:2}. Therefore
\begin{equation*}
\eta\cdot\Gamma_B\eta=\sum_{b\in\supp(p_B)}[\eta\cdot(b-\mu)]^2p_B(b)>0
\end{equation*}
whenever $\eta\neq 0$ which is precisely what we aimed to prove.
\end{proof}

One implication of the previous lemma and, in particular, \eqref{eq:LogExpansion} is that
\begin{equation*}
\Re(\operatorname{Log}(\widehat{p_B}(\eta)))+\frac{\eta\cdot \Gamma_B\eta}{2}=o(\abs{\eta}^2)
\end{equation*}
as $\eta\to 0$. Using the fact that positive-definite quadratic forms on $\mathbb{R}^{d_B}$ are comparable/equivalent, it follows that
\begin{equation*}
\Re(\operatorname{Log}(\widehat{p_B}(\eta)))\leq -\frac{\eta\cdot\Gamma_B\eta}{2}+\frac{1}{2}\left(\frac{\eta\cdot\Gamma_B\eta}{2}\right)=-\frac{\eta\cdot\Gamma_B\eta}{4}
\end{equation*}
for sufficiently small $\eta$. From this, we immediately obtain the following lemma.
\begin{lemma}\label{lem:p_B_Gauss_Envelope}
If $p\in\mathcal{M}_2(H)$ satisfies Properties \ref{property:1} and \ref{property:2}, then there is some open neighborhood $\mathcal{U}$ of $0$ in $\mathbb{T}^{d_B}$ for which
\begin{equation*}
\abs{\widehat{p_B}(\eta)}\leq e^{\frac{-\eta\cdot\Gamma_B \eta}{4}}
\end{equation*}
for $\eta\in \mathcal{U}$.
\end{lemma}

\begin{corollary}\label{cor:OmegaFinite}
If $p\in\mathcal{M}_2(H)$ satisfies Properties \ref{property:1} and \ref{property:2}, then
$\Omega(p_{AB})$ is finite.
\end{corollary}
\begin{proof}
Given $(\alpha_0,\beta_0)\in \Omega(p_{AB})$, we appeal to Lemmas \ref{lem:FTpB} and \ref{lem:p_B_Gauss_Envelope} to see that
\begin{equation*}
\abs{\widehat{p_{AB}}(\alpha_0,\beta_0+\eta)}=\abs{\widehat{p_{AB}}(\alpha_0,\beta_0)\widehat{p_B}(\eta)}\leq e^{-\frac{\eta\cdot\Gamma_B\eta}{4}}
\end{equation*}
for $\eta\in\mathcal{U}$. Since $\Gamma_B$ is positive definite in view of Lemma \ref{lem:MeanVarLog},
\begin{equation*}
\abs{\widehat{p_{AB}}(\alpha_0,\beta_0+\eta)}<1
\end{equation*}
whenever $\eta\in\mathcal{U}\setminus\{0\}$.  Using the finiteness of $\widehat{A}$ (and the fact that it carries the discrete topology), $(\alpha_0,\beta_0)$ must therefore be an isolated point of $\widehat{A}\times \mathbb{T}^{d_B}$. This argument shows that $\Omega(p_{AB})$ consists only of isolated points and, since $\widehat{A}\times\mathbb{T}^{d_B}$ is compact,  $\Omega(p_{AB})$ must be finite.
\end{proof}

We are now in a position to introduce the Gaussian density/attractor that will appear in our local limit theorems. We start by introducing the heat kernel $K_{p_B}^{(\cdot)}(\cdot):(0,\infty)\times\mathbb{R}^{d_B}\to (0,\infty)$ given by
\begin{eqnarray*}
K_{p_B}^t(b)&=&\frac{1}{(2\pi)^{d_B}}\int_{\mathbb{R}^{d_B}}e^{-t\frac{ \beta\cdot \Gamma_B\beta}{2}}e^{-i\beta\cdot b}\,d\beta\\
&=&\frac{1}{(2\pi t)^{d_B/2}\sqrt{\det(\Gamma_B)}}\exp\left(-\frac{b\cdot (\Gamma_B^{-1}b)}{2t}\right)
\end{eqnarray*}
for $t>0$ and $b\in\mathbb{R}^{d_B}$. We have the following result.

\begin{lemma}\label{lem:pBLLT}
Let $p\in\mathcal{M}_2(H)$ satisfy Properties \ref{property:1} and \ref{property:2}. Then, for any $\epsilon>0$, there is an open neighborhood $\mathcal{U}$ of $0$ in $\mathbb{T}^{d_B}$, which can be taken as small as desired, and a natural number $N=N(\epsilon,\mathcal{U})$ for which
\begin{equation}\label{eq:pBLLT}
\abs{\frac{1}{(2\pi)^{d_B}}\int_V \widehat{p_B}(\eta)^n e^{-i\eta\cdot b}\,d\eta-K_{p_B}^n(b-n\mu_B)}<\frac{\epsilon}{n^{d_B/2}}
\end{equation}
for $n\geq N$ and $b\in\mathbb{Z}^{d_B}$.
\end{lemma}
\begin{proof}
By an abuse of notation, we shall suppress some of the subscripts and write  $d=d_B$, $p=p_B$, $\mu=\mu_B$, and $\Gamma=\Gamma_B$ throughout the course of the proof. Let's fix $\epsilon>0$ and, by making an appeal to Lemma \ref{lem:p_B_Gauss_Envelope}, choose a neighborhood $\mathcal{U}\subseteq\mathbb{T}^d$ of $0$ for which
\begin{equation*}
\abs{\widehat{p}(\eta)}\leq e^{-\frac{\eta\cdot \Gamma\eta}{4}}
\end{equation*}
for $\eta\in\mathcal{U}$. With this, we define
\begin{equation*}
\mathcal{E}(n,b)=n^{d/2}\abs{\frac{1}{(2\pi)^d}\int_{\mathcal{U}} \widehat{p}(\eta)^ne^{-i\eta\cdot b}d\eta - K^n_p(b-n\mu)} 
\end{equation*}
for $n\in\mathbb{N}$ and $b\in\mathbb{Z}^d$. Observe that
\begin{eqnarray}\label{eq:pBLLT1}\nonumber
\mathcal{E}(n,b)&=& \abs{\frac{n^{d/2}}{(2\pi)^d}\int_{\mathcal{U}} \widehat{p}(\eta)^ne^{-i\eta\cdot b}d\eta - \frac{n^{d/2}}{(2\pi)^d}\int_{\mathbb{R}^d} e^{-n\frac{\eta \cdot \Gamma\eta}{2}}e^{-i\eta \cdot (b-n\mu)} d\eta}  \\ \nonumber
    &=&\abs{\frac{1}{(2\pi)^d}\int_{n^{1/2}(\mathcal{U})}\widehat{p}(u/\sqrt{n})^n e^{-i\frac{u\cdot b}{\sqrt{n}}}\,du-\frac{1}{(2\pi)^d}\int_{\mathbb{R}^d}e^{-\frac{u\cdot\Gamma u}{2}}e^{-i\frac{u\cdot(b-n\mu)}{\sqrt{n}}}\,du}\\ \nonumber
    &=&\abs{ \frac{1}{(2\pi)^d} \int_{n^{1/2}(\mathcal{U})} \widehat{p}(u/\sqrt{n})^ne^{-i\frac{u\cdot b}{\sqrt{n}}} - e^{-\frac{u\cdot\Gamma u}{2}}e^{-i\frac{u\cdot(b-n\mu)}{\sqrt{n}}} \,du + \frac{1}{(2\pi)^d} \int_{\mathbb{R}^d \setminus n^{1/2}(\mathcal{U})} e^{-\frac{u \cdot \Gamma\eta}{2}}e^{-i\frac{u\cdot(b-n\mu)}{\sqrt{n}}}\,du } \\ \nonumber
    &\leq & \int_{n^{1/2}(\mathcal{U})}f_n(u)\,du+\int_{\mathbb{R}^d\setminus n^{1/2}(\mathcal{U})}e^{-\frac{u\cdot\Gamma u}{2}}\,du\\
\end{eqnarray}
for $n\in\mathbb{N}$ and uniformly for $b\in\mathbb{Z}^d$; here we made the change of variables $\eta=u/\sqrt{n}$, invoked the triangle inequality, and set
\begin{equation*}
f_n(u)=\abs{\widehat{p}(u/\sqrt{n})^n-\exp\left(i\sqrt{n}u\cdot\mu-\frac{u\cdot\Gamma u}{2}\right)}.
\end{equation*}
Since $\eta=u/\sqrt{n}\in\mathcal{U}$ if and only if $u\in n^{1/2}(\mathcal{U})$, we have
\begin{equation*}
f_n(u)\leq \abs{\widehat{p}(n/\sqrt{n})^n}+\abs{e^{-\frac{u\cdot\Gamma u}{2}}}\leq \left(e^{-\frac{(u/\sqrt{n})\cdot \Gamma(u/\sqrt{n})}{4}}\right)^n+e^{-\frac{u\cdot\Gamma u}{2}}\\
=e^{-\frac{u\cdot\Gamma u}{4}}+e^{-\frac{u\cdot\Gamma u}{2}}
\end{equation*}
for all $u\in n^{1/2}(\mathcal{U})$ and $n\in\mathbb{N}$. Consequently, for all $n\in\mathbb{N}$,
\begin{equation}\label{eq:pBLLT2}
u\mapsto f(n,u)\mathds{1}_{n^{1/2}(\mathcal{U})}(u)\leq 2\exp\left(-\frac{u\cdot \Gamma u}{4}\right)\in L^1(\mathbb{R}^d).
\end{equation}
Observe also that
\begin{eqnarray*}
f(n,u)&=&\abs{\exp\left(n\operatorname{Log}(\widehat{p}(u/\sqrt{n}))\right)-\exp\left(i\sqrt{n}u\cdot\mu-\frac{u\cdot\Gamma u}{2}\right)}\\
&=&\abs{e^{nR(u/\sqrt{n})}-1}e^{-u\cdot\Gamma u}
\end{eqnarray*}
where
\begin{equation*}
R(\eta):=\operatorname{Log}(\widehat{p}(\eta))-\left(i\eta\cdot\mu-\frac{\eta\cdot\Gamma\eta}{2}\right)=o(\abs{\eta}^2)
\end{equation*}
as $\eta\to 0$ in view of Lemma \ref{lem:FTpB}. Consequently, for each $u\in\mathbb{R}^d$, $u\in n^{1/2}(\mathcal{U})$ for sufficiently large $n$ and
\begin{equation}\label{eq:pBLLT3}
\lim_{n\to\infty}f_n(u)=\lim_{t\to 0}\abs{\exp\left(\frac{R(tu)}{t^2}\right)-1}e^{-u\cdot\Gamma u}=0.
\end{equation}
Upon combining \eqref{eq:pBLLT2} and \eqref{eq:pBLLT3}, an appeal to the dominated convergence theorem gives us $N\in\mathbb{N}$ for which
\begin{equation*}
\int_{n^{1/2}(\mathcal{U})}f_n(u)\,du<\frac{\epsilon}{2}
\end{equation*}
for $n\geq N$. Since $u\mapsto e^{-\frac{u\cdot\Gamma u}{2}}\in L^1(\mathbb{R}^d)$, we can increase $N$ if necessary so that
\begin{equation*}
\int_{\mathbb{R}^d\setminus n^{1/2}(\mathcal{U})}e^{-\frac{u\cdot\Gamma u}{2}}\,du<\frac{\epsilon}{2}
\end{equation*}
for $n\geq N$. Inserting the two preceding estimates into \eqref{eq:pBLLT1}, we conclude that
\begin{equation*}
\mathcal{E}(n,b)<\epsilon
\end{equation*}
for all $n\geq N$ and $b\in\mathbb{Z}^d$.

\end{proof}

\noindent The following key lemma is a local limit theorem for those $p\in\mathcal{M}_2(H)$ satisfying Properties \ref{property:1} and \ref{property:2}. Though the proof is involved, it is straightforward and we've done our best to draw attention to the natural way in which $\Theta_p$ appears.

\begin{lemma}\label{lem:LLTforH}
Let $p\in\mathcal{M}_2(H)$ satisfy Properties \ref{property:1} and \ref{property:2} and define $\Theta_p$ by \eqref{eq:ThetaDef} where, in view of Corollary \ref{cor:OmegaFinite}, the Haar measure $\omega_p$ is normalized so that $\omega_p(\Omega(p))=\abs{\Omega(p_{AB})}$. Then
\begin{equation*}
p^{(n)}(h)=\frac{\Theta_p(n,h)}{\abs{A}}K_{p_B}^n(b-n\mu_B)+o(n^{-d_B/2})
\end{equation*}
uniformly for $h=(a,b,c)\in H$ as $n\to\infty$.
\end{lemma}

\begin{proof}
We will work under the assumption that $d_C\geq 1$; the simpler case in which $d_C=0$ can be easily handled by the argument below by suppressing all mentions of $\mathbb{Z}^{d_C}$, $\mathbb{T}^{d_C}$, $c$, $c_0$, $\gamma$, etc. By virtue of Lemma \ref{lem:p_to_pAB} and  the Fourier-convolution identity \eqref{eq:FTConvIdentity} in the present setting, we have
\begin{eqnarray*}
p^{(n)}(h)&=&\frac{1}{(2\pi)^{d_C}}\int_{\mathbb{T}^{d_C}}\frac{1}{(2\pi)^{d_B}}\int_{\mathbb{T}^{d_B}}\frac{1}{\abs{A}}\sum_{\alpha\in\widehat{A}}\widehat{p}(\alpha,\beta,\gamma)^n \chi_{\alpha}(-a)e^{-i\beta\cdot b}e^{-i\gamma\cdot c}\,d\beta\,d\gamma\\
&=&\frac{1}{(2\pi)^{d_C}}\int_{\mathbb{T}^{d_C}}\frac{1}{(2\pi)^{d_B}}\int_{\mathbb{T}^{d_B}}\frac{1}{\abs{A}}\sum_{\alpha\in\widehat{A}}e^{i n\gamma\cdot c_0}\widehat{p_{AB}}(\alpha,\beta)^n \chi_{\alpha}(-a)e^{-i\beta\cdot b}e^{-i\gamma\cdot c}\,d\beta\,d\gamma\\
&=&\left(\frac{1}{(2\pi)^{d_C}}\int_{\mathbb{T}^{d_C}}e^{in\gamma\cdot c_0}e^{-i\gamma\cdot c}\,d\gamma\right)\frac{1}{\abs{A}}\sum_{\alpha\in \widehat{A}}\frac{1}{(2\pi)^{d_B}}\int_{\mathbb{T}^{d_B}}\widehat{p_{AB}}(\alpha,\beta)^n \chi_{\alpha}(-a)e^{-i\beta\cdot b}\,d\beta
\end{eqnarray*}
for $n\in\mathbb{N}_+$ and $h=(a,b,c)\in H$. Setting
\begin{equation*}
\theta(c,n)=\frac{1}{(2\pi)^{d_C}}\int_{\mathbb{T}^{d_C}}e^{in\gamma \cdot c_0}e^{-i\gamma\cdot c}\,d\gamma
\end{equation*}
and expressing the remaining double integral as one over $\widehat{A}\times\mathbb{T}^{d_B}$ with respect to the Haar measure 
\begin{equation*}
d(\alpha,\beta)=\frac{d\#(\alpha)}{\abs{A}}\frac{d\beta}{(2\pi)^{d_B}}
\end{equation*}
we obtain
\begin{equation}\label{eq:LLTLem1}
p^{(n)}(h)=\theta(n,c)\int_{\widehat{A}\times\mathbb{T}^{d_B}}\widehat{p_{AB}}(\alpha,\beta)^n \chi_{\alpha}(-a)e^{-i\beta\cdot b}\,d(\alpha,\beta)
\end{equation}
for $n\in\mathbb{N}_+$ and $h=(a,b,c)\in H$.

With the above formula in hand, let's now focus our attention on $\Omega(p_{AB})$ and use it to split up the above integral in a way similar to that done in \eqref{eq:IntroSplit} in the introduction. Fix $\epsilon>0$ and, using Corollary \ref{cor:OmegaFinite}, enumerate
\begin{equation*}
    \Omega(p_{AB})=\{(\alpha_1,\beta_1),(\alpha_2,\beta_2),\dots,(\alpha_K,\beta_K)\}
\end{equation*}
so that $K=\abs{\Omega(p_{AB})}$. Appealing to Lemma \ref{lem:pBLLT}, select an open neighborhood $\mathcal{U}$ of $0\in\mathbb{T}^{d_B}$ and a natural number $N$ for which
\begin{equation}\label{eq:LLTLem2}
\abs{\frac{1}{(2\pi)^{d_B}}\int_\mathcal{U} \widehat{p_B}(\eta)^ne^{-ib\cdot\eta}\,d\eta-K_{p_B}^n(b-n\mu_B)}<\frac{\abs{A}}{2K}\frac{\epsilon}{n^{d_B/2}}
\end{equation}
for $b\in\mathbb{Z}^{d_B}$ and $n\geq N$. We now stick the members of $\Omega(p_{AB})$ inside the open neighborhoods
\begin{equation*}
\mathcal{O}_k:=\{(\alpha_k,\beta_k+\eta):\eta\in \mathcal{U}\}\subseteq \widehat{A}\times\mathbb{T}^{d_B}
\end{equation*}
for $k=1,2,\dots,K$ where, if necessary, we further decrease the size of $\mathcal{U}$ so that the above neighborhoods are disjoint. Given that each $\mathcal{O}_k$ is open and contains $(\alpha_k,\beta_k)$,
\begin{equation*}
\mathcal{R}=\widehat{A}\times\mathbb{T}^{d_B}\setminus \left(\bigcup_{k=1}^K\mathcal{O}_k\right)=\bigcap_{k=1}^K (\widehat{A}\times \mathbb{T}^{d_B}\setminus \mathcal{O}_k)
\end{equation*}
is a compact set disjoint from $\Omega(p_{AB})$ and so it follows that
\begin{equation}\label{eq:LLTLem3}
\rho:=\sup_{(\alpha,\beta)\in \mathcal{R}}\abs{\widehat{p_{AB}}(\alpha,\beta)}<1.
\end{equation}
With things set up in this way, we make use of \eqref{eq:LLTLem1} to write
\begin{equation*}
p^{(n)}(h)=\theta(n,c)\left(\sum_{k=1}^K\int_{\mathcal{O}_k}\widehat{p_{AB}}(\alpha,\beta)^n \chi_{\alpha}(-a)e^{-i\beta\cdot b}\,d(\alpha,\beta)+\int_{\mathcal{R}}\widehat{p_{AB}}(\alpha,\beta)^n \chi_{\alpha}(-a)e^{-i\beta\cdot b}\,d(\alpha,\beta)\right)
\end{equation*}
so that
\begin{equation}\label{eq:LLTLem4}
p^{(n)}(h)=\mathcal{I}(n,h)+\mathcal{E}(n,h)
\end{equation}
where
\begin{equation*}
    \mathcal{I}(n,h):=\theta(n,c)\sum_{k=1}^K\int_{\mathcal{O}_k}\widehat{p_{AB}}(\alpha,\beta)^n\chi_{\alpha}(-a)e^{-i\beta\cdot b}\,d(\alpha,\beta)
\end{equation*}
and
\begin{equation*}
    \mathcal{E}(n,h):=\theta(n,c)\int_\mathcal{R}\widehat{p_{AB}}(\alpha,\beta)^n\chi_\alpha(-a)e^{-i\beta\cdot b}\,d(\alpha,\beta)
\end{equation*}
for $n\in\mathbb{N}$ and $h=(a,b,c)\in H$. By the triangle inequality, we see that
\begin{equation*}
\abs{\mathcal{E}(n,h)}\leq \abs{\theta(n,c)\int_{\mathcal{R}}\abs{\widehat{p_{AB}}(\alpha,\beta)}^n\,d(\alpha,\beta)}\leq \frac{1}{\abs{A}(2\pi)^{d_B+d_C}}\sum_{\alpha\in\widehat{A}}\int_{\mathbb{T}^{d_C}}\int_{\mathbb{T}^{d_B}}\rho^n \,d\beta\,d\gamma=\rho^n
\end{equation*}
for $n\in\mathbb{N}_+$ and $h=(a,b,c)\in H$. By virtue of \eqref{eq:LLTLem3}, let's increase the value of $N\in\mathbb{N}_+$ (if necessary) to obtain
\begin{equation}\label{eq:LLTLem5}
    \abs{\mathcal{E}(n,h)}<\frac{1}{2}\frac{\epsilon}{n^{d_B/2}}
\end{equation}
for $n\geq N$ and $h\in H$. We now focus our attention on $\mathcal{I}(n,h)$. First observe that
\begin{eqnarray*}
\lefteqn{\hspace{-2cm}\sum_{k=1}^K\int_{\beta_k+\mathcal{U}}\widehat{p_{AB}}(\alpha_k,\beta)^n\chi_{\alpha_k}(-a)e^{-i\beta\cdot b}d\beta}\\ \hspace{2cm}&=& \sum_{k=1}^K\int_{\mathcal{U}}\widehat{p_{AB}}(\alpha_k,\beta_k+\eta)^n\chi_{\alpha_k}(-a)e^{-i(\beta_k+\eta) \cdot b}d\eta\\
 &=&\sum_{k=1}^K\int_{\mathcal{U}}\widehat{p_{AB}}(\alpha_k\beta_k)^n\widehat{p_B}(\eta)^n \chi_{\alpha_k}(-a)e^{-i\beta_k\cdot b} e^{-\eta\cdot b}\,d\eta\\
&=&\sum_{k=1}^K \widehat{p_{AB}}(\alpha_k,\beta_k)^n \chi_{\alpha_k}(-a)e^{-i\beta_k\cdot b}\int_{\mathcal{U}} \widehat{p_B}(\eta)^n e^{-i\eta\cdot b}\,d\eta\\
&=&\left(\sum_{(\alpha,\beta)\in\Omega(p_{AB})}\widehat{p_{AB}}(\alpha,\beta)^n\chi_{\alpha}(-a)e^{-i\beta\cdot b}\right)\int_{\mathcal{U}}\widehat{p_B}(\eta)^n e^{-i\eta\cdot b}\,d\eta
\end{eqnarray*}
for $n\in\mathbb{N}_+$ and $h=(a,b,c)\in H$ where we have made use of Lemma \ref{lem:FTpB}. Consequently,
\begin{equation*}
\mathcal{I}(n,h)=\left(\theta(n,c)\sum_{(\alpha,\beta)\in\Omega(p_{AB})}\widehat{p_{AB}}(\alpha,\beta)^n\chi_{\alpha}(-a)e^{-i\beta\cdot b}\right)\left(\frac{1}{\abs{A}(2\pi)^{d_B}}\int_V\widehat{p_B}(\eta)^n e^{-i\eta\cdot b}\,d\eta\right)
\end{equation*}
for $n\in\mathbb{N}_+$ and $h=(a,b,c)\in H$. As we will see shortly, the first term in parentheses above is precisely $\Theta_p$. Taking this for granted momentarily, i.e., assuming
\begin{equation}\label{eq:LLTLem6}
\Theta_p(n,h)=\theta(n,c)\sum_{(\alpha,\beta)\in\Omega(p_{AB})}\widehat{p_{AB}}(\alpha,\beta)^n\chi_{\alpha}(-a)e^{-i\beta\cdot b}
\end{equation}
for $n\in\mathbb{N}_+$ and $h=(a,b,c)\in H$, let us quickly deduce the stated local limit theorem. 

With our stated assumption, we have
\begin{equation}\label{eq:LLTLem7}
\mathcal{I}(n,h)=\frac{\Theta_p(n,h)}{\abs{A}}\frac{1}{(2\pi)^{d_B}}\int_{\mathcal{U}}\widehat{p_B}(\eta)^n e^{-i\eta\cdot b}\,d\eta
\end{equation}
for $n\in\mathbb{N}$ and $h=(a,b,c)\in H$. Combining \eqref{eq:LLTLem4} and \eqref{eq:LLTLem7}, we have
\begin{eqnarray*}
p^{(n)}(h)&=&\mathcal{I}(n,h)+\mathcal{E}(n,h)\\
&=&\frac{\Theta_p(n,h)}{\abs{A}}\frac{1}{(2\pi)^{d_B}}\int_{\mathcal{U}}\widehat{p_B}(\eta)^n e^{-i\eta\cdot b}\,d\eta+\mathcal{E}(n,h)\\
&=&\frac{\Theta_p(n,h)}{\abs{A}}K_{p_B}^n(b-n\mu_B)+\frac{\Theta_p(n,h)}{\abs{A}}\left(\frac{1}{(2\pi)^{d_B}}\int_{\mathcal{U}}\widehat{p_B}(\eta)^n e^{-i\eta\cdot b}\,d\eta-K_{p_B}^t(b-n\mu_B)\right)+\mathcal{E}(n,h)
\end{eqnarray*}
for $n\in\mathbb{N}$ and $h\in H$. By virtue of \eqref{eq:LLTLem2} and \eqref{eq:LLTLem5} and upon noting that $\abs{\Theta_p(n,h)}\leq \abs{\Omega(p_{AB})}=K$ thanks to \eqref{eq:LLTLem6}, we have
\begin{eqnarray*}
    \abs{p^{(n)}(h)-\frac{\Theta_p(n,h)}{\abs{A}}K_{p_B}^n(b-n\mu_B)}&\leq&
    \frac{\abs{\Theta_p(n,h)}}{\abs{A}}\abs{\frac{1}{(2\pi)^{d_B}}\int_{\mathcal{U}}\widehat{p_B}(\eta)^n e^{-i\eta\cdot b}\,d\eta-K_{p_B}^n(b-n\mu_B)}+\abs{\mathcal{E}(n,h)}\\
    &<&\frac{K}{\abs{A}}\frac{\abs{A}}{2K}\frac{\epsilon}{n^{d_B/2}}+\frac{\epsilon}{2n^{d_B/2}}\\
    &=&\frac{\epsilon}{n^{d_B/2}}
\end{eqnarray*}
for $n\geq N$ and $h=(a,b,c)\in H$.

It remains to prove \eqref{eq:LLTLem6}. Unraveling our definition of $\theta(n,c)$ and making use of Lemma \ref{lem:p_to_pAB} we see that
\begin{eqnarray*}
\lefteqn{\hspace{-3cm}\theta(n,c)\sum_{(\alpha,\beta)\in\Omega(p_{AB})}\widehat{p_{AB}}(\alpha,\beta)^n\chi_{\alpha}(-a)e^{-i\beta\cdot b}}\\  \hspace{3cm}&=&\sum_{(\alpha,\beta)\in\Omega(p_{AB})}\frac{1}{(2\pi)^{d_C}}\int_{\mathbb{T}^{d_C}}\widehat{p_{AB}}(\alpha,\beta)^n e^{in\gamma\cdot c_0}\chi_{\alpha}(-a)e^{-i\beta\cdot b}e^{-i\gamma\cdot c}\,d\gamma\\
 &=&\sum_{(\alpha,\beta)\in\Omega(p_{AB})}\frac{1}{(2\pi^{d_C})}\int_{\mathbb{T}^{d_C}}\widehat{p}(\alpha,\beta,\gamma)^n \chi_{\alpha}(-a)e^{-i\beta\cdot b}e^{-i \gamma c}\,d\gamma
\end{eqnarray*}
for $n\in\mathbb{N}_+$ and $h=(a,b,c)\in H$. By virtue of Lemma \ref{lem:p_to_pAB}, sum/integral above is simply an integral over $\Omega(p)=\Omega(p_{AB})\times\mathbb{T}^{d_C}$ with respect to Haar measure $\omega_p$  (which is the product of counting measure on $\Omega(p_{AB})$ and the normalized Lebesgue measure $d\gamma/(2\pi)^{d_{C}}$ on $\mathbb{T}^{d_C}$) and so we have
\begin{eqnarray*}
\theta(n,c)\sum_{(\alpha,\beta)\in\Omega(p_{AB})}\widehat{p_{AB}}(\alpha,\beta)^n\chi_{\alpha}(-a)e^{-i\beta\cdot b} 
&=&\sum_{(\alpha,\beta)\in\Omega(p_{AB})}\frac{1}{(2\pi^{d_C})}\int_{\mathbb{T}^{d_C}}\widehat{p}(\alpha,\beta,\gamma)^n \chi_{\alpha}(-a)e^{-i\beta\cdot b}e^{-i \gamma c}\,d\gamma\\
&=&\int_{\Omega(p)}\widehat{p}(\alpha,\beta,\gamma)^n \chi_{(\alpha,\beta,\gamma)}(-h)\,d\omega_p(\alpha,\beta,\gamma)\\
&=&\Theta_p(n,h)
\end{eqnarray*}
for $n\in\mathbb{N}_+$ and $h=(a,b,c)\in H$. Finally, to see that the normalization for $\omega_p$ is correct, we simply observe that
\begin{equation*}
\omega_p(\Omega(p))=\Theta_p(0,0)=\theta(0,0)\sum_{(\alpha,\beta)\in\Omega(p_{AB})}\widehat{p_{AB}}^0\chi_{\alpha}(0)e^{-i\beta\cdot 0}=\abs{\Omega(p_{AB})}.
\end{equation*}
\end{proof}

\begin{remark}\label{rmk:NormalizationH}
In the case that $d_C=0$ and $p=p_{AB}\in\mathcal{M}_2(A\times\mathbb{Z}^{d_B})$ satisfies Properties \ref{property:1} (automatically) and \ref{property:2}, the final paragraph of the proof of Lemma \ref{lem:LLTforH} shows that $\omega_p$ is counting measure on $\Omega(p_{AB})=\Omega(p)$. 
\end{remark}

\noindent We now turn to the general situation and let $G$ be a finitely-generated abelian group and $p\in\mathcal{M}_2(G)$. To complete our proof of Theorem \ref{thm:MainLLT} in the case that $G_p$ is infinite, we shall produce an isomorphism between $G$ and $H=A\times\mathbb{Z}^{d_B}\times\mathbb{Z}^{d_C}$ to push forward $p\in\mathcal{M}_2(G)$ to an element of $\mathcal{M}_2(H)$ satisfying \ref{property:1} and \ref{property:2} with $\rank(G_p)=d_B$. Still, it isn't clear that such a $p$-adapted isomorphism exists automatically; we know only that $G\cong \Tor(G)\times \mathbb{Z}^k$ for some natural number $k$ because $G$ is finitely-generated and abelian \cite[Theorem 3.13]{Jacobson}. Given such an isomorphism, our first lemma will allow us to show that $p$'s pushforward into $\mathbb{Z}^k$ is $\rank(G_p)$ dimensional. The second (and final) lemma will allow us to use this dimension information to ``twist" space in such a way that gives the isomorphism from $G$ into $H$ with $p$ satisfying \ref{property:1} and \ref{property:2}. 

\begin{lemma}\label{lem:rankdim}
Let $G$ be a finitely-generated abelian group and $\psi:G\to \Tor(G)\times \mathbb{Z}^k$ an isomorphism. Given a subset $X$ of $G$ containing $0$, denote by $\langle X\rangle$ the subgroup of $G$ generated by $X$\ and define $Y=\Proj_{\mathbb{Z}^k}\circ\psi(X)$. If $\rank(\langle X\rangle)=d$, then the set $Y\subseteq\mathbb{Z}^k$ is $d$ dimensional.
\end{lemma}
\begin{proof}
If $d=0$, then all elements of $X$ must have finite order and correspondingly must be sent to $0$ in $\mathbb{Z}^k$ by the homomorphism $\Proj_{\mathbb{Z}^k}\circ\psi$. Correspondingly, $Y$ is $0$ dimensional. We now assume that $d\geq 1$. Since $\langle X\rangle$ is itself a finitely-generated abelian group of rank $d$, there is a set of generators of the form
\begin{equation*}
\{x_1,x_2,\dots,x_d\}\cup\{z_1,z_2,\dots,z_n\}
\end{equation*}
where $\Tor(\langle X\rangle)=\{z_1,z_2,\dots,z_n\}$ and the list $x_1,x_2,\dots,x_d$ is (maximally) $\mathbb{Z}$-linearly independent. For each $j=1,2,\dots,d$, set 
\begin{equation*}
y_j=\Proj_{\mathbb{Z}_k}\circ\psi(x_j)
\end{equation*}
and observe that, since the list $\{x_1,x_2,\dots,x_d,z_1,z_2,\dots,z_n\}$ generates $X$, then
\begin{equation*}
Y\in\mbox{span}\{y_1,y_2,\dots,y_d\};
\end{equation*}
here, we are using the fact that the projection/quotient map collapses $\Tor(\langle X\rangle)\subseteq \Tor(G)$ into $0\in\mathbb{Z}^k$. To complete the proof, we must show two things:
\begin{enumerate}
\item The collection $\{y_1,y_2,\dots,y_d\}$ is $\mathbb{R}$-linearly independent.
\item No proper subset of $y_1,y_2,\dots,y_d$ spans $Y$.
\end{enumerate}
To see the first item, we first observe that $y_1,y_2,\dots,y_d$ are $\mathbb{Z}$-linearly independent. This follows immediately from the fact that they are the images of the $\mathbb{Z}$-linearly independent collection $x_1,x_2,\dots,x_d$ under the isomorphism followed by the projection/quotient map (which kills only torsion elements of $G$). Now, consider the equation
\begin{equation*}
\alpha_1y_1+\alpha_2y_2+\cdots+\alpha_dy_d=0.
\end{equation*}
If this equation had some non-zero solution over the real numbers, then Gaussian elimination would easily give us a basis for the solution space (with, at least, one non-zero vector). Further, since the vectors $y_1,y_2,\dots,y_d$  all have integer entries, the elementary row operators of the elimination yield a non-zero solution $(\alpha_1,\alpha_2,\dots,\alpha_d)\in\mathbb{Q}^d$ which can be scaled to be a non-zero member of $\mathbb{Z}^d$. Of course, this is impossible because we know that $y_1,y_2,\dots,y_d$ are $\mathbb{Z}$-linearly independent. We remark that another argument for $\mathbb{R}$-linear independence can be made on the basis that $\mathbb{R}$ is flat as a $\mathbb{Z}$ module.

To see the second item, suppose to reach a contradiction that some proper subset of $\{y_1,y_2,\dots,y_d\}$ spans $Y$, e.g., without loss of generality, assume that 
\begin{equation*}
Y\subseteq\mbox{span}\{y_2,y_3,\dots,y_d\}.
\end{equation*}
Since $\langle X\rangle$ has rank $d$, there must be some $x\in X$ for which $x=r_1x_1+\cdots+r_dx_d+t_1z_1+\cdots+t_n z_n$ with $r_1\neq 0$. Here, we have $y=\Proj_{\mathbb{Z}^k}\circ\psi(x)=r_1 y_1+r_2y_2+\cdots+r_dy_d\in Y.$ Since, by our supposition, $y\in\mbox{span}\{y_2,\dots,y_d\}$, we also have $y=s_2y_2+\cdots s_dy_d$. Thus
\begin{equation*}
0=y-y=r_1y_1+(r_2-s_2)y_2+\cdots+(r_d-s_d)y_d,
\end{equation*}
but this is impossible because $y_1,y_2,\dots,y_d$ are linearly independent over $\mathbb{Z}$ and, here, $r_1\neq 0$.
\end{proof}

\noindent While the following lemma is straightforward, its proof requires some delicate linear algebra and so we have decided to give a complete proof in Appendix \ref{app:reduce}.

\begin{lemma}{\label{lem:reduce_many}}
Let $S\subseteq \mathbb{Z}^k$ be $d$ dimensional for some $1\leq d< k$. Then, there exists $\Phi\in\Aut(\mathbb{Z}^k)$ and $w\in \mathbb{Z}^{k-d}$ for which
\begin{eqnarray*}
\Phi(S)\subseteq \mathbb{Z}^d\times\{w\}
\end{eqnarray*}
and $\Proj_{\mathbb{Z}^d}(\Phi(S))$ is $d$ dimensional; here, $\Proj_{\mathbb{Z}^d}$ denotes the projection from $\mathbb{Z}^k$ onto the first $d$ coordinates.
\end{lemma}

\noindent We are now in a position to prove the main theorem.

\begin{proof}[Proof of Theorem \ref{thm:MainLLT}] We take $G$ to be a finitely-generated abelian group and $p\in\mathcal{M}_2(G)$. As we discussed above, well-known structure theorems give us an isomorphism $\psi:G\to\Tor(G)\times\mathbb{Z}^k$ for some non-negative integer $k$. Of course, $k=0$ exactly when $G$ is finite. In this case, $G_p$ is necessarily finite, $\Tor(G)=G$, and the stated result follows directly from Proposition \ref{prop:IntroLLT}. We therefore assume that $k\geq 1$ and consider the following subsets of $\mathbb{Z}^k$:
\begin{equation*}
S=\Proj_{\mathbb{Z}^k}(\supp(p\circ \psi^{-1}))=\Proj_{\mathbb{Z}^k}\circ\psi(\supp(p))
\end{equation*}
and, for $x_0\in\supp(p)$,
\begin{equation*}
Y:=S-s_0=\Proj_{\mathbb{Z}^k}\circ\psi(\supp(p)-x_0).
\end{equation*}
where $s_0=\Proj_{\mathbb{Z}^k}\circ\psi(x_0)$. Taking $d=\rank(G_p)$, Lemma \ref{lem:rankdim} guarantees that $Y$ is $d$ dimensional (using $X=\supp(p)-x_0$) and so, equivalently, $S$ is $d$ dimensional. With this, we break our argument into two cases:\\

\noindent \textbf{Case 1: $d\geq 1$.} If $d=1,2,\dots,k-1$, we apply Lemma \ref{lem:reduce_many} to $S$ 
 for $d_B:=d=1,2,\dots,k-1$ to obtain $\Phi\in\Aut(\mathbb{Z}^k)$ having
\begin{equation*}
\Phi(S)\subseteq \mathbb{Z}^{d_B}\times \{c_0\}
\end{equation*}
for $c_0=w\in \mathbb{Z}^{d_C}$ where $d_B+d_C=k$. In the case that $d=k$, we take $\Phi$ to be the identity automorphism (here, $d_B=k$ and $d_C=0$) and, in this case, the following argument goes through completely with the factor $\mathbb{Z}^{d_C}$ suppressed along with all mentions of $c$ and $c_0$. With $\Phi$ in hand, we form a new isomorphism 
\begin{equation*}
T=\zeta\circ\psi:G\to H=A\times\mathbb{Z}^{d_B}\times\mathbb{Z}^{d_C}
\end{equation*}
where $A=\Tor(G)$ and $\zeta^{-1}:H\to \Tor(G)\times \mathbb{Z}^k$ is given by $\zeta^{-1}(h)=(a,\Phi^{-1}(b,c))$ for $h=(a,b,c)\in H$. Set
\begin{equation*}
q=T_*(p)=p\circ T^{-1}\in \mathcal{M}(H).
\end{equation*}
If $\mathcal{G}$ is a finite symmetric generating set of $G$, it is easy to see that $\mathcal{H}=T(\mathcal{G})$ is a finite symmetric generating set of $H$ and $\abs{h}_\mathcal{H}=\abs{T^{-1}(h)}_\mathcal{G}$ for all $h\in H$. Consequently,
\begin{equation*}
\sum_{h\in H}\abs{h}_{\mathcal{H}}^2\,q(h)=\sum_{x\in G}\abs{x}_\mathcal{G}^2\, p(x)<\infty
\end{equation*}
and therefore $q\in\mathcal{M}_2(H)$. We claim that $q$ satisfies Properties \ref{property:1} and \ref{property:2}. First, if $h=(a,b,c)\in\supp(q)$ so that
\begin{equation*}
p\circ\psi^{-1}( \zeta^{-1}(h))=p(T^{-1}(h))=q(h)>0,
\end{equation*}
we have
\begin{equation*}
\Phi^{-1}(b,c)=\Proj_{\mathbb{Z}^k}(\zeta^{-1}(h))\in\Proj_{\mathbb{Z}^k}(\supp(p\circ\psi^{-1}))=S.
\end{equation*}
Consequently,
\begin{equation*}
\Proj_{\mathbb{Z}^{d_B}\times\mathbb{Z}^{d_C}}(\supp(q))\subseteq \Phi(S).
\end{equation*}
By reversing this argument, it is straightforward to verify that the reverse inclusion holds as well, i.e.,
\begin{equation}\label{eq:MainLLT1}
\Proj_{\mathbb{Z}^{d_B}\times\mathbb{Z}^{d_C}}(\supp(q))= \Phi(S).
\end{equation}
Since $\Phi$ is that guaranteed by Lemma \ref{lem:reduce_many} (for $d=d_B$), we immediately obtain
\begin{equation*}
\supp(q)\subseteq A\times \mathbb{Z}^{d_B}\times\{c_0\}
\end{equation*}
where $c_0\in \mathbb{Z}^{d_C}$; this is Property \ref{property:1}. Further, given that $S$ is $d_B$ dimensional, Lemma \ref{lem:reduce_many} also tells us that $\Proj_{\mathbb{Z}^{d_B}}(\Phi(S))$ is also $d_B$ dimensional and and hence Property \ref{property:2} is satisfied in view of \eqref{eq:MainLLT1}. 

We are now in a position to make use of Lemma \ref{lem:LLTforH}. We have
\begin{equation*}
q^{(n)}(h)=\frac{\Theta_{q}(n,h)}{\abs{A}}K_{q_B}^n(b-n\mu_B)+o(n^{-d_B/2})
\end{equation*}
uniformly for $h=(a,b,c)\in H$ as $n\to\infty$. In terms of the isomorphism $T:G\to H$, this yields
\begin{equation*}
p^{(n)}(x)=q^{(n)}(T(x))=\frac{\Theta_{q}(n,T(x))}{\abs{A}}K_{q_B}^n(b-n\mu_B)+o(n^{-d_B/2})
\end{equation*}
uniformly for $x\in G$ as $n\to\infty$ where $T(x)=(a,b,c)$. Observe that $\abs{A}=\abs{\Tor(G)}$ and, by virtue of Proposition \ref{prop:ThetaUnderIso},
\begin{equation*}
\Theta_{p}(n,x)=\Theta_{q}(n,T(x))
\end{equation*}
provided that the Haar measure $\omega_p$ on $\Omega(p)$ is taken with the normalization
\begin{equation}\label{eq:MainLLTNormalization}
\omega_p(\Omega(p))=\omega_{q}(\Omega(q))=\abs{\Omega(q_{AB})}.
\end{equation}
Consequently, 
\begin{equation}\label{eq:MainLLT2}
p^{(n)}(x)=\frac{\Theta_p(n,x)}{\abs{\Tor(G)}}K_{q_B}^n(b-n\mu_B)+o(n^{-d_B/2})
\end{equation} 
uniformly for $x\in G$ as $n\to\infty$ where $T(x)=(a,b,c)$. 

It remains to show that the Gaussian $K_{p_B}$ coincides with that in the theorem. To this end, consider now the surjective homomorphism $\varphi:G\to\mathbb{Z}^{d_B}$ given by
$\varphi=\Proj_{\mathbb{Z}^{d_B}}\circ T$ so that $\varphi(x)=b$ for $x\in G$. We have, for $b_0\in\mathbb{Z}^{d_B}$,
\begin{equation*}
\varphi_*(p)(b_0)=p(\varphi^{-1}(b_0))=q(\Proj_{\mathbb{Z}^{d_B}}^{-1}(b_0))=q(\{(a,b,c)\in H:b=b_0\})
\end{equation*}
and so
\begin{equation}\label{eq:MainLLT*3}
\varphi_*(p)(b_0)=\sum_{a\in A}\sum_{c\in \mathbb{Z}^{d_C}}q(a,b_0,c)=\sum_{a\in A}q(a,b_0,c_0)=q_B(b_0)
\end{equation}
in view of \eqref{eq:pBdef} and thanks to our previous appeal to Lemma \ref{lem:reduce_many}. Correspondingly, $q_B\in\mathcal{M}_2(\mathbb{Z}^{d_B})$ is precisely the pushforward $\varphi_*(p)$ and hence $\mu=\E[\varphi_*(p)]=\mu_B$ and $\Gamma=\Cov[\varphi_*(p)]=\Gamma_B.$ From this, we immediately obtain
\begin{equation}\label{eq:MainLLT3}
K_{\varphi_*(p)}(\varphi(x)-n\mu)=K_{q_B}(b-n\mu_B)=\frac{1}{(2\pi n)^{d/2}\sqrt{\det(\Gamma)}}\exp\left(-\frac{(\varphi(x)-n\mu)\cdot\Gamma^{-1}(\varphi(x)-n\mu)}{2n}\right)
\end{equation}
for $x\in G$ and $n\in\mathbb{N}_+$ where we recalled that $d=d_B$. The result now follows immediately upon combining \eqref{eq:MainLLT2} and \eqref{eq:MainLLT3}.\\

\noindent \textbf{Case 2: $d=0$.} As we will see, in this case, there is no reason to look for an additional isomorphism.  Thus, we set $H=A\times\mathbb{Z}^k$ where $A=\Tor(G)$,  $T=\psi:G\to H$, and $q=T_*(p)$. Since $Y=S-s_0$ contains only the zero vector in view of Lemma \ref{lem:rankdim},
\begin{equation*}
\supp(q)\subseteq A\times\{c_0\}
\end{equation*}
where $c_0:=s_0=\Proj_{\mathbb{Z}^k}(y_0)\in\mathbb{Z}^k$ and $y_0=\psi(x_0)\in\supp(q)$. With this in mind, we make an appeal to Lemma \ref{lem:p_to_pAB} (here, with all variables $b$ and $\beta$ suppressed) to find that
\begin{equation}\label{eq:MainLLT4}
\Omega(q)=\Omega(q_A)\times\mathbb{T}^k\hspace{1cm}\mbox{and}\hspace{1cm}\widehat{q}(\alpha,\gamma)=\widehat{q_A}(\alpha)e^{i\gamma\cdot c_0}
\end{equation}
for all $(\alpha,\gamma)\in \widehat{H}=\widehat{A}\times\mathbb{T}^k$; here, $q_A\in\mathcal{M}(A)$ is given by $q_A(a)=q(a,c_0)=\sum_{c}q(a,c)$ for $a\in A$. In this setting, the Fourier-convolution identity \eqref{eq:FTConvIdentity} becomes
\begin{equation*}
q^{(n)}(a,c)=\frac{1}{\abs{A}}\sum_{\alpha\in\widehat{A}}\frac{1}{(2\pi)^k}\int_{\mathbb{T}^k}\widehat{q}(\alpha,\gamma)^n\chi_\alpha(-a)e^{-i\gamma\cdot c}\,d\gamma
\end{equation*}
for $n\in\mathbb{N}_+$ and $(a,c)\in H$. In view of \eqref{eq:MainLLT4}, it is easy to see that
\begin{eqnarray*}
q^{(n)}(a,c)&=&\frac{1}{\abs{A}}\sum_{\Omega(q_A)}\frac{1}{(2\pi)^k}\int_{\mathbb{T}^k}\widehat{q}(\alpha,\gamma)^n\chi_{\alpha}(-a)e^{-i\gamma\cdot c}\,d\gamma+\frac{1}{\abs{A}}\sum_{\widehat{A}\setminus\Omega(q_A)}\frac{1}{(2\pi)^k}\int_{\mathbb{T}^k}\widehat{q}(\alpha,\gamma)^n\chi_{\alpha}(-a)e^{-i\gamma\cdot c}\,d\gamma\\
&=&\frac{1}{\abs{A}}\int_{\Omega(q)} \widehat{q}(\alpha,\gamma)^n\chi_{\alpha}(-a)e^{-i\gamma\cdot c}\,d\omega_q(\alpha,\gamma)+\frac{1}{\abs{A}}\sum_{\widehat{A}\setminus\Omega(q_A)}\frac{1}{(2\pi)^k}\int_{\mathbb{T}^k}\widehat{q_A}(\alpha)^n e^{i\gamma\cdot(nc_0-c)} \chi_{\alpha}(-a)\,d\gamma\\
&=&\frac{1}{\abs{A}}\Theta_q(n,a,c)+\mathcal{E}(n,a,c)
\end{eqnarray*}
for $n\in\mathbb{N}_+$ and $(a,c)\in H$ where we have put
\begin{equation*}
\mathcal{E}(n,a,c)=\frac{1}{\abs{A}}\sum_{\widehat{A}\setminus\Omega(q_A)}\frac{1}{(2\pi)^k}\int_{\mathbb{T}^k}\widehat{q_A}(\alpha)^n e^{i\gamma\cdot(nc_0-c)}\chi_{\alpha}(-a)\,d\gamma
\end{equation*}
and noted that $d\omega_q(\alpha,\gamma)=d\#(\alpha)\times d\gamma/(2\pi)^k$ is the Haar measure on $\Omega(q)=\Omega(q_A)\times\mathbb{T}^k$ with normalization
\begin{equation*}
\omega_q(\Omega(q))=\abs{\Omega(q_A)}=\Theta_q(0,0,0).
\end{equation*}
Of course,
\begin{equation}\label{eq:MainLLT5}
\abs{\mathcal{E}(n,a,c)}\leq \frac{\abs{\widehat{A}\setminus \Omega(q_A)}}{\abs{A}}\rho^n<\rho^n
\end{equation}
where $\rho=\max\left\{\abs{\widehat{q_A}(\alpha)}:\alpha\in \widehat{A}\setminus\Omega(q_A)\right\}<1$ because $\widehat{A}$ is finite. Consequently,
\begin{equation*}
q^{(n)}(a,c)=\frac{\Theta_q(n,a,c)}{\abs{A}}+O(\rho^n)
\end{equation*}
uniformly for $(a,c)\in H$ as $n\to\infty$. The proof is now completed by repeating the same argument used in the previous case to obtain $p^{(n)}$ from $q^{(n)}$ via Proposition \ref{prop:ThetaUnderIso}. As this argument is simpler in this case because there is no Gaussian density to deal with, we leave these details to the reader.
\end{proof}

\noindent The following proposition is a follow-up to Remark \ref{rmk:Normalization}. By \textit{``$\omega_p$ is counting measure''}, we mean that the choice of $\omega_p$ appearing in the statement of Theorem \ref{thm:MainLLT} is counting measure. 

\begin{proposition}\label{prop:CountingMeasure}
For a finitely-generated abelian group $G$, let $p\in\mathcal{M}_2(G)$ and let $\Omega(p)$ be given by \eqref{eq:OmegaDef}. If $\Omega(p)$ is finite, then $\omega_p$ is counting measure. In particular, $\omega_p$ is counting measure whenever $p\in\mathcal{M}_2(G)$ is irreducible and periodic of period $s=\omega_p(\Omega(p))$ in view of Corollary \ref{cor:OmegaFinite}.
\end{proposition}

\begin{proof}
In the case of a finite group $G$, the assertion is made clear by \eqref{eq:IntroLLT1}. In the case of an infinite group $G\cong A\times\mathbb{Z}^k$ and $p\in\mathcal{M}_2(G)$, the above proof shows that $\Omega(p)$ is finite exactly when $d=k$ and $q=T_*(p)\in\mathcal{M}_2(H)$ satisfies Properties \ref{property:1} and \ref{property:2} with $d_C=0$ so that $q=q_{AB}$.  Since $\omega_q$ is counting measure in view of Remark \ref{rmk:NormalizationH} and $\widehat{T}$ is an isomorphism from $\Omega(q)$ to $\Omega(p)$, the normalization \eqref{eq:MainLLT3} gives $\omega_p(\Omega(p))=\abs{\Omega(q)}=\abs{\Omega(p)}$.
\end{proof}

\begin{proof}[Proof of Theorem \ref{thm:ConvergenceToUniform}]
In the case that $G_p$ is finite, $p\in\mathcal{M}_2(G)$ since it has finite support and, so we may follow the second case of the proof of Theorem \ref{thm:MainLLT}.  After transforming back to $p$ from $q=T_*(p)$, the theorem's proof gives
\begin{equation*}
p^{(n)}(x)=\frac{\Theta_p(n,x)}{\abs{A}}+\mathcal{E}(n,x)
\end{equation*}
for all $x\in G$ and $n\in\mathbb{N}_+$ where $\mathcal{E}(n,x)=\mathcal{E}(n,a,c)$ satisfies \eqref{eq:MainLLT5} and the pushforward $p_A$ is precisely $q_A$. Now, for all $x\in G$ and $n\in\mathbb{N}$,
\begin{equation*}
\Theta_p(n,x)=\Theta_q(n,a,c)=\abs{\Omega(p_A)}\mathds{1}_{G_p}(x-nx_0)
\end{equation*}
where $T(x)=(a,c)$. Using the normalization argument by which we obtained Proposition \ref{prop:IntroLLTinIndicator}, it is easy to see that $\Theta_p(0,0)=\abs{\Omega(p_A)}=\abs{A}/\abs{G_p}.$ Thus, in view of Proposition \ref{prop:ThetaIsIndicator}, we find that, for all $n\in\mathbb{N}_+$ and $x\in G$, 
\begin{equation*}
p^{(n)}(x)=\frac{1}{\abs{G_p}}\mathds{1}_{G_p}(x-nx_0)+\mathcal{E}(n,x)
\end{equation*} where
\begin{equation*}
\abs{\mathcal{E}(n,x)}\leq \frac{\abs{\widehat{A}\setminus\Omega(p_A)}}{\abs{A}}\rho^n=\frac{\abs{A}-\abs{\Omega(p_A)}}{\abs{A}}\rho^n=\left(1-\frac{1}{\abs{G_p}}\right)\rho^n.
\end{equation*}
Upon recalling that $\supp(p^{(n)})\subseteq \supp(\Theta_p(n,\cdot))=G_p+nx_0$ thanks to Proposition \ref{prop:ThetaCapturesSupport}, we obtain
\begin{equation*}
\|p^{(n)}-\tau_{nx_0}(U_{G_p})\|=\frac{1}{2}\sum_{x\in G_p+nx_0}\abs{p^{(n)}(x)-\frac{1}{\abs{G_p}}\mathds{1}_{G_p}(x-nx_0)}=\frac{1}{2}\sum_{x\in G_p+nx_0}\abs{\mathcal{E}(n,x)}\leq \left(\frac{\abs{G_p}-1}{2}\right)\rho^n
\end{equation*}
for all $n\in\mathbb{N}_+$.
\end{proof}

\section{Examples}\label{sec:Examples}

In this section, we study several examples that illustrate our local limit theorems and the appearance of the dance function $\Theta_p$. We encourage the reader to see Section 7 of \cite{RSC17} in which several additional examples can be found, all of which are genuinely $d$-dimensional walks on $\mathbb{Z}^d$ (with non-trivial dance functions).

\begin{example}[A Simple Random Walk on $\mathbb{Z}_9$] In this example, we study a class of walks on $\mathbb{Z}_9$ where the driving measure is supported on two points. Fix distinct $a,b\in\mathbb{Z}_9$ and set
\begin{equation}\label{eq:Example1}
p(x) = \begin{cases}
    1/2 & \mbox{if}\quad x = a,b \\
    0 & \mbox{else}
\end{cases}
\end{equation}
for $x\in \mathbb{Z}_9$. We have
\begin{eqnarray*}
\widehat{p}(k) = \frac{1}{2}\left ( e^{2\pi iak/9} + e^{2\pi ibk/9} \right )
\end{eqnarray*}
for $k = 1,\dots, 8$. With this, it is easy to see that $\abs{\widehat{p}(k)} = 1$ exactly when $9|(a-b)k$.  Consequently, $\Omega(p)$ depends only on $a-b$ and we have exactly two cases:
\begin{enumerate}
\item If $\gcd(a-b,9)=1$, i.e., $a-b$ and $9$ are relatively prime, $\Omega(p)=\{0\}$. 
\item If $\gcd(a-b,9)=3$, i.e., $a-b = \pm 3$ or $\pm 6$, $\Omega(p) = \{0,3,6\}$,
\end{enumerate}
In the first case, we easily see that $\Theta_p\equiv 1$ and our random walk is irreducible and aperiodic (thanks to Corollary \ref{cor:AperAndIrreducPartialConverse}). Also, nothing depends on $a$ or $b$ (beyond having $\gcd(a-b,9=1)$). In the second case, observe that $\widehat{p}(3)=e^{2\pi i a/3}$ and $\widehat{p}(6)=e^{4\pi i a/3}$ by virtue of Proposition \ref{prop:SamePhase}. Consequently,
\begin{equation*}
\Theta_p(n,x)=1+e^{2\pi i (an-x)/3}+e^{4\pi i(an-x)/3}=\begin{cases} 3 &\mbox{if}\quad 3\vert (x-an)\\ 0 &\mbox{ else }
\end{cases}
\end{equation*}
for $n\in\mathbb{N}$ and $x\in\mathbb{Z}_9$. From this we see that $\Theta_p(n,x)=3\mathds{1}_{3\mathbb{Z}_9}(x-an)$ consistent with Proposition \ref{prop:ThetaIsIndicator} where $G_p=3\mathbb{Z}_9\leq \mathbb{Z}_9$ and $a=x_0\in\supp(p)$. Of course, $a$ can be replaced by $b$ with no change. Whether this random walk explores the entire group $\mathbb{Z}_9$ periodically or sticks to $3\mathbb{Z}_9$ now simply depends on the value of $a$ (or $b$) modulo $3$. Let's consider several sub-examples.
\begin{enumerate}
\item Consider $p$ defined by \eqref{eq:Example1} where $a=1$ and $b=3$. Here, we see that $a-b=2$ and $9$ are relatively prime, and so we follow the first case. We have $\Omega(p)=\{0\}$, $\Theta_p\equiv 1$,  $G_p=\mathbb{Z}_9$, and it is readily calculated that
\begin{equation*}
\rho=\max\{\abs{\widehat{p}(k)}:k\in\mathbb{Z}_9\setminus\Omega(p)\}=\frac{1}{2} \sqrt{2+\sqrt{3} \sin \left(\frac{\pi }{9}\right)+\cos \left(\frac{\pi
   }{9}\right)}<0.94.
\end{equation*}
Thus, an appeal to Theorem \ref{thm:MainLLT} (or Proposition \ref{prop:IntroLLTinIndicator}) gives
\begin{equation*}
p^{(n)}(x)=\frac{1}{9}+O(\rho^n)=\frac{1}{9}+o((0.94)^n)
\end{equation*}
uniformly for $x\in\mathbb{Z}_9$ as $n\to\infty$. Further, an appeal to Theorem \ref{thm:ConvergenceToUniform} (with $U_{G_p}=U_G=U$) gives us the total-variation norm estimate
\begin{equation*}
\|p^n-U\|\leq \frac{\abs{G_p}-1}{2}\rho^n< 4(0.94)^n
\end{equation*}
for $n\in\mathbb{N}$. These asymptotics are illustrated in Figure \ref{fig:Example1.0} where we see the random walk reaches all points in $\mathbb{Z}_9$, albeit slowly. 

\begin{table}[!h]
  \centering
  \begin{tabular}{  | c | c | c | }
    \hline
    $n=1$ & $n=2$ & $n=3$ \\ \hline
    
    \begin{minipage}{.25\textwidth}
      \includegraphics[width=\linewidth]{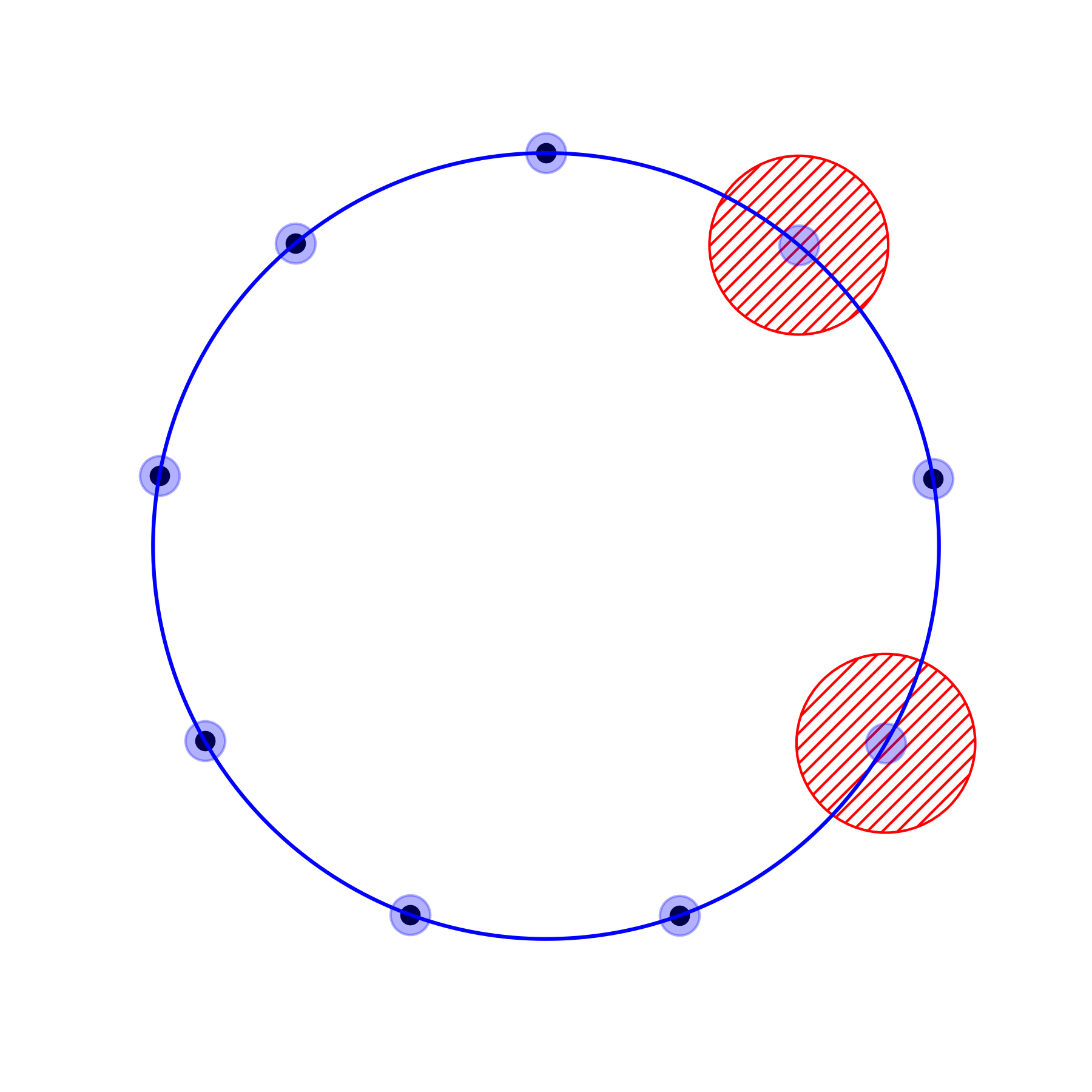}
    \end{minipage}
    &
    \begin{minipage}{.25\textwidth}
      \includegraphics[width=\linewidth]{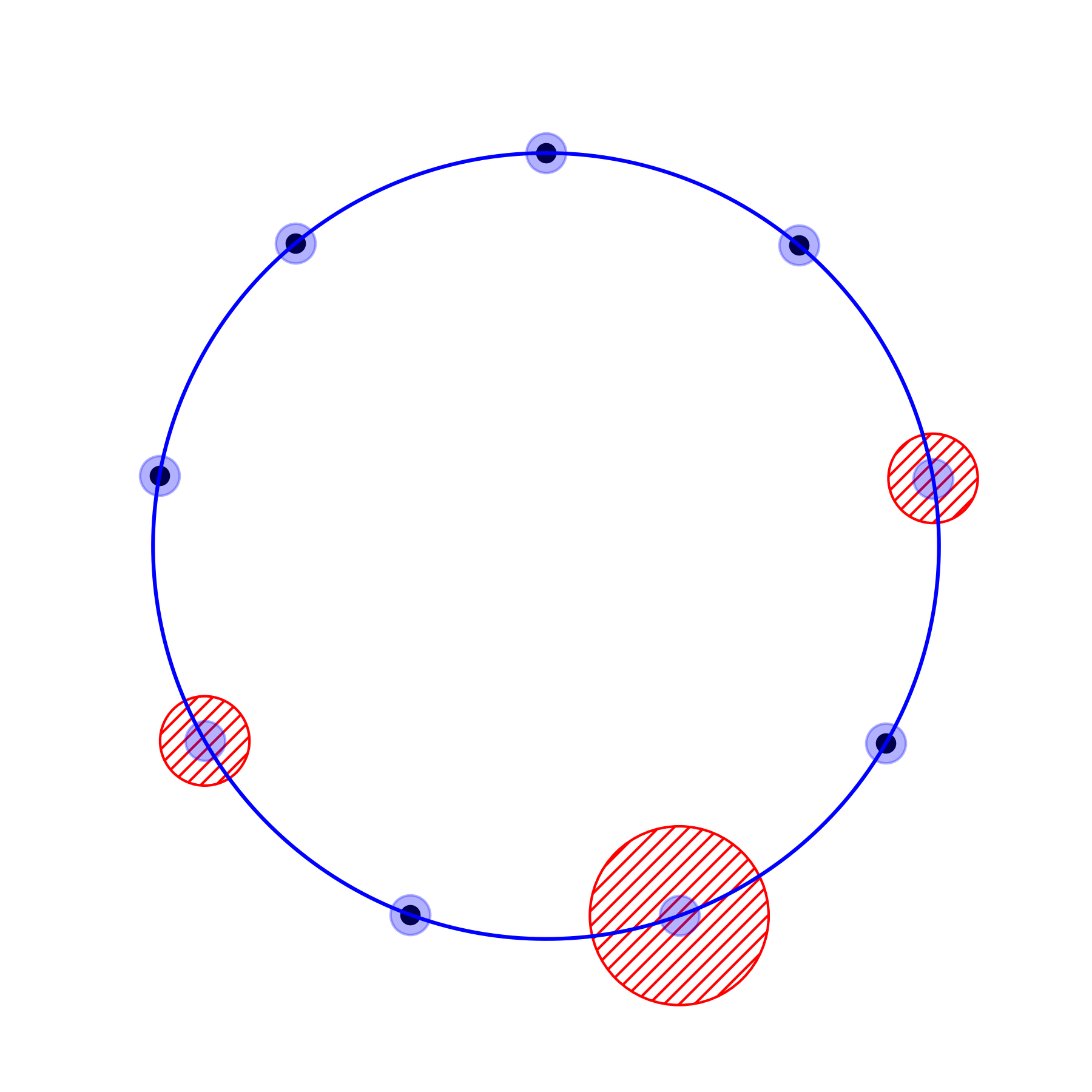}
    \end{minipage}
	&
      \begin{minipage}{.25\textwidth}
      \includegraphics[width=\linewidth]{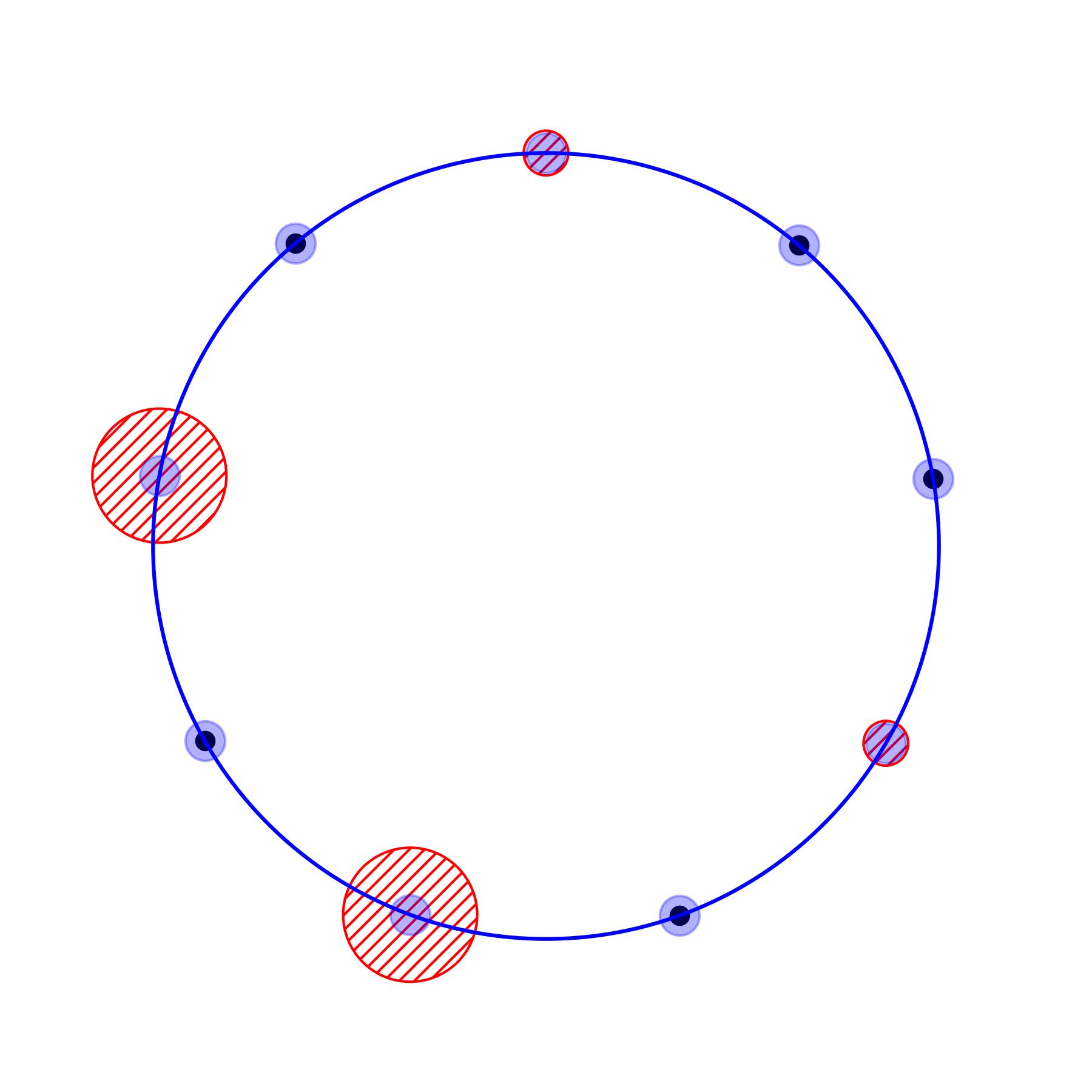}
    \end{minipage}
        \\ \hline
       $n=10$ & $n=20$ & $n=30$ \\ \hline 
    \begin{minipage}{.25\textwidth}
      \includegraphics[width=\linewidth]{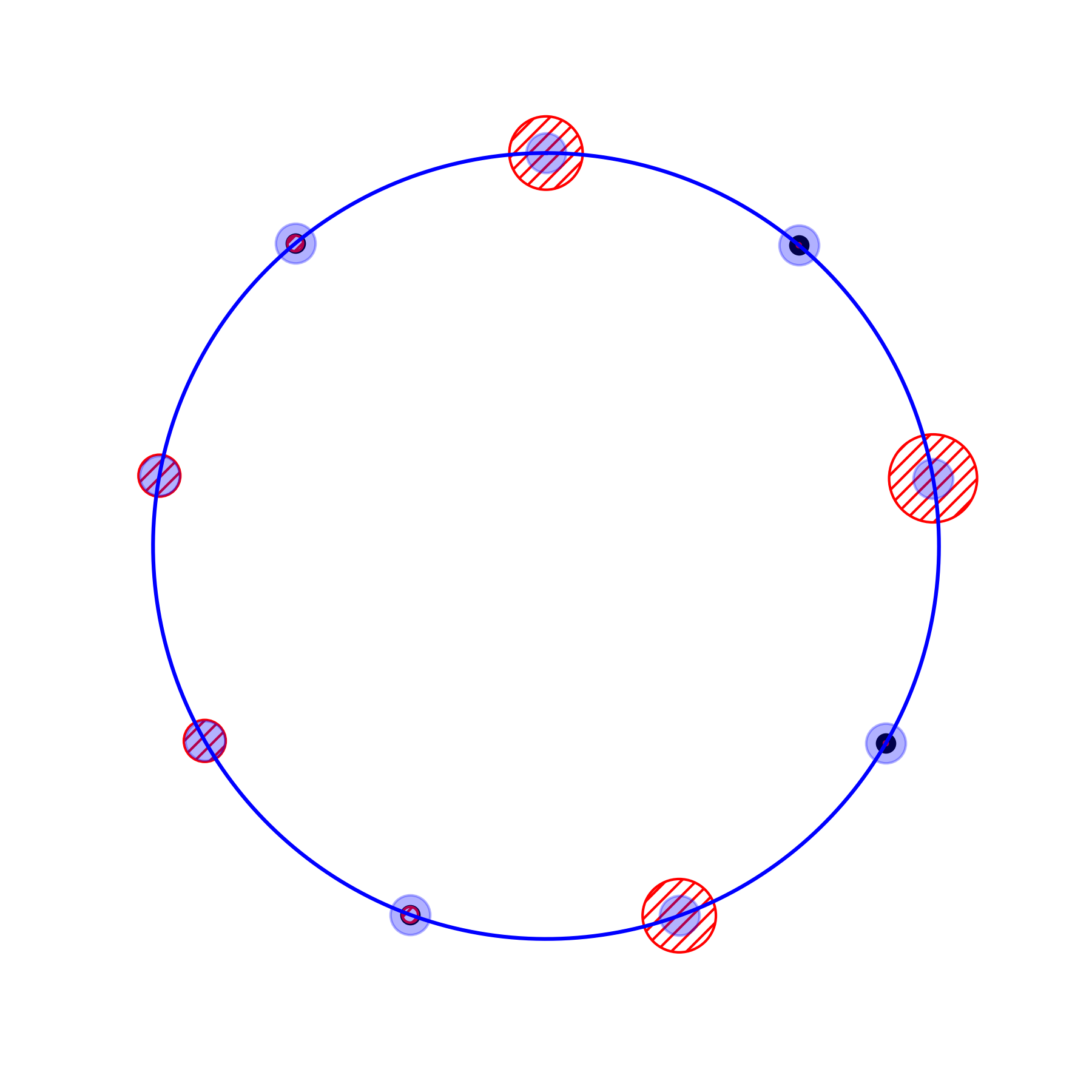}
    \end{minipage}
	&
    \begin{minipage}{.25\textwidth}
      \includegraphics[width=\linewidth]{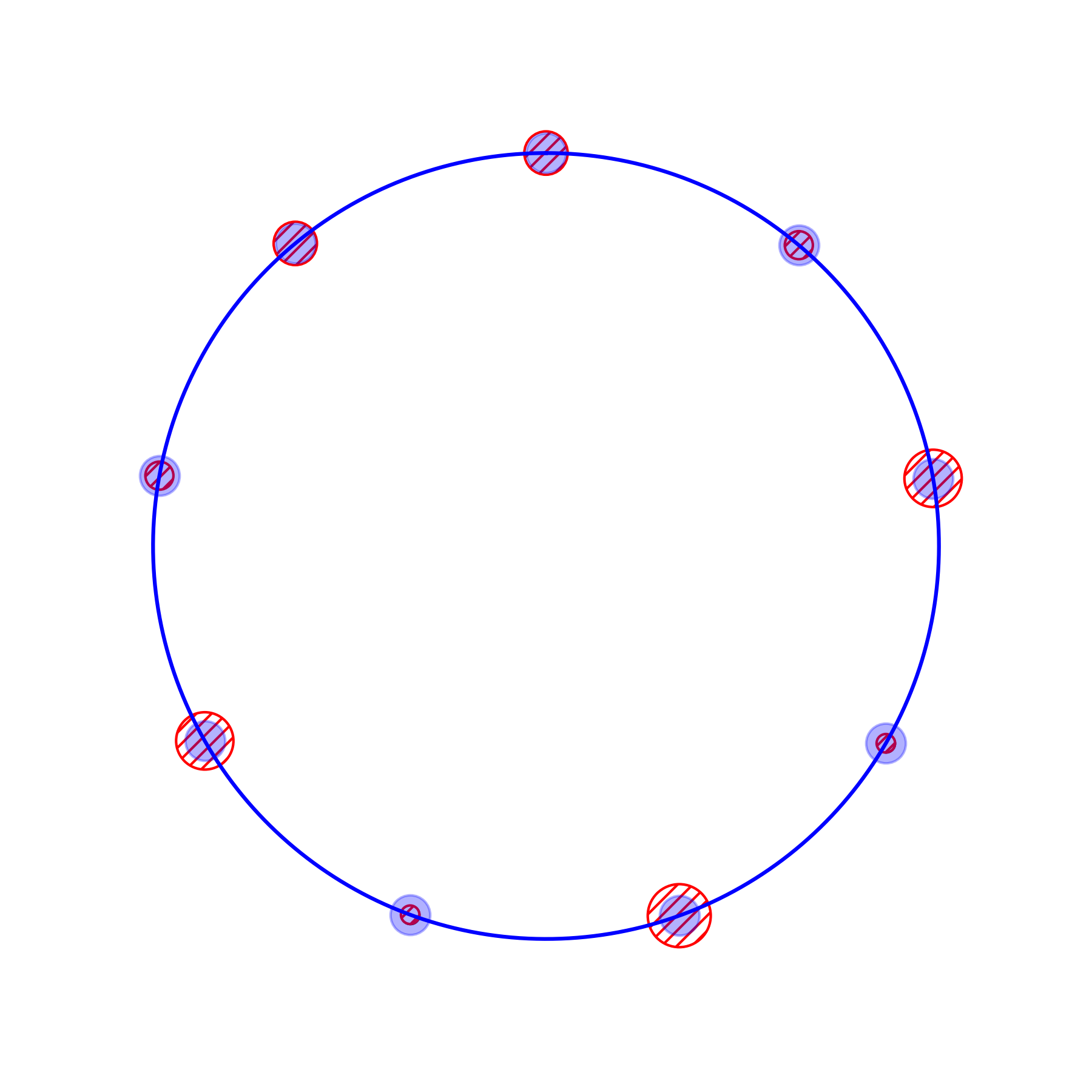}
    \end{minipage}
	&
      \begin{minipage}{.25\textwidth}
      \includegraphics[width=\linewidth]{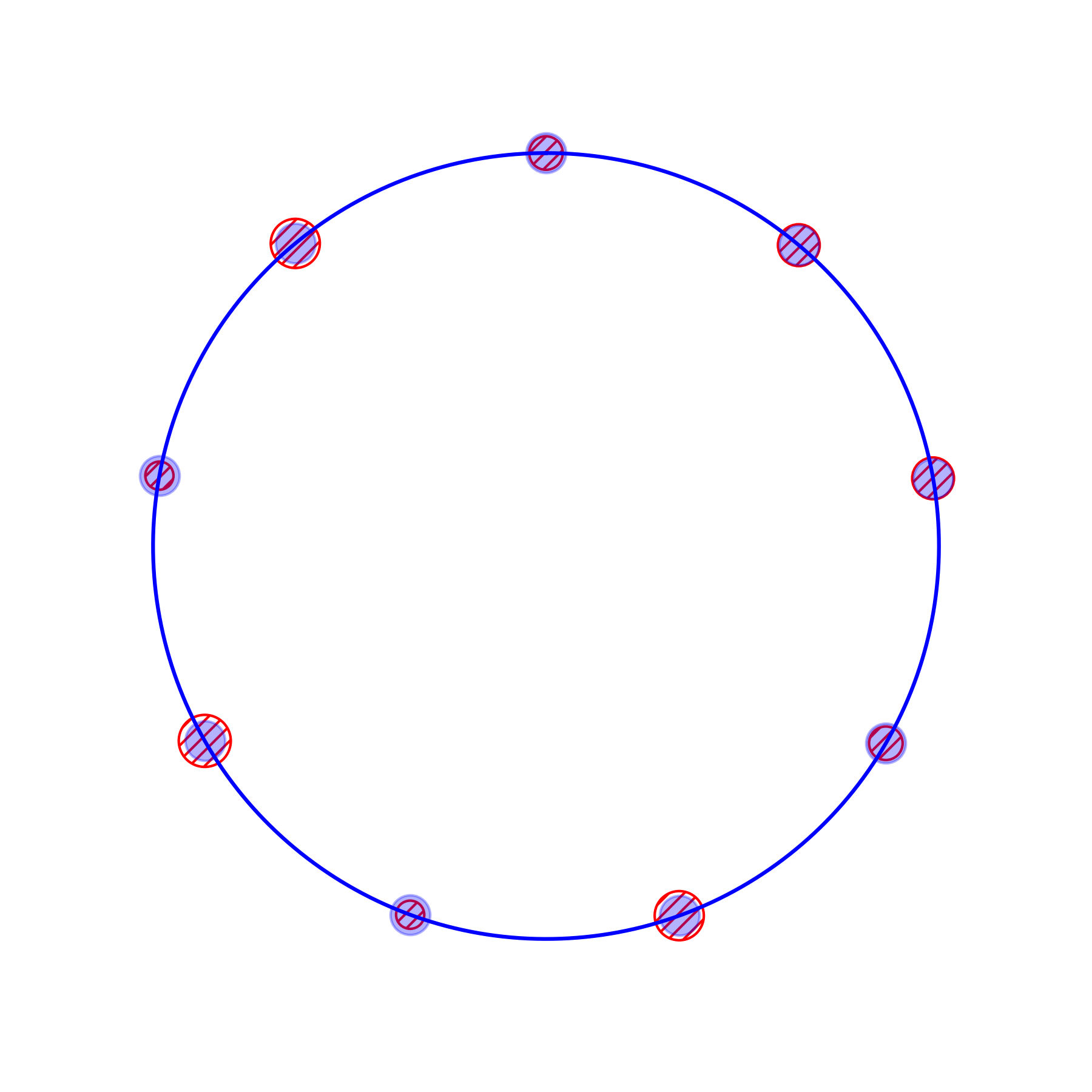}
    \end{minipage}
    \\ \hline
  \end{tabular}
  \captionof{figure}{An illustration of the walk on $\mathbb{Z}_9$ driven by \eqref{eq:Example1} where $a=1$ and $b=3$. The red (hashed) disks indicate the probabilities $p^{(n)}(x)$ and the blue disks indicate the values of the attractor $\Theta_p(n,x)/9=1/9$ for $x\in\mathbb{Z}_9$ and $n\in\{1,2,3,10,20,30\}$. The area of the disks are proportional to the represented values.}\label{fig:Example1.0}
  \end{table}
\item Consider $p$ defined by \eqref{eq:Example1} where $a=1$ and $b=4$. Here, $\gcd(a-b,9)=3$ and so we follow the second case. Here, $\Theta_p(n,x)=3\mathds{1}_{3\mathbb{Z}_9}(x-n)$, $G_p=3\mathbb{Z}_9$, and it is readily calculated that 
\begin{equation*}\rho=\max\{\abs{\widehat{p}(k)}:k\in\mathbb{Z}_9\setminus \Omega(p)\}=1/2.
\end{equation*}
In looking at $\Theta_p(n,x)=3\mathds{1}_{3\mathbb{Z}_9}(x-n)$, we easily see that our random walk irreducible but not aperiodic (it is periodic of period $s=3$). By virtue of Theorem \ref{thm:MainLLT} (or Proposition \ref{prop:IntroLLTinIndicator}), we have
\begin{equation*}
p^{(n)}(x)=\frac{\mathds{1}_{3\mathbb{Z}_9}(x-n)}{3}+O(2^{-n})
\end{equation*}
uniformly for $x\in\mathbb{Z}_9$ as $n\to\infty$. Further, an application of Theorem \ref{thm:ConvergenceToUniform} gives
\begin{equation*}
\|p^{(n)}-\tau_{n}(U_{3\mathbb{Z}_9})\|\leq\frac{\abs{3\mathbb{Z}_9}-1}{2}2^{-n}=2^{-n}
\end{equation*}
for $n\in\mathbb{N}$. These limits are illustrated in Figure \ref{fig:Example1.1} in which we see that the random walk dances around $\mathbb{Z}_9$ while spreading its mass evenly across cosets of the subgroup $3\mathbb{Z}_9$. 

\begin{table}[!h]
  \centering
  \begin{tabular}{  | c | c | c | }
    \hline
    $n=1$ & $n=2$ & $n=3$ \\ \hline
    
    \begin{minipage}{.25\textwidth}
      \includegraphics[width=\linewidth]{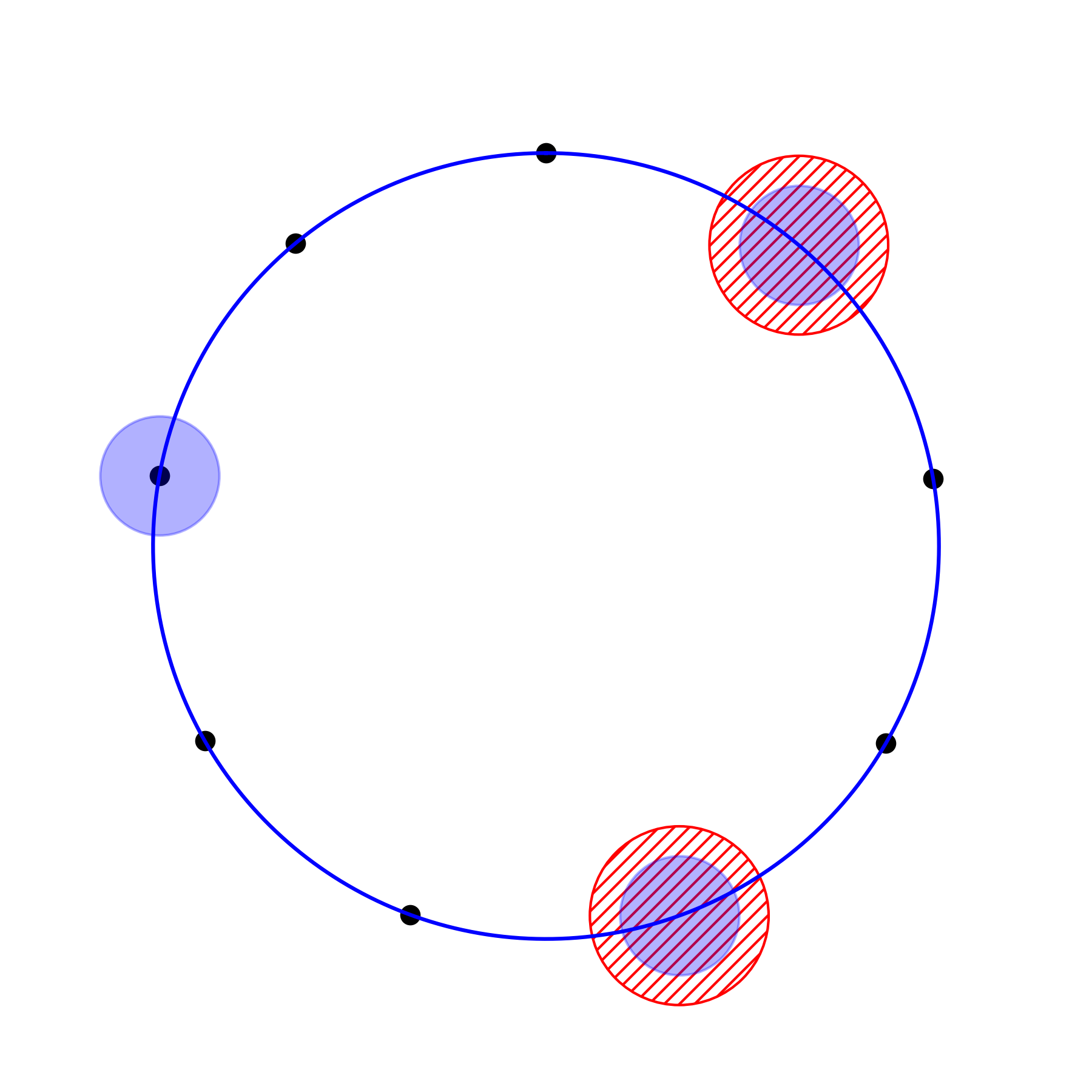}
    \end{minipage}
    &
    \begin{minipage}{.25\textwidth}
      \includegraphics[width=\linewidth]{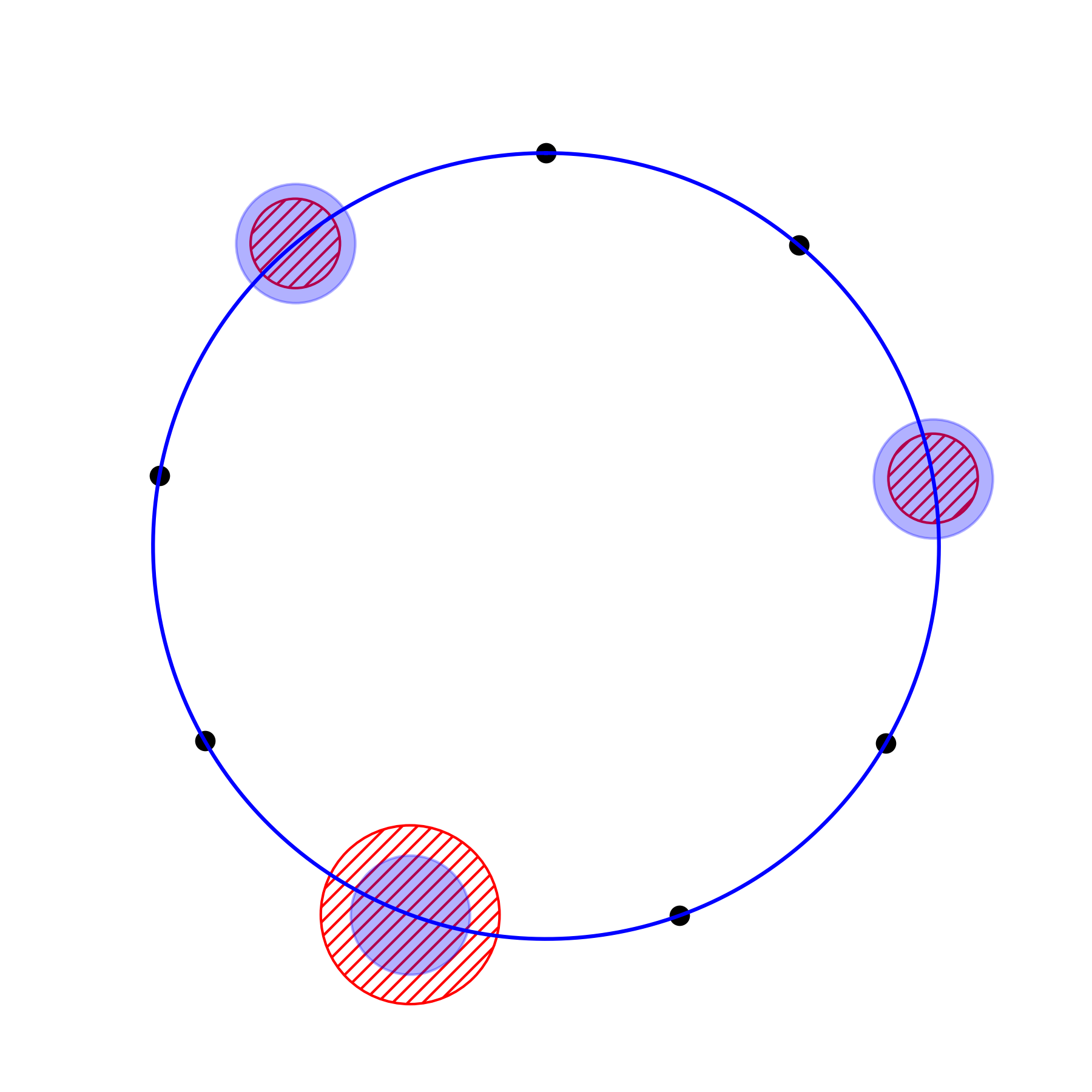}
    \end{minipage}
	&
      \begin{minipage}{.25\textwidth}
      \includegraphics[width=\linewidth]{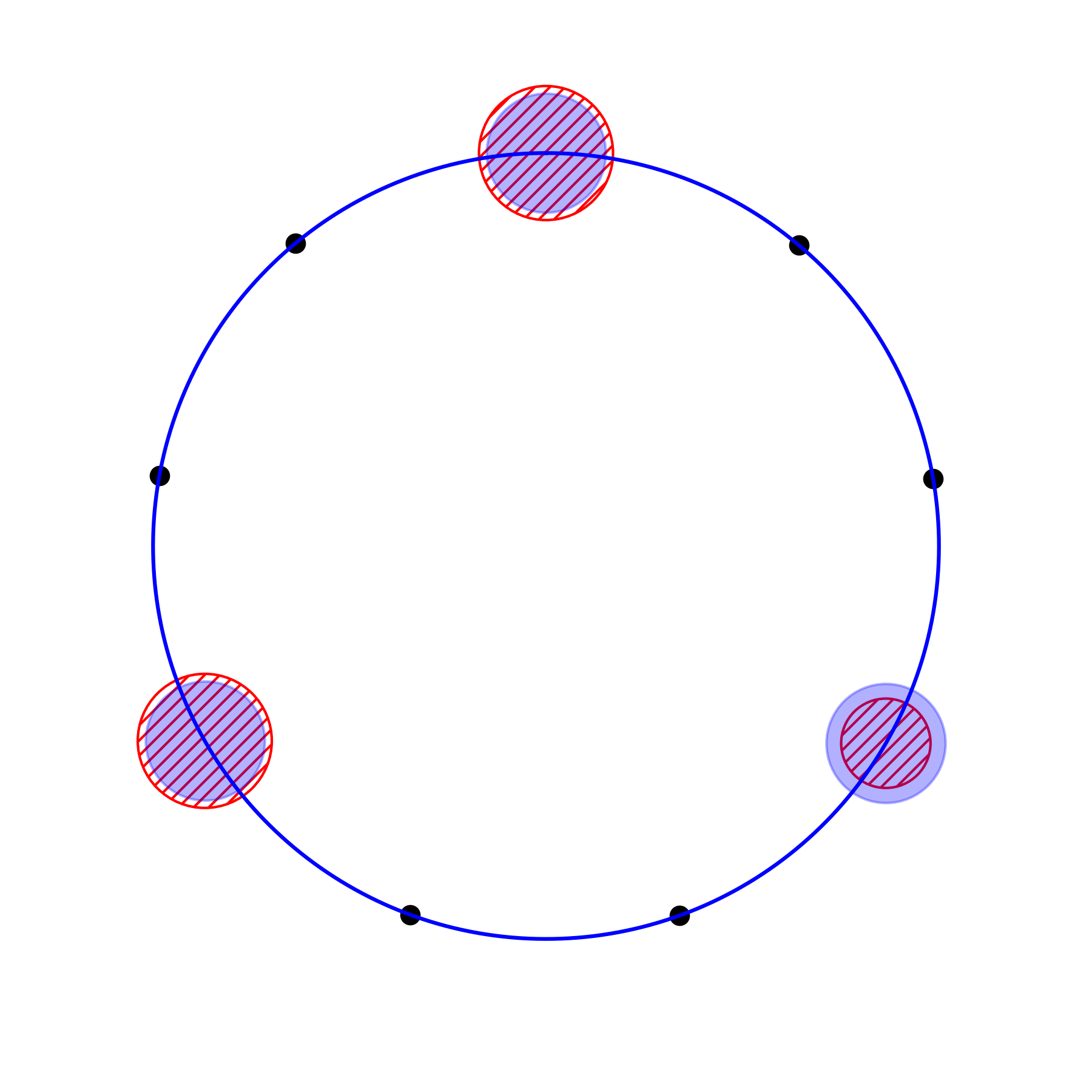}
    \end{minipage}
        \\ \hline
       $n=4$ & $n=5$ & $n=6$ \\ \hline 
    \begin{minipage}{.25\textwidth}
      \includegraphics[width=\linewidth]{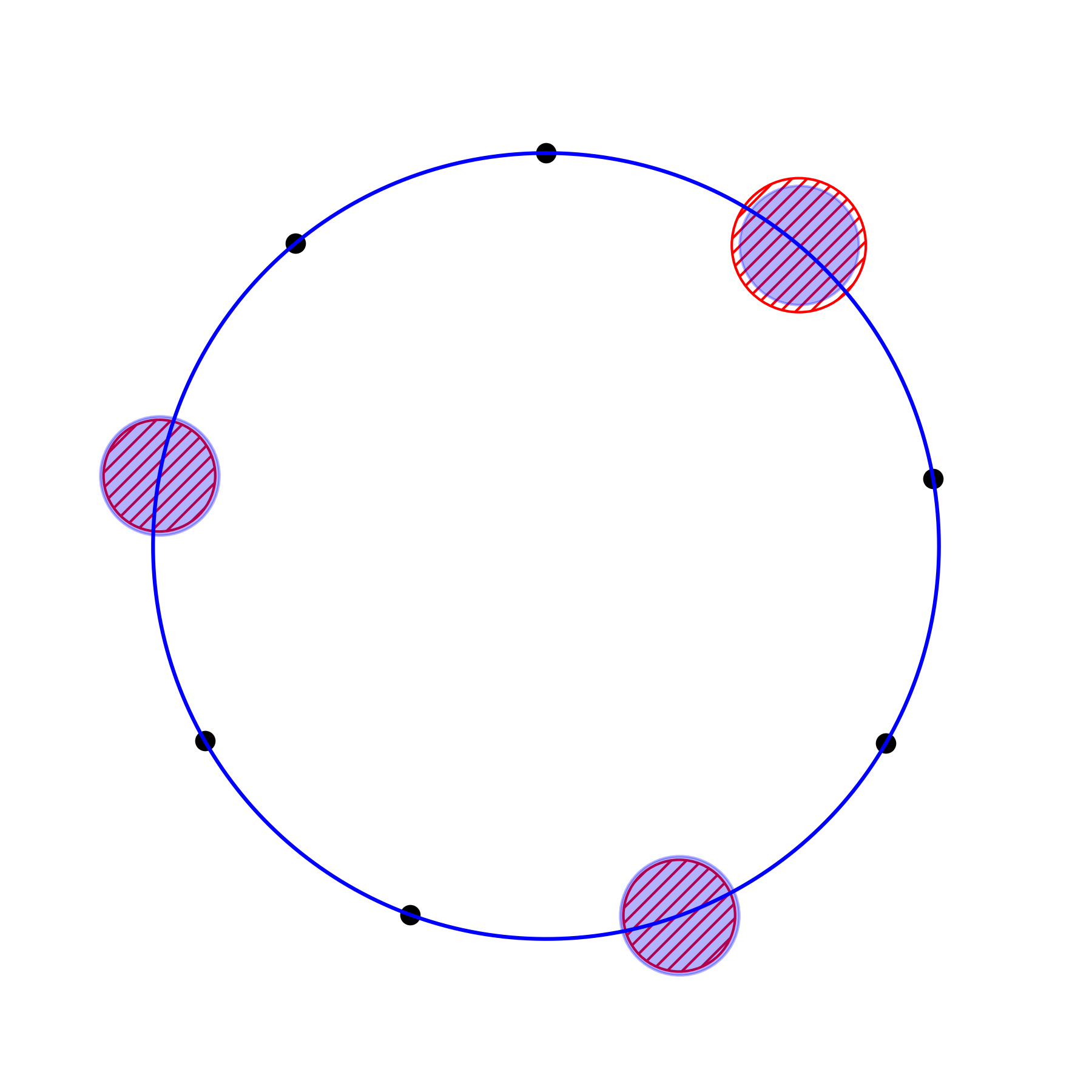}
    \end{minipage}
	&
    \begin{minipage}{.25\textwidth}
      \includegraphics[width=\linewidth]{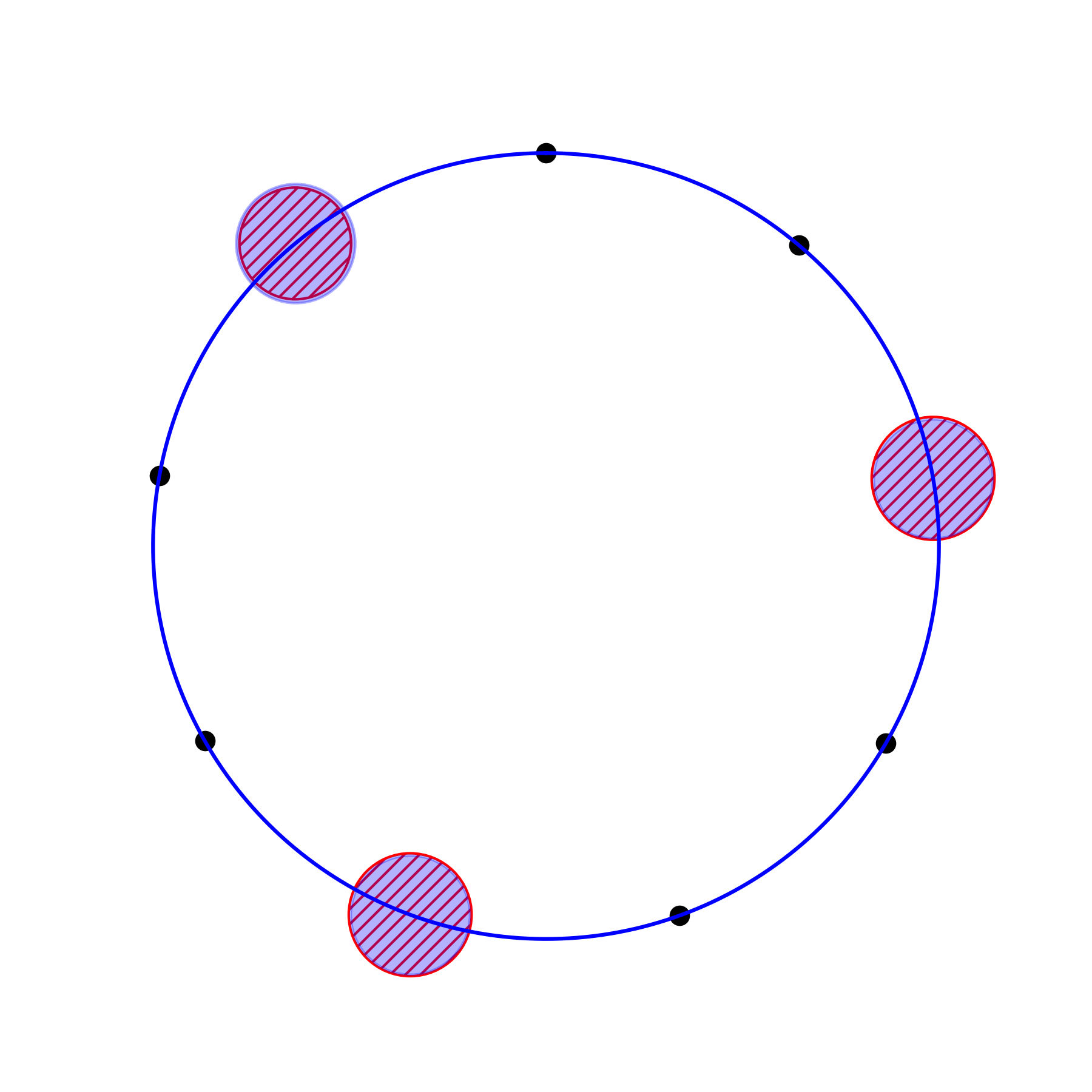}
    \end{minipage}
	&
      \begin{minipage}{.25\textwidth}
      \includegraphics[width=\linewidth]{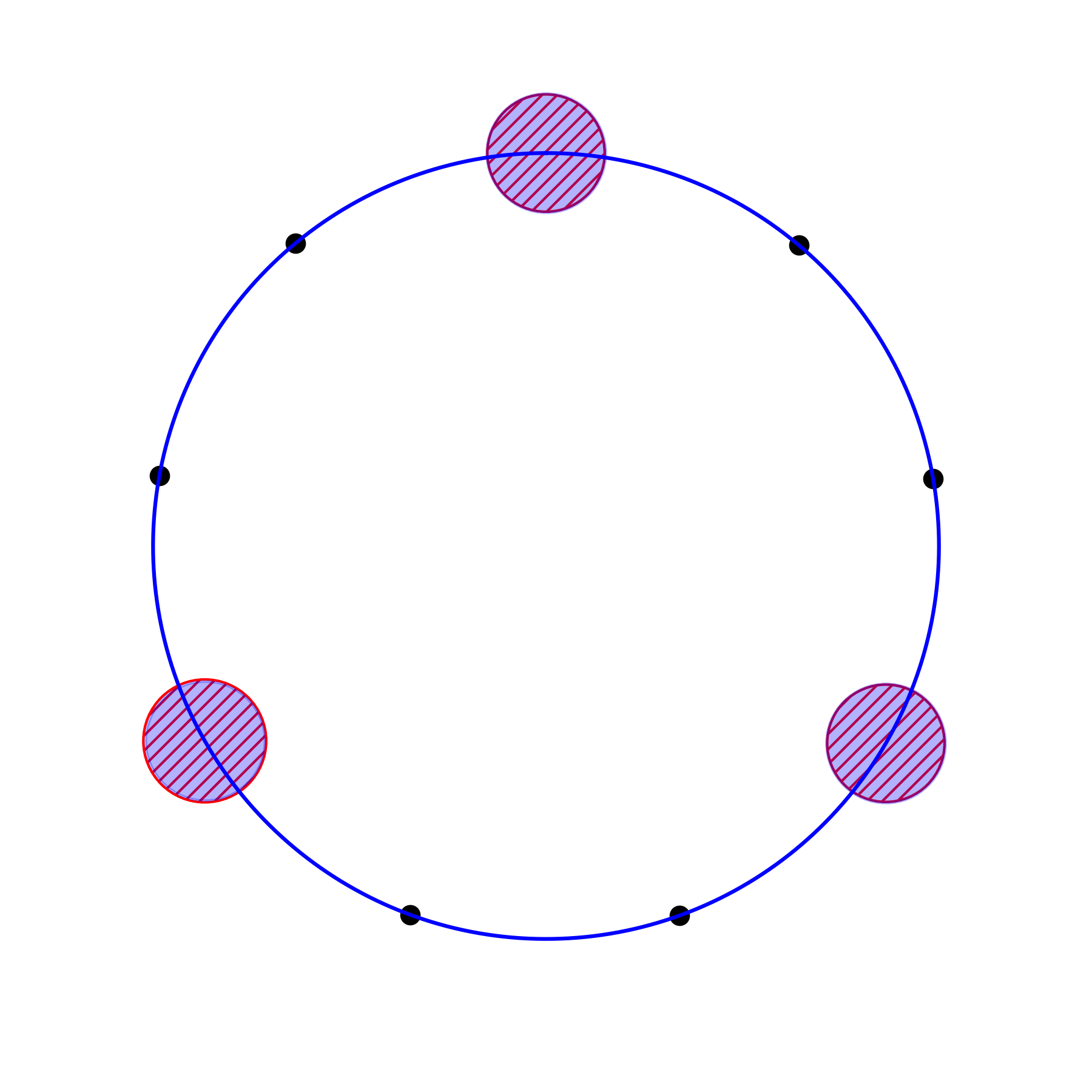}
    \end{minipage}
    \\ \hline
  \end{tabular}
  \captionof{figure}{An illustration of the walk on $\mathbb{Z}_9$ driven by \eqref{eq:Example1} where $a=1$ and $b=4$. The red (hashed) disks indicate the probabilities $p^{(n)}(x)$ and the blue disks indicate the values of the attractor $\Theta_p(n,x)/9=\mathds{1}_{3\mathbb{Z}_9}(x-an)$ for $x\in\mathbb{Z}_9$ and $n=\{1,2,3,4,5,6\}$. The area of the disks are proportional to the represented values.}\label{fig:Example1.1}
\end{table}

\item Given $p$ defined by \eqref{eq:Example1} where $a=0$ and $b=3$. Here, $\Theta_p(n,x)=3\mathds{1}_{3\mathbb{Z}_9}(x-(0)n)=3\mathds{1}_{3\mathbb{Z}_9}(x)$ and $\rho=1/2$. In this case, our random walk is not irreducible (as guaranteed by Proposition \ref{prop:ThetaCapturesSupport}, it never reaches $x=2$, for example). Here, Theorem \ref{thm:MainLLT} gives
\begin{equation*}
p^{(n)}(x)=\frac{\mathds{1}_{3\mathbb{Z}_9}(x)}{3}+O(2^{-n})
\end{equation*}
uniformly for $x\in\mathbb{Z}_9$ as $n\to\infty$. Thanks to Theorem \ref{thm:ConvergenceToUniform}, we also obtain the total-variation norm estimate
\begin{equation*}
\|p^{(n)}-U_{3\mathbb{Z}_9}\|\leq 2^{-n}
\end{equation*}
for $n\in\mathbb{N}$. These asymptotics are illustrated in Figure \ref{fig:Example1.2} where we see, in contrast to the previous case, there is no ``dance".

\begin{table}[!h]
  \centering
  \begin{tabular}{  | c | c | c | }
    \hline
    $n=1$ & $n=2$ & $n=3$ \\ \hline
    
    \begin{minipage}{.3\textwidth}
      \includegraphics[width=\linewidth]{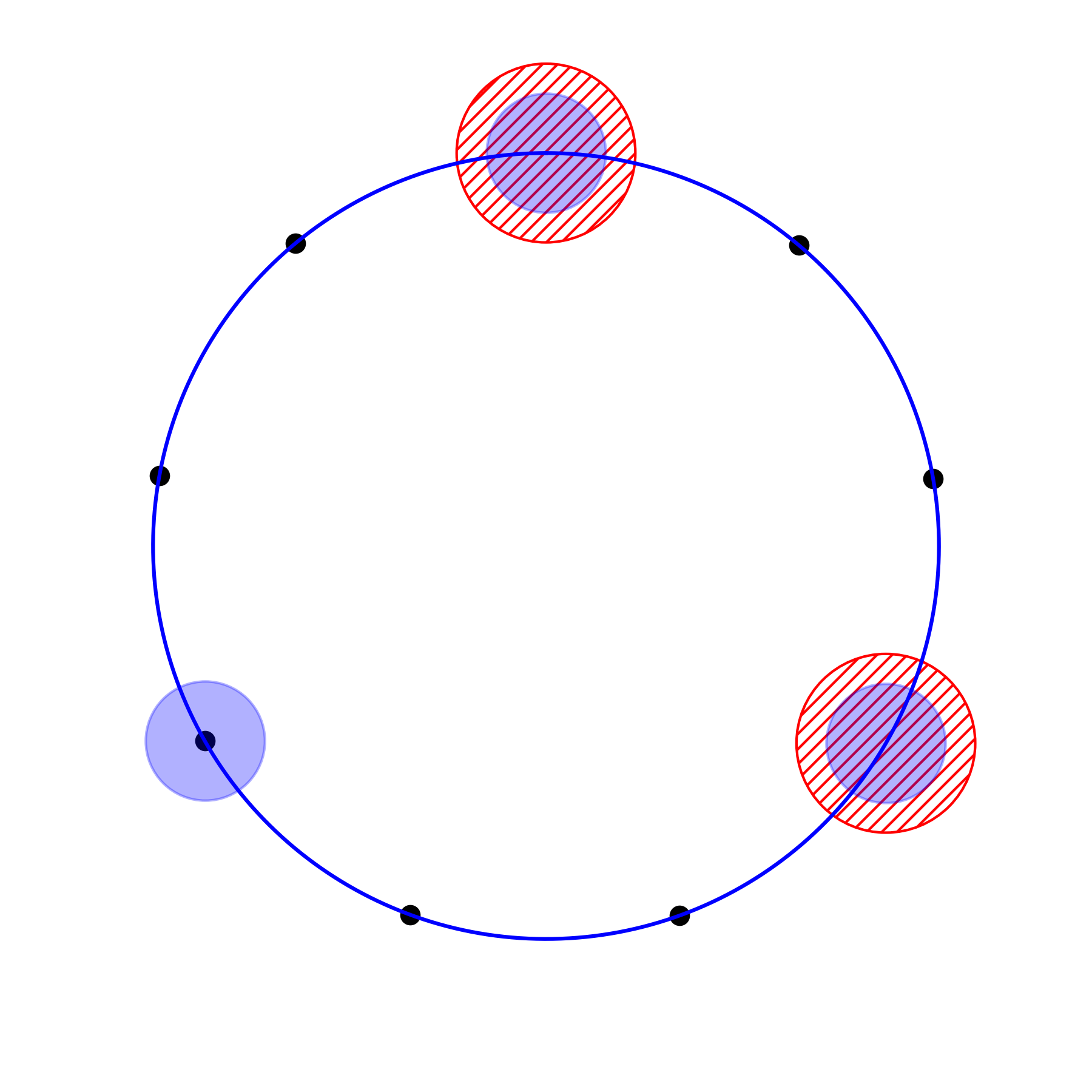}
    \end{minipage}
    &
    \begin{minipage}{.25\textwidth}
      \includegraphics[width=\linewidth]{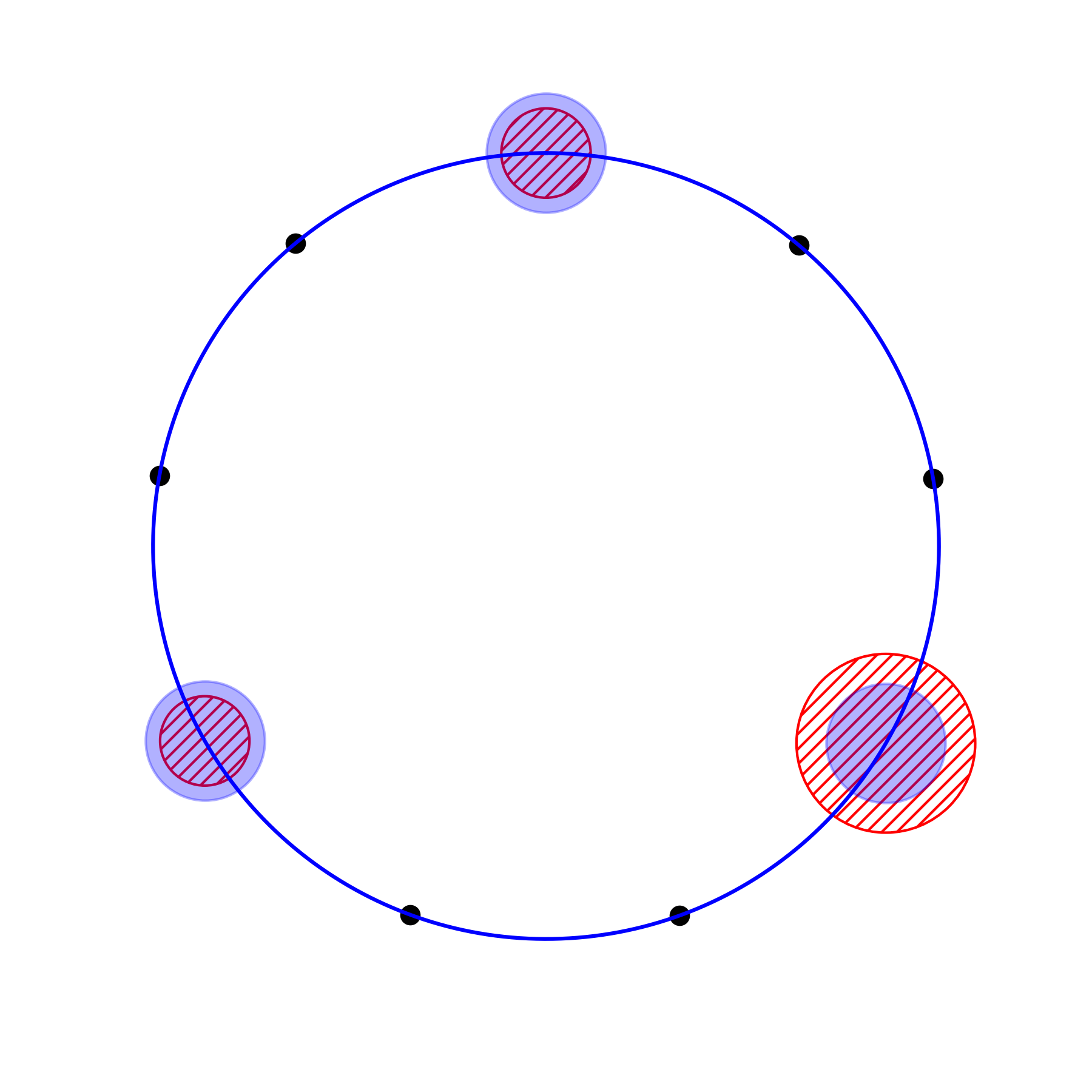}
    \end{minipage}
	&
      \begin{minipage}{.25\textwidth}
      \includegraphics[width=\linewidth]{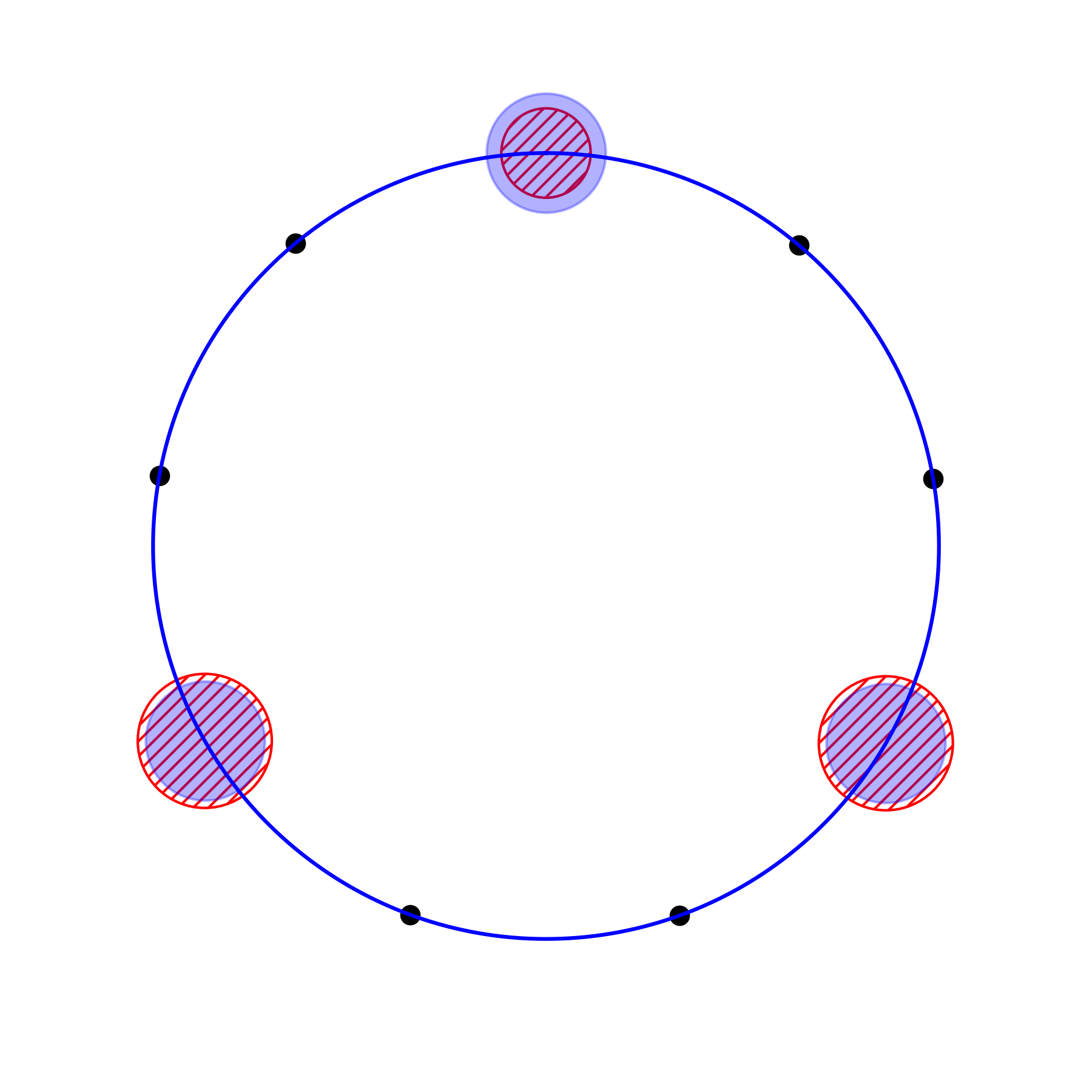}
    \end{minipage}
        \\ \hline
       $n=4$ & $n=5$ & $n=6$ \\ \hline 
    \begin{minipage}{.25\textwidth}
      \includegraphics[width=\linewidth]{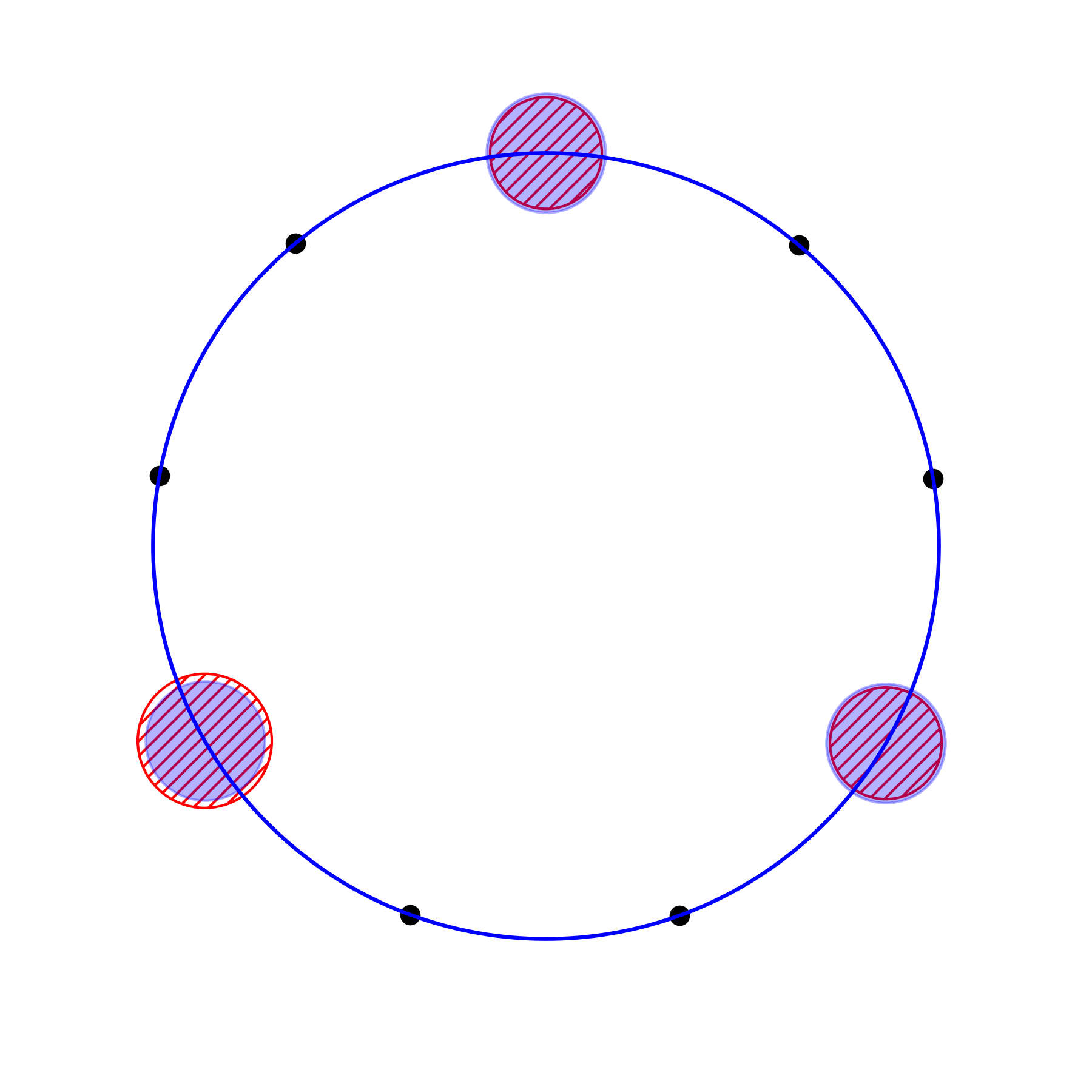}
    \end{minipage}
	&
    \begin{minipage}{.25\textwidth}
      \includegraphics[width=\linewidth]{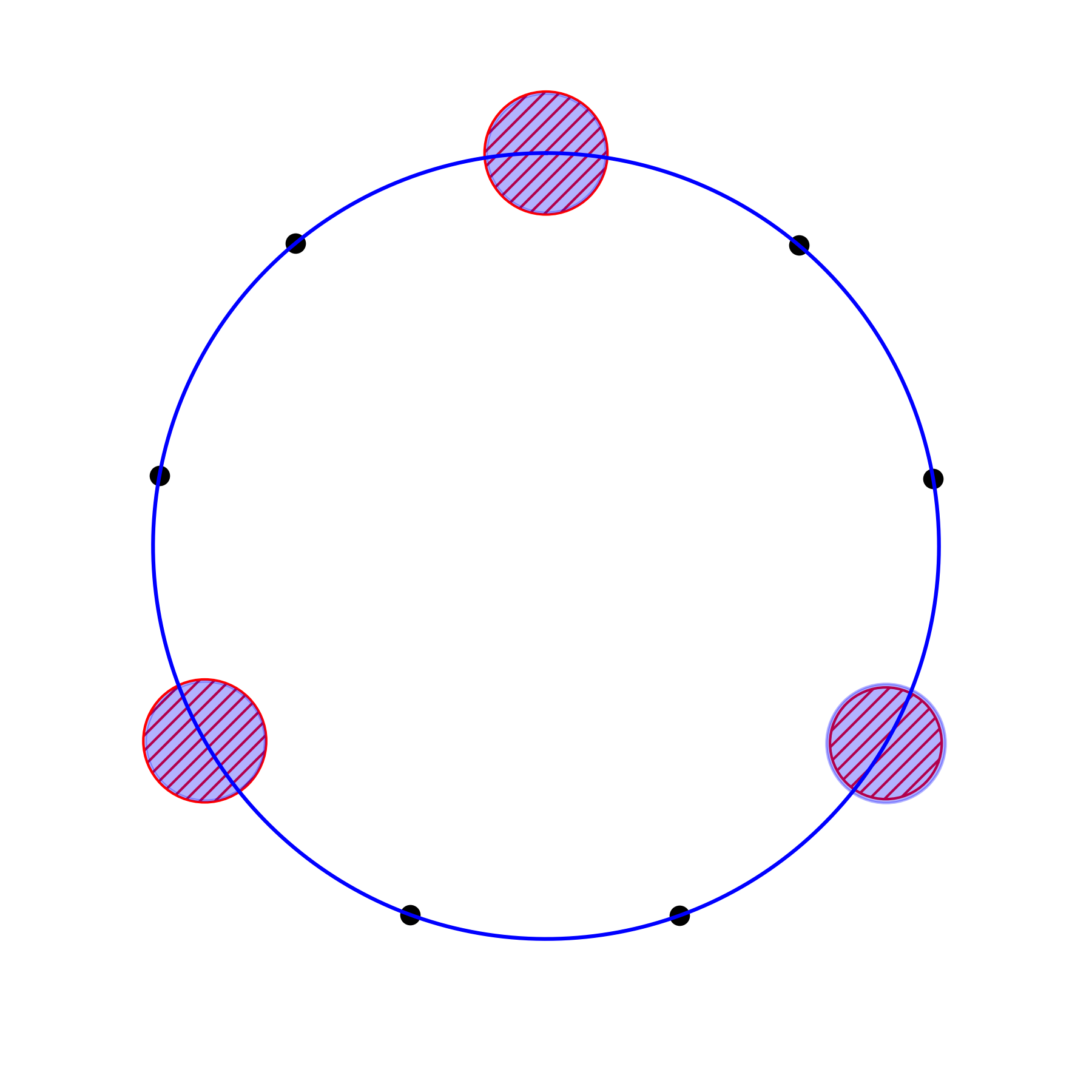}
    \end{minipage}
	&
      \begin{minipage}{.25\textwidth}
      \includegraphics[width=\linewidth]{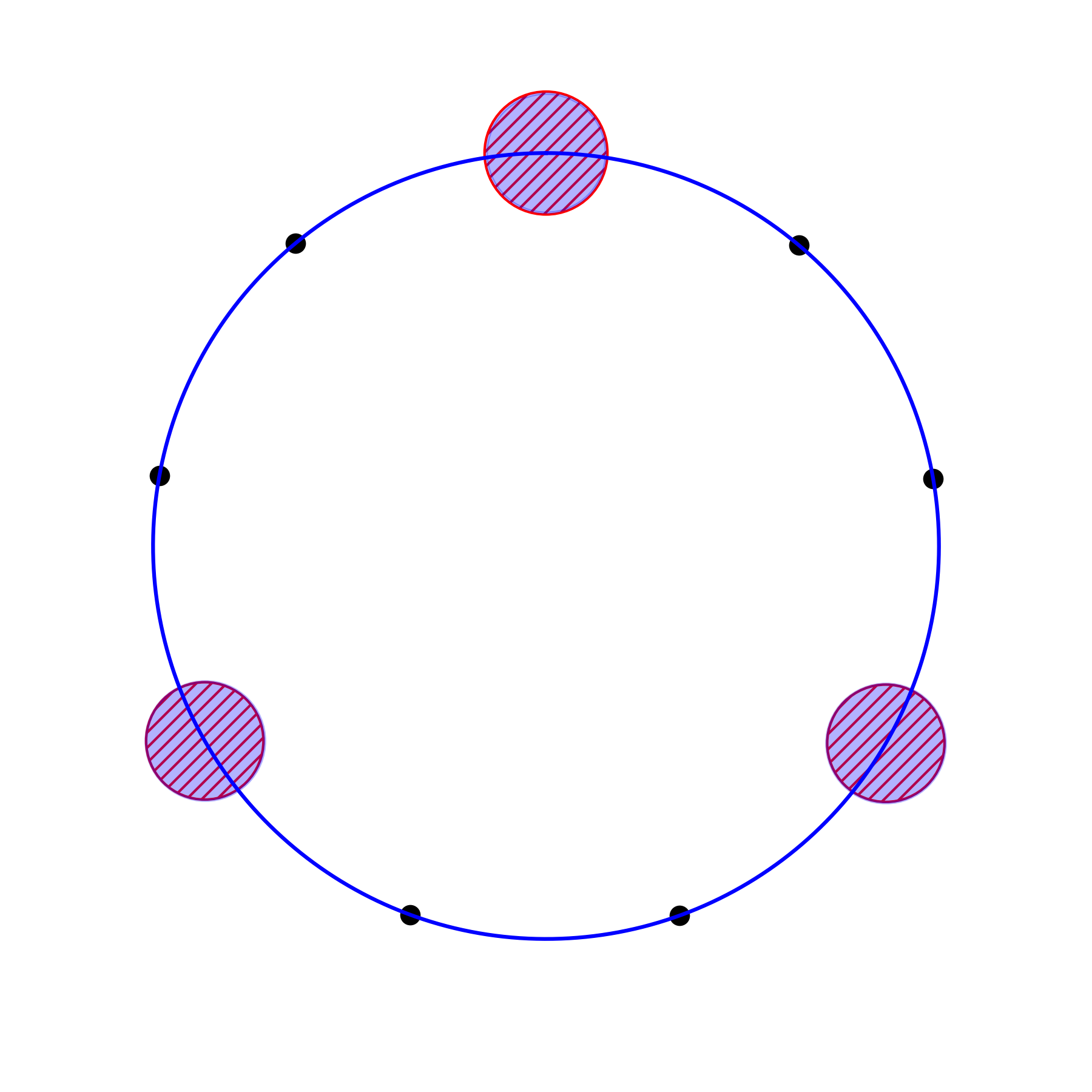}
    \end{minipage}
    \\ \hline
  \end{tabular}
  \captionof{figure}{An illustration of the walk on $\mathbb{Z}_9$ driven by \eqref{eq:Example1} where $a=0$ and $b=3$. The red (hashed) disks indicate the probabilities $p^{(n)}(x)$ and the blue disks indicate the value of the attractor $\Theta_p(n,x)/9=\mathds{1}_{3\mathbb{Z}_3}(x)$ for $x\in\mathbb{Z}_9$ and $n=\{1,2,3,4,5,6\}$. The area of the disks are proportional to the represented values.}\label{fig:Example1.2}
  \end{table}
\end{enumerate}
\end{example}

\begin{example}[A Random Walk on $\mathbb{Z}_4\times\mathbb{Z}_6$ supported on two points]

Given $0<q_1,q_2<1$ with $q_1+q_2=1$, $a_1,a_2\in\mathbb{Z}_4$ and $b_1,b_2\in\mathbb{Z}_6$, consider the random walk on $\mathbb{Z}_4\times\mathbb{Z}_6$ driven by $p\in\mathcal{M}(\mathbb{Z}_4\times\mathbb{Z}_6)$ defined by
\begin{equation}\label{eq:DiscreteTorusGen}
p(x,y)=\begin{cases}
q_1 & \mbox{if}\quad (x,y)=(a_1,b_1)\\
q_2 & \mbox{if}\quad (x,y)=(a_2,b_2)\\
0 & \mbox{else}
\end{cases}.
\end{equation}
For $(\eta,\zeta)\in\widehat{\mathbb{Z}_4\times\mathbb{Z}_6}=\mathbb{Z}_4\times\mathbb{Z}_6$, we have
\begin{equation*}
\widehat{p}(\eta,\zeta)=q_1 e^{2\pi i a_1\eta/4}e^{2\pi i b_1\zeta/6}+q_2 e^{2\pi i a_2\eta/4}e^{2\pi i b_2\zeta/6}.
\end{equation*}
It is not hard to see that $(\eta,\zeta)\in\Omega(p)$ if and only if 
\begin{equation*}
\frac{(a_1-a_2)\eta}{4}=\frac{-(b_1-b_2)\zeta}{6} \mod \mathbb{Z}
\end{equation*}
and from this we can easily describe all possibilities of $\Omega(p)$. These are listed in Table \ref{tab:DiscreteTorus} according to the values of $A:=a_1-a_2\in\mathbb{Z}_4$ and $B:=b_1-b_2\in\mathbb{Z}_6$.

\begin{table}[h!]
\begin{center}
\begin{tabular}{| c | c | c | c | c |}
\hline
$\Omega(p)$ & $A=0$ & $A=1$ & $A=2$ & $A=3$ \\
\hline
$B=0$ & $\mathbb{Z}_4\times\mathbb{Z}_6$ & $\{0\}\times \mathbb{Z}_6$ & $2\mathbb{Z}_4\times \mathbb{Z}_6$ & $\{0\}\times\mathbb{Z}_6$\\
\hline
$B=1$ & $\mathbb{Z}_4\times\{0\}$ & $\langle (2,3)\rangle$ & $\langle (2,0),(1,3)\rangle$ & $\langle (2,3)\rangle$\\
\hline
$B=2$ & $\mathbb{Z}_4\times 3\mathbb{Z}_6$ & $\langle (0,3)\rangle$ & $2\mathbb{Z}_4\times 3\mathbb{Z}_6$ & $\{0\}\times 3\mathbb{Z}_6$ \\
\hline
$B=3$ & $\mathbb{Z}_4\times 2\mathbb{Z}_6$ & $\langle (0,2), (2,1)\rangle$ & $\langle (0,2), (1,1) \rangle$ & $\langle (0,2), (2,1)\rangle$\\
\hline
$B=4$ &  $\mathbb{Z}_4\times 3\mathbb{Z}_6$ & $\{0\}\times 3\mathbb{Z}_6$ & $\langle (0,3), (2,0)\rangle$ & $\langle (0,3)\rangle$ \\
\hline 
$B=5$ & $\mathbb{Z}_4\times\{0\}$ & $\langle (2,3)\rangle$ & $\langle (2,0),(1,3)\rangle$ & $\langle (2,3)\rangle$\\
\hline 
\end{tabular}
\caption{An enumeration of all possibilities of $\Omega(p)$ for $p$ defined by \eqref{eq:DiscreteTorusGen}.}\label{tab:DiscreteTorus}
\end{center}
\end{table}

In looking at Table \ref{tab:DiscreteTorus}, observe that the trivial subgroup $\{(0,0)\}$ never appears as a possibility for $\Omega(p)$ and, in view of Proposition \ref{prop:AperAndIrreduc}, we conclude that the group $\mathbb{Z}_4\times\mathbb{Z}_6$ admits no such (two point-support) random walk that is both irreducible and aperiodic. This should come as no surprise as $\mathbb{Z}_4\times\mathbb{Z}_6$ is not cyclic and therefore $G_p=\langle\supp(p)-(a_2,b_2)\rangle=\langle (A,B)\rangle$ must always be a proper subgroup of $\mathbb{Z}_4\times\mathbb{Z}_6$. By contrast, whenever $\gcd(k,l)=1$, the group $\mathbb{Z}_k\times\mathbb{Z}_l\cong\mathbb{Z}_{kl}$ does admit a probability distribution supported on two points driving an irreducible and aperiodic random walk.

While $\mathbb{Z}_4\times\mathbb{Z}_6$ admits no such (simultaneously) irreducible and aperiodic walk, it does admit several irreducible ones. For example, taking $q_1=q_2=1/2$, $(a_1,b_1)=(1,1)$ and $(a_2,b_2)=(0,3)$, we find that $\Omega(p)=\{0\}\times 3\mathbb{Z}_6$, $\Theta_p(n,x,y)=1+(-1)^{n+y}=2\mathds{1}_{\langle (1,4)\rangle}(x,y-3n)$, and
\begin{equation*}
\rho=\max\{\abs{\widehat{p}(\eta,\zeta)}:(\eta,\zeta)\in\mathbb{Z}_4\times\mathbb{Z}_6\setminus \Omega(p)\}=\frac{\sqrt{2+\sqrt{3}}}{2}<0.966.
\end{equation*}
Thus, by an application of Theorem \ref{thm:MainLLT} (or Proposition \ref{prop:IntroLLT}), we find that
\begin{equation*}
p^{(n)}(x,y)=\frac{\Theta_p(n,x,y)}{24}+O(\rho^n)=\begin{cases} \frac{1}{12} &\mbox{if}\quad n+y\mbox{ is even}\\ 0 & \mbox{else}\end{cases}+ o((0.966)^n)
\end{equation*}
uniformly for $(x,y)\in\mathbb{Z}_4\times\mathbb{Z}_6$ as $n\to\infty$. This local limit theorem is illustrated in Figure \ref{fig:DiscreteTorus}.

\begin{table}[!h]
  \centering
  \begin{tabular}{ | c | c | c | c | }
    \hline
     & $n=5$ & $n=100$ & $n=101$ \\ \hline
     \rotatebox[origin=c]{90}{\,\,$\supp(p^{(n)})$ \,\,}
     &
    \begin{minipage}{.25\textwidth}
      \includegraphics[trim = 4cm 4.5cm 4cm 4.5cm, clip, width=\linewidth]{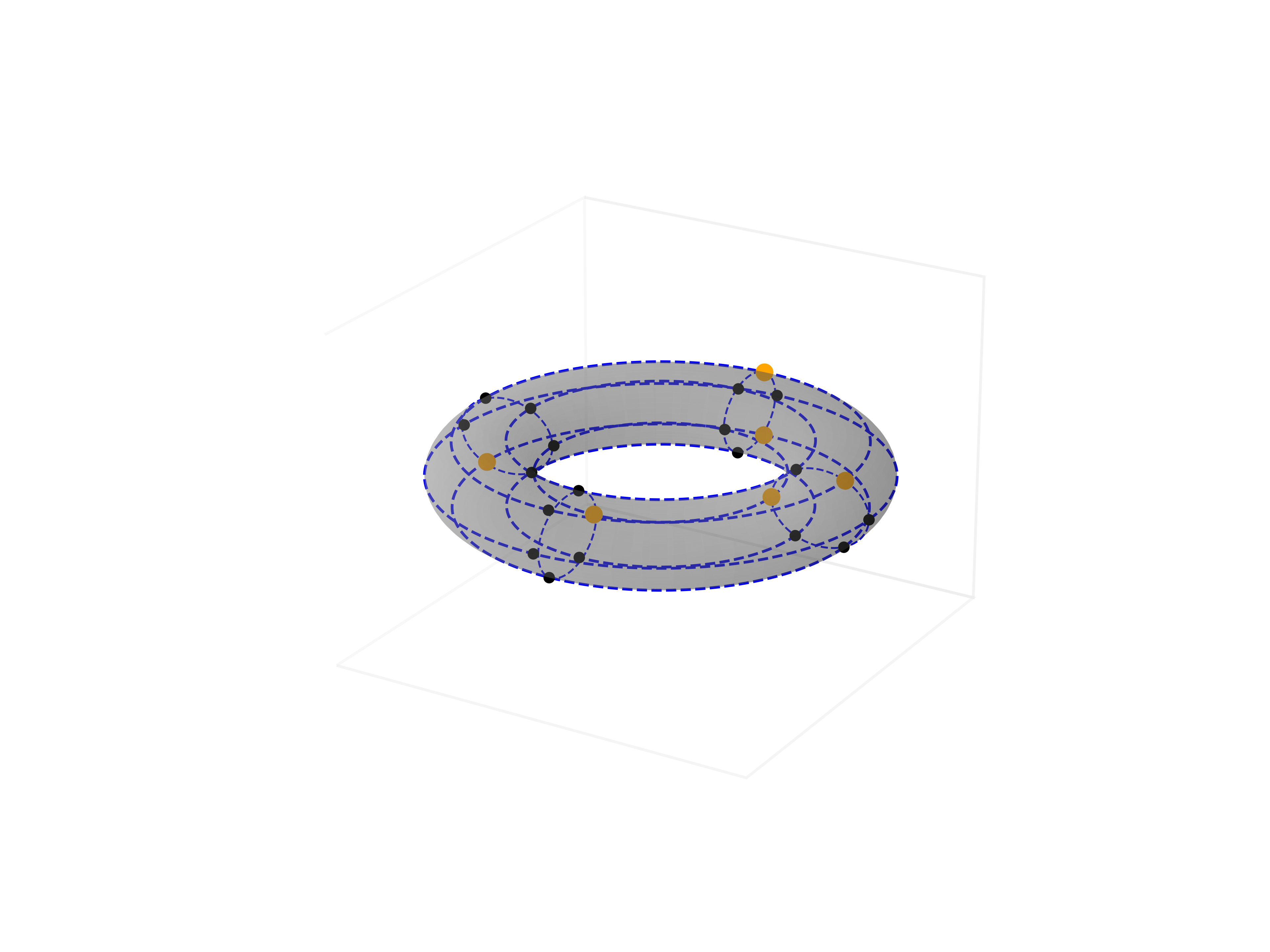}
    \end{minipage}
    &
    \begin{minipage}{.25\textwidth}
      \includegraphics[trim = 4cm 4.5cm 4cm 4.5cm, clip, width=\linewidth]{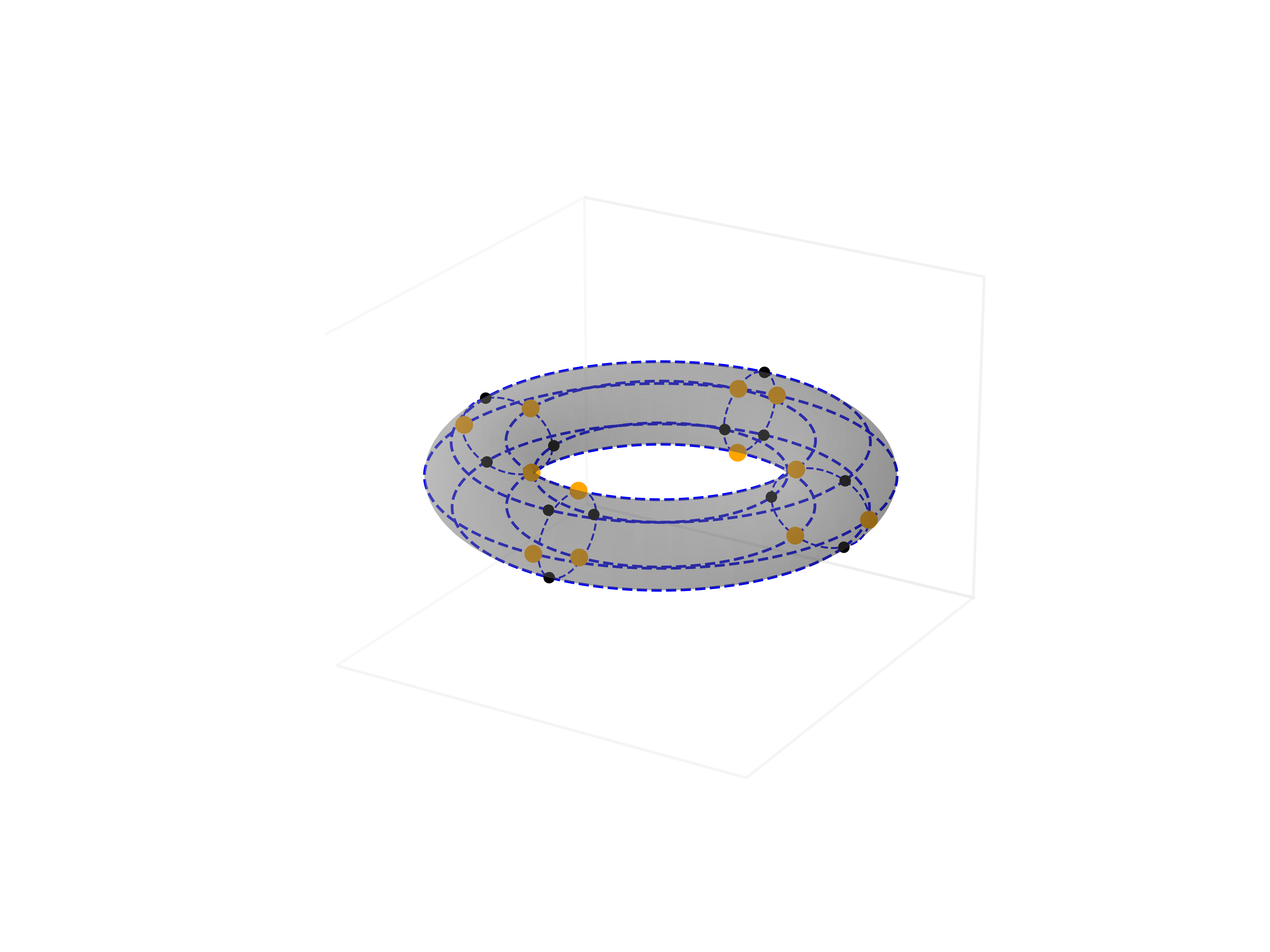}
    \end{minipage}
    & 
      \begin{minipage}{.25\textwidth}
      \includegraphics[trim = 4cm 4.5cm 4cm 4.5cm, clip, width=\linewidth]{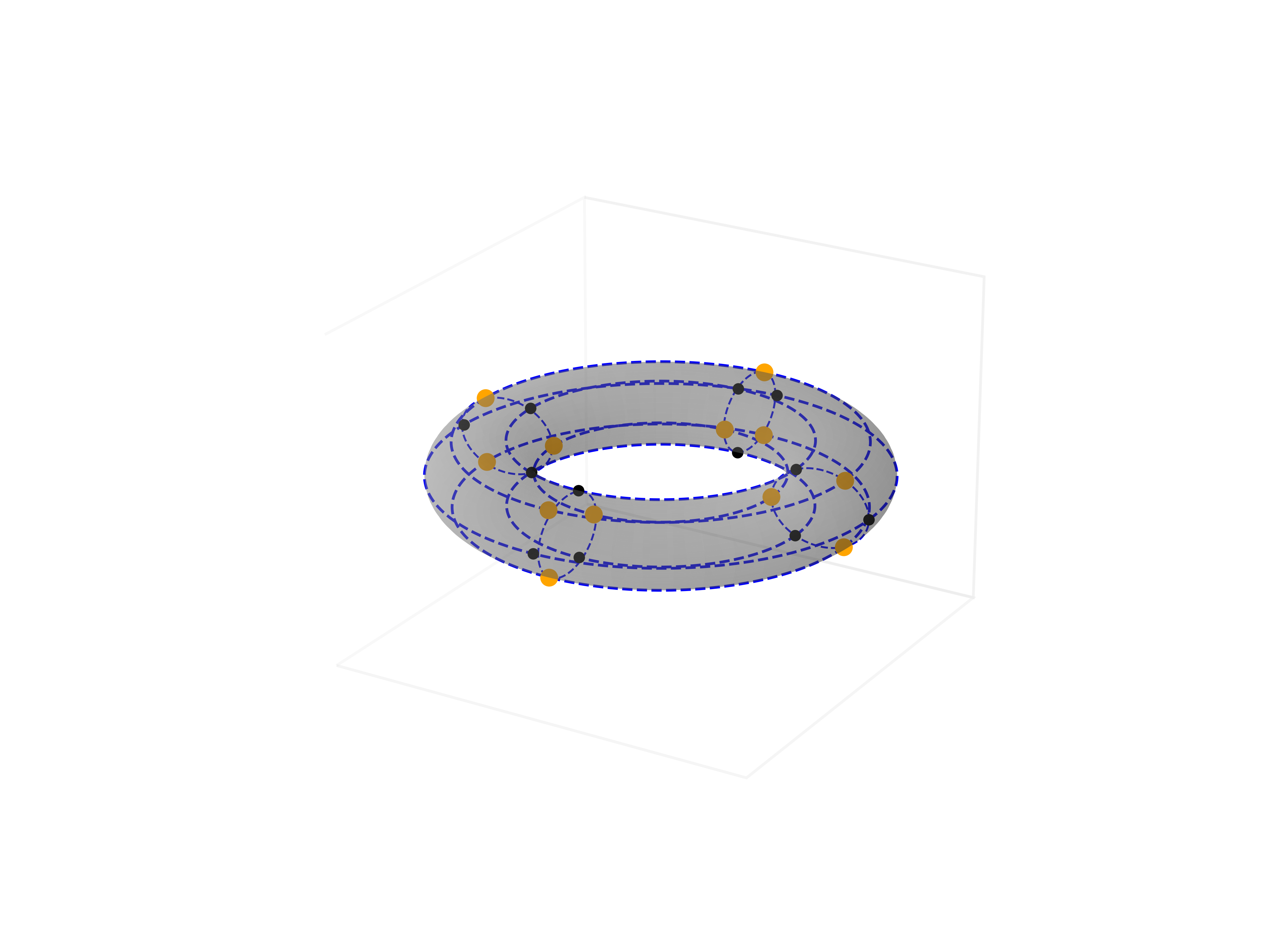}
    \end{minipage}
    \\ \hline
    \rotatebox[origin=c]{90} {\,\,$\supp(\Theta_{p}(n,\cdot))$\,}
     &
    \begin{minipage}{.25\textwidth}
      \includegraphics[trim = 4cm 4.5cm 4cm 4.5cm, clip, width=\linewidth]{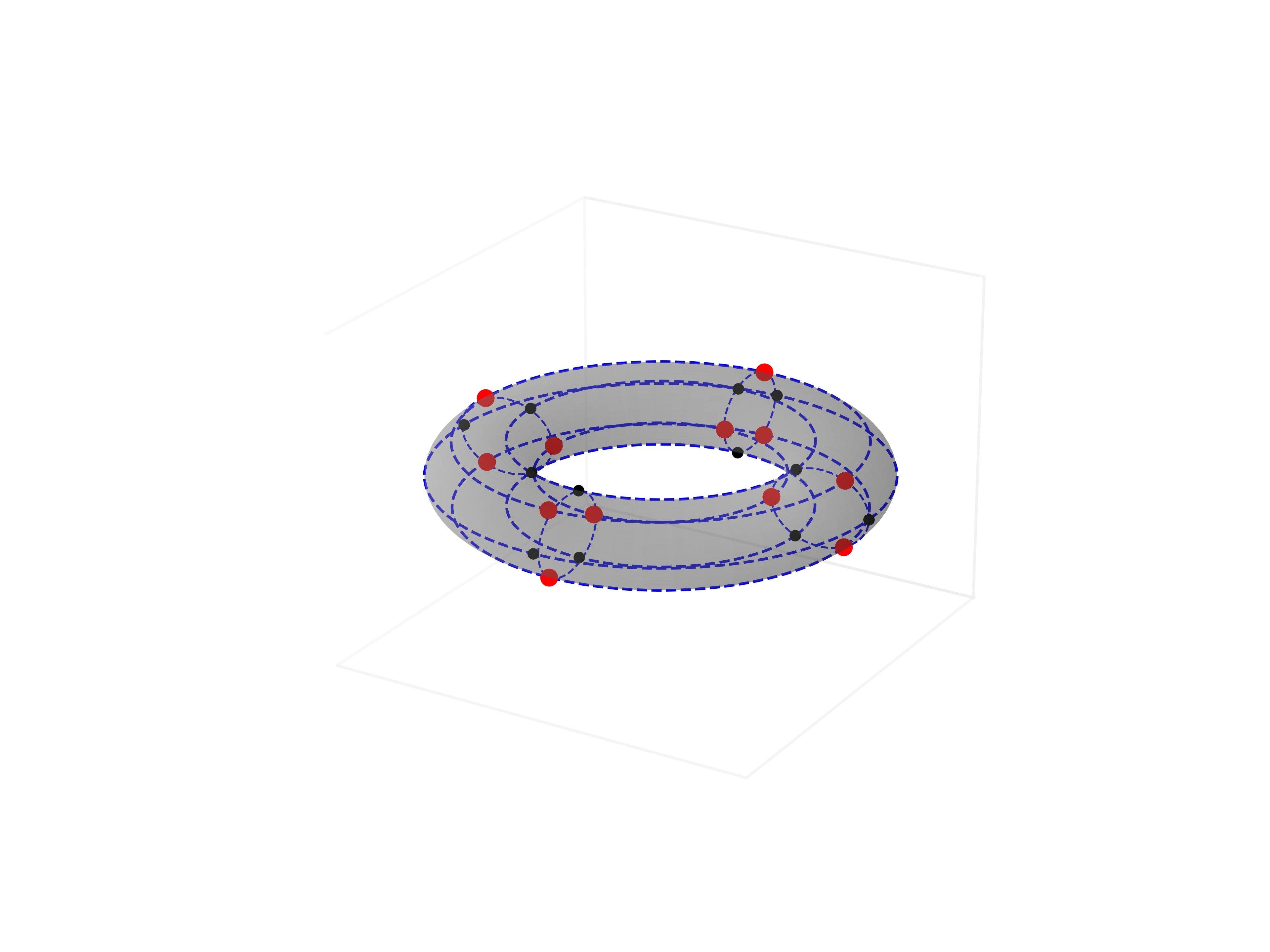}
    \end{minipage}
    &
    \begin{minipage}{.25\textwidth}
      \includegraphics[trim = 4cm 4.5cm 4cm 4.5cm, clip, width=\linewidth]{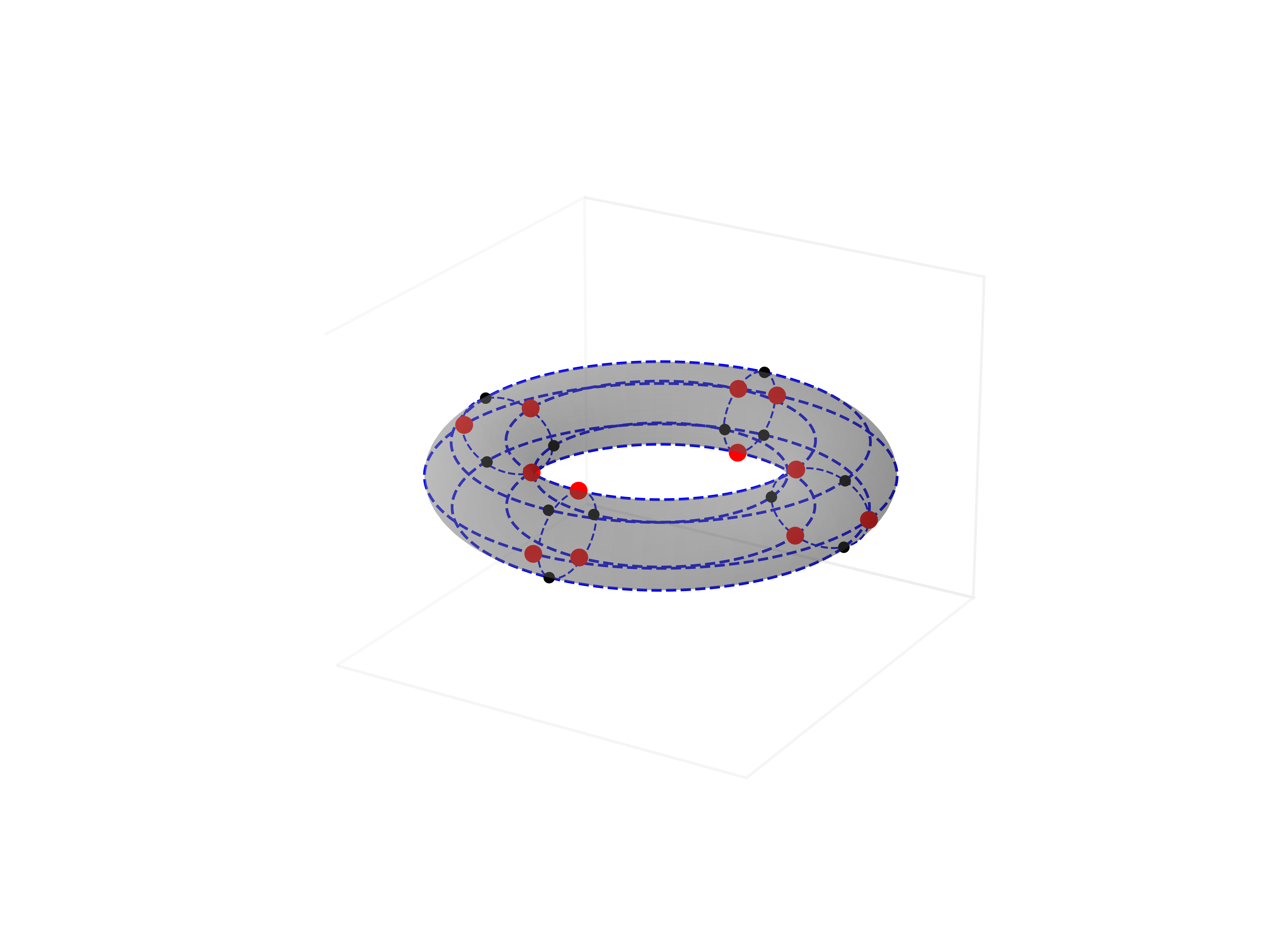}
    \end{minipage}
    & 
      \begin{minipage}{.25\textwidth}
      \includegraphics[trim = 4cm 4.5cm 4cm 4.5cm, clip, width=\linewidth]{goodtorus_theta_odd.png}
    \end{minipage}
    \\ \hline
  \end{tabular}
  \captionof{figure}{An illustration for a walk on $\mathbb{Z}_4\times \mathbb{Z}_6$ driven by $p$ (defined by \eqref{eq:DiscreteTorusGen}) for $q_1=q_2=1/2$, $(a_1,b_1)=(1,1)$, $(a_2,b_3)=(0,3)$.}\label{fig:DiscreteTorus}
\end{table}
\end{example}

\begin{example}[Random Walks on the Infinite Elevator $\mathbb{Z}_4\times\mathbb{Z}$] In this example we consider two random walks on $\mathbb{Z}_4\times\mathbb{Z}$. While the support of each driving measure contains at least one element of infinite order, only one sees a Gaussian-like diffusion. 
\begin{enumerate}
\item Define $p_1\in\mathcal{M}_2(\mathbb{Z}_4\times\mathbb{Z})$ by
\begin{eqnarray*}
p_1(a,c) = 
\begin{cases}
1/2 & \mbox{if} \quad (a,c) = (\pm 1, 1) \\
0 & \mbox{else}  
\end{cases}.
\end{eqnarray*}
Thinking of the factor $\mathbb{Z}_4$ as labeling position $a$ within an elevator car and $\mathbb{Z}$ labeling the height $c$ of the car, the random walk driven by $p_1$ is one that diffuses within the car while the elevator is (deterministically) moving up by one floor per step. Since $G_p=\langle\supp(p_1)-(-1,1)\rangle=2\mathbb{Z}_4\times\{0\}$ is rank free, we place this random walk into the $d=0$ case of Theorem \ref{thm:MainLLT}. For reference, in the proof of this case, $T:G\to A\times \mathbb{Z}^k$ is the identity map with $A=\mathbb{Z}_4$ and $\mathbb{Z}^k=\mathbb{Z}^1$, and $q_A(a)=1/2$ for $a=\pm 1\in\mathbb{Z}_4$. We have
\begin{eqnarray*}
\widehat{p_1}(\alpha,\gamma) = \frac{1}{2}e^{2\pi i \alpha/4}e^{i\gamma} + \frac{1}{2}e^{-2\pi i \alpha/4}e^{i\gamma}=\cos(2\pi\alpha/4)e^{i\gamma}
\end{eqnarray*}
$(\alpha,\gamma)\in\mathbb{Z}_4\times\mathbb{T}\cong \widehat{\mathbb{Z}_4\times\mathbb{Z}}.$ We see easily that $\Omega(p_1) = \{0, 2\}\times\mathbb{T}$ (consistent with \eqref{eq:MainLLT4}) and $\omega_{p_1}(\alpha,\gamma)=d\#(\alpha)\times d\gamma/2\pi$ so that
\begin{eqnarray*}
\Theta_{p_1}(n,a,c) &=& \sum_{\alpha \in \{0,2\}} \frac{1}{2\pi}\int_{\mathbb{T}}\left( \cos(2\pi\alpha/4) e^{i \gamma}\right)^n e^{-2\pi a\alpha/4}e^{-ic\gamma} d\gamma \\
&=&\left( \sum_{\alpha \in \{0,2\}} \left(\cos(2\pi\alpha/4)\right)^n  e^{-2\pi i a \alpha/4}\right) \frac{1}{2\pi} \int_\mathbb{T} e^{i(n-c)\gamma} d\gamma \\
&=& \left[ 1 + (-1)^{n-a} \right] \delta_{c,n},
\end{eqnarray*}
for $n\in\mathbb{N}$ and $(a,c)\in\mathbb{Z}_4\times\mathbb{Z}$; here $\delta$ denotes the Kronecker delta function.  By virtue of Theorem \ref{thm:MainLLT}, we find
\begin{eqnarray*}
p_1^{(n)}(a,c)= \Theta_{p_1}(n,a,c) \frac{1}{4}+O(\rho^n)=\frac{(1+(-1)^{n-a})\delta_{n,c}}{4}+O(\rho^n)
\end{eqnarray*}
uniformly for $(a,c) \in \mathbb{Z} \times \mathbb{Z}_4$ as $n \rightarrow \infty$. In thinking about this walk carefully, a person inside the elevator car must walk left or right (and so completely and evenly onto the next coset of $2\mathbb{Z}_4$ in $\mathbb{Z}_4$). In other words, it takes no time to converge to the uniform distribution on cosets of $G_p=2\mathbb{Z}_4\times\{0\}$ and so we really only see the dance. This suspicion is confirmed upon noting that $\rho=\max\{\abs{\widehat{q_A}(\alpha)}=\abs{\cos(2\pi \alpha/4)}:\alpha\in\mathbb{Z}_4\setminus \{0,2\}\}=0$ so that the error in the above limit theorem is identically zero. This result and relationship between the supports of $p_1^{(n)}$ and $\Theta_{p_1}$ (as guaranteed by Proposition \ref{prop:ThetaCapturesSupport}) are illustrated in Figure \ref{fig:ElevatorUp}.

\begin{table}[!h]
  \centering
  \begin{tabular}{ |c | c | c | c | }
    \hline
     & $n=1$ & $n=2$ & $n=3$ \\ \hline
     \rotatebox[origin=c]{90}{$\supp(p_1^{(n)})$}
     &
    \begin{minipage}{.25\textwidth}
      \includegraphics[trim = 3cm 3cm 3cm 3cm, clip, width=\linewidth]{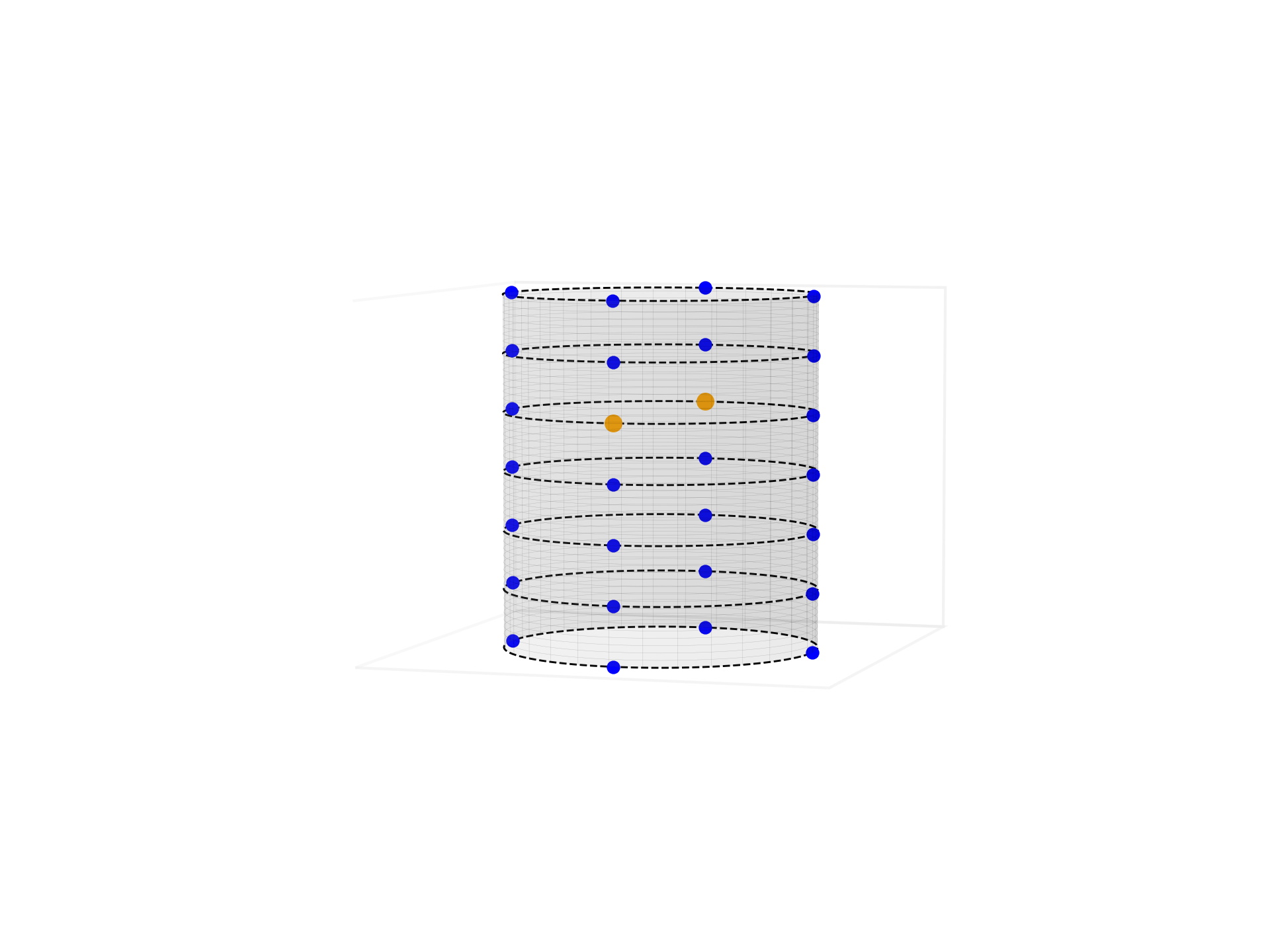}
    \end{minipage}
    &
    \begin{minipage}{.25\textwidth}
      \includegraphics[trim = 3cm 3cm 3cm 3cm, clip, width=\linewidth]{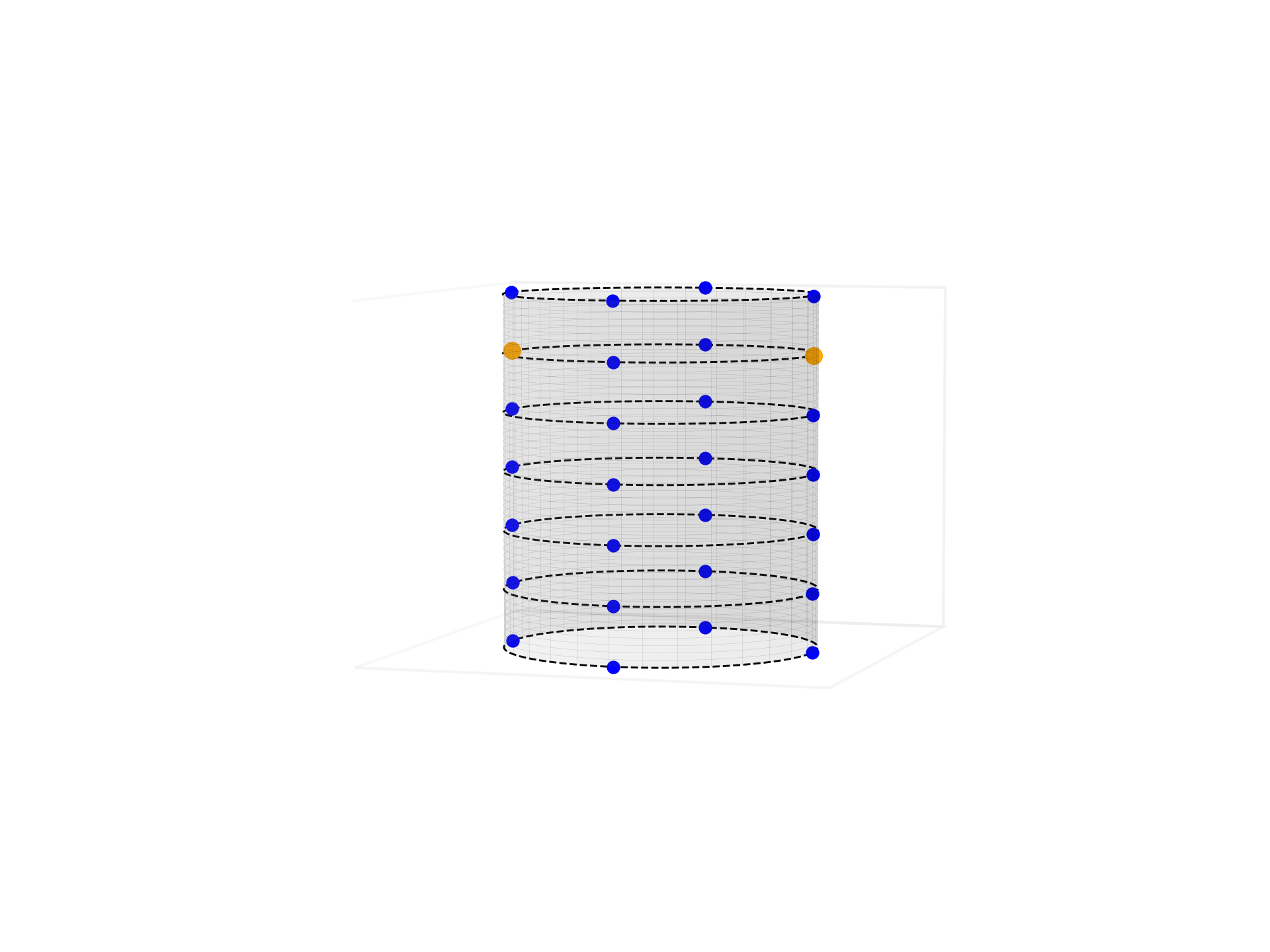}
    \end{minipage}
    & 
      \begin{minipage}{.25\textwidth}
      \includegraphics[trim = 3cm 3cm 3cm 3cm, clip, width=\linewidth]{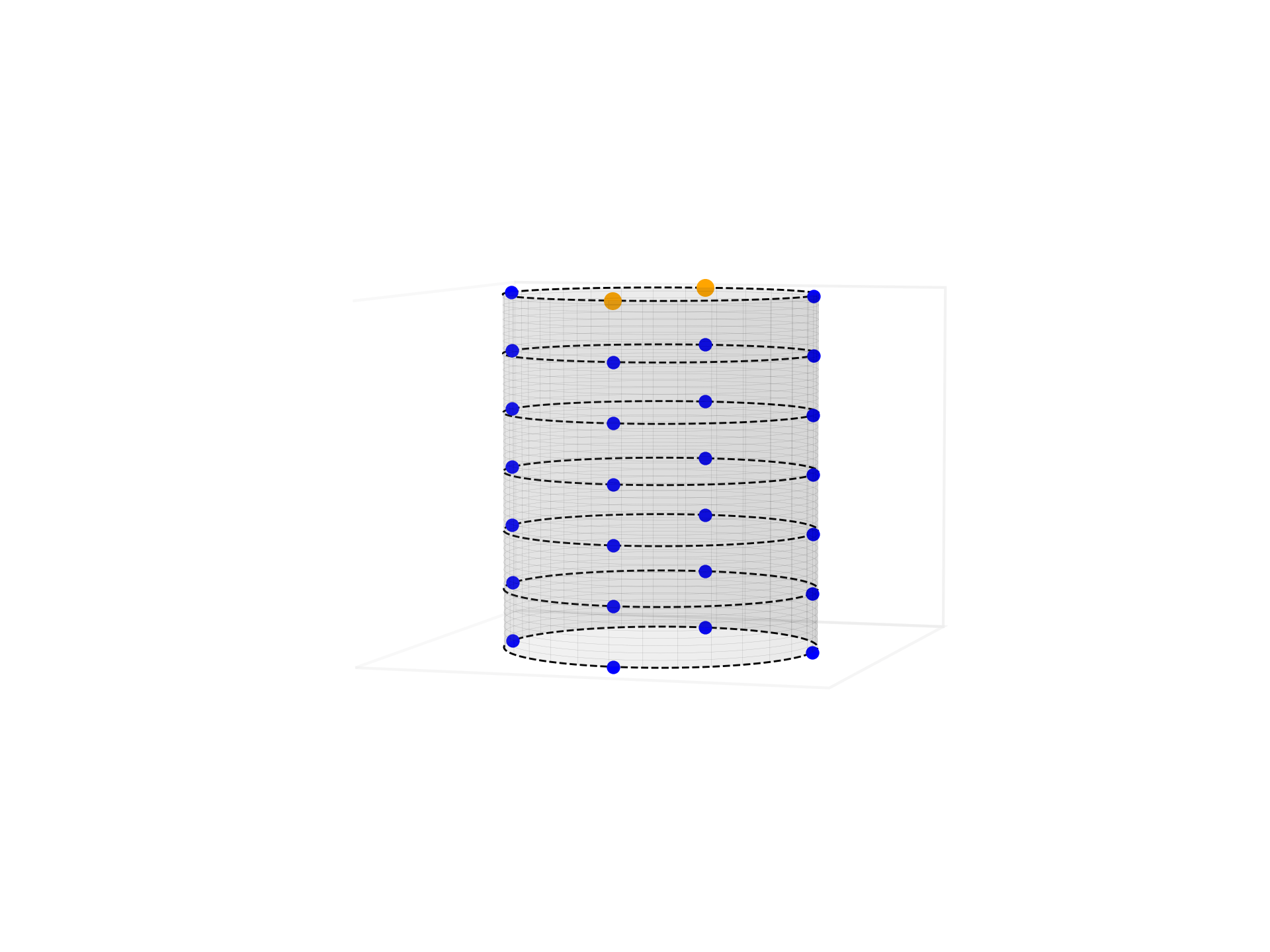}
    \end{minipage}
    \\ \hline
    \rotatebox[origin=c]{90} {$\supp(\Theta_{p_1}(n,\cdot))$}
     &
    \begin{minipage}{.25\textwidth}
      \includegraphics[trim = 3cm 3cm 3cm 3cm, clip, width=\linewidth]{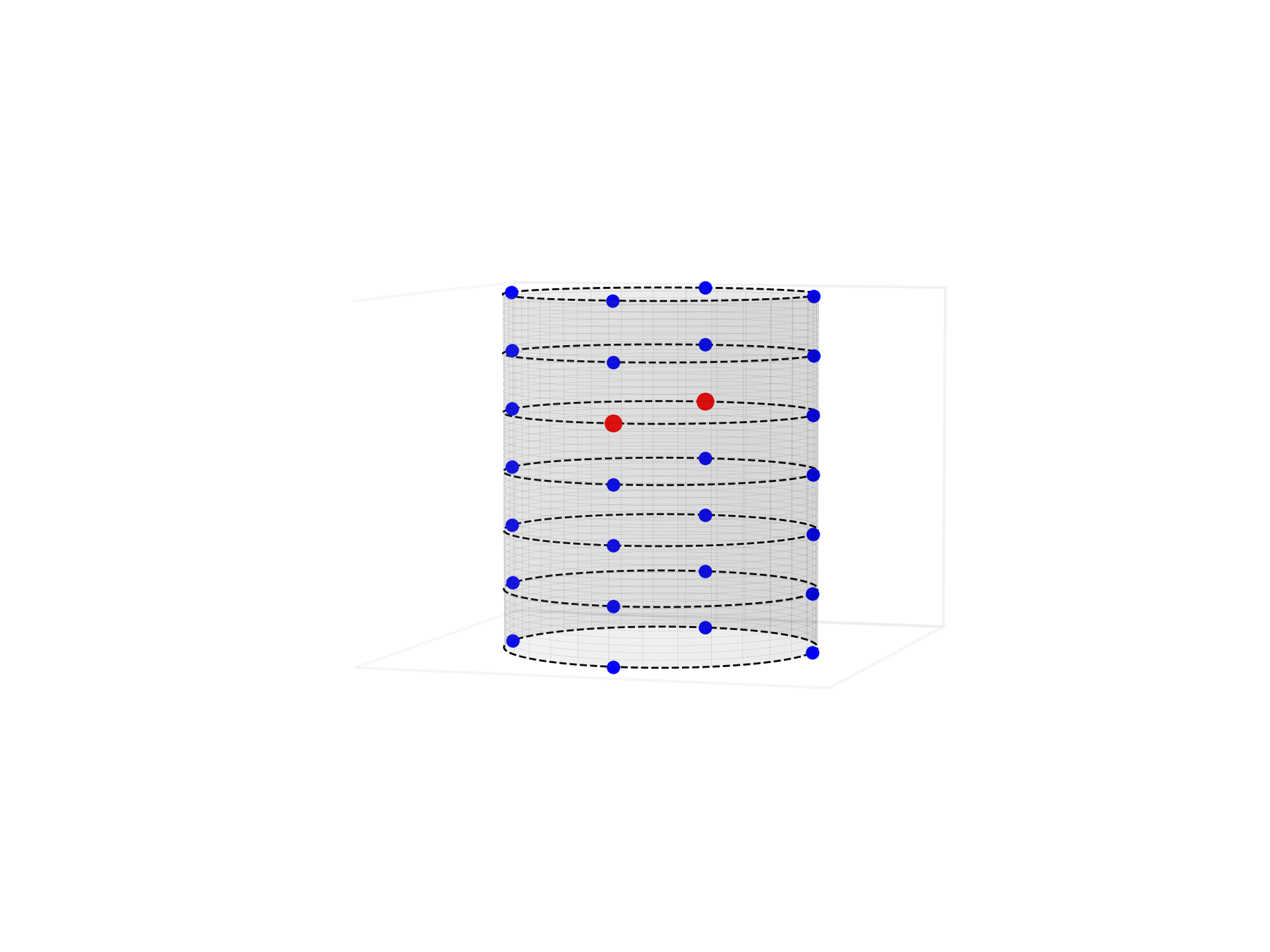}
    \end{minipage}
    &
    \begin{minipage}{.25\textwidth}
      \includegraphics[trim = 3cm 3cm 3cm 3cm, clip, width=\linewidth]{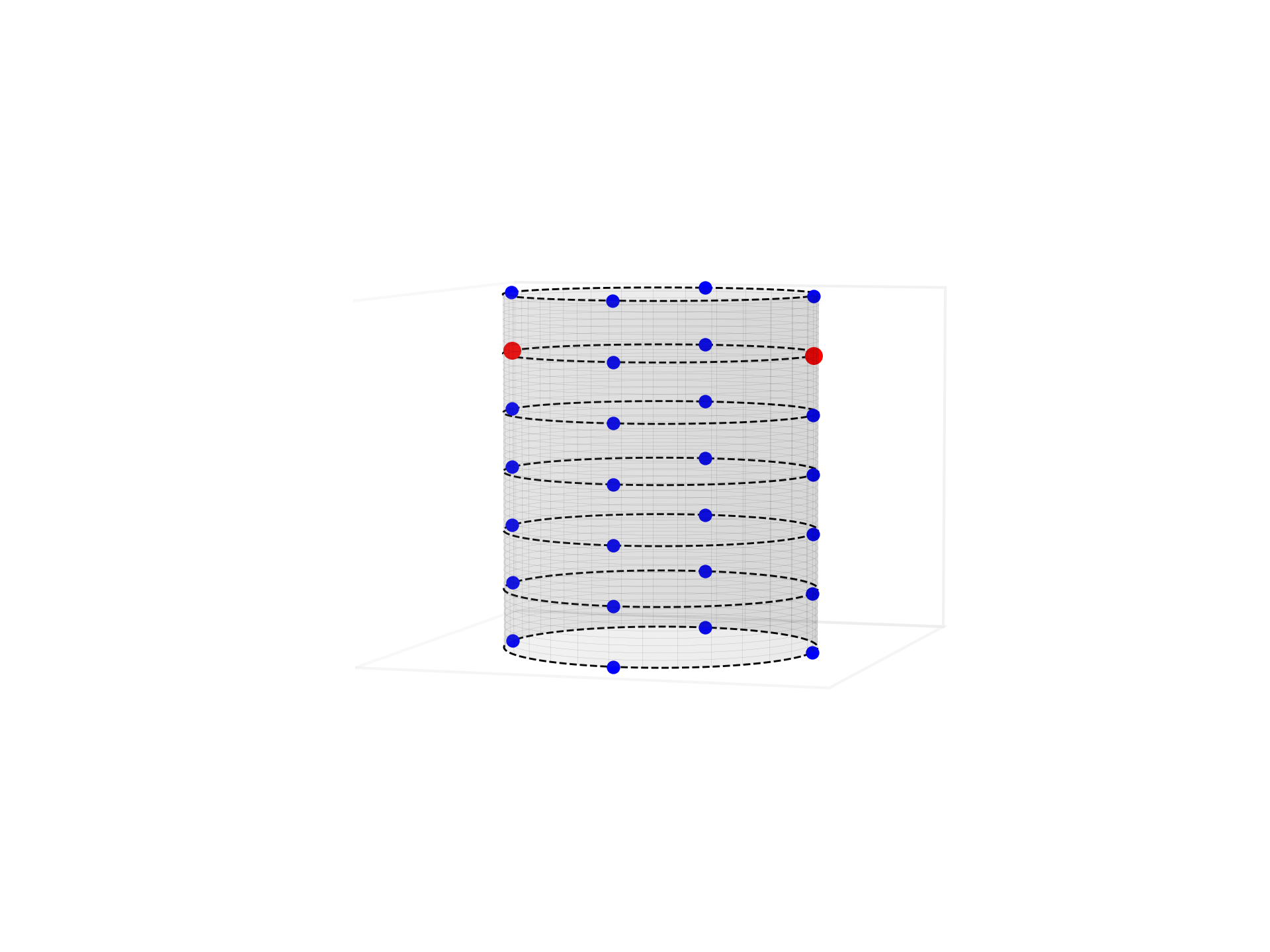}
    \end{minipage}
    & 
      \begin{minipage}{.25\textwidth}
      \includegraphics[trim = 3cm 3cm 3cm 3cm, clip, width=\linewidth]{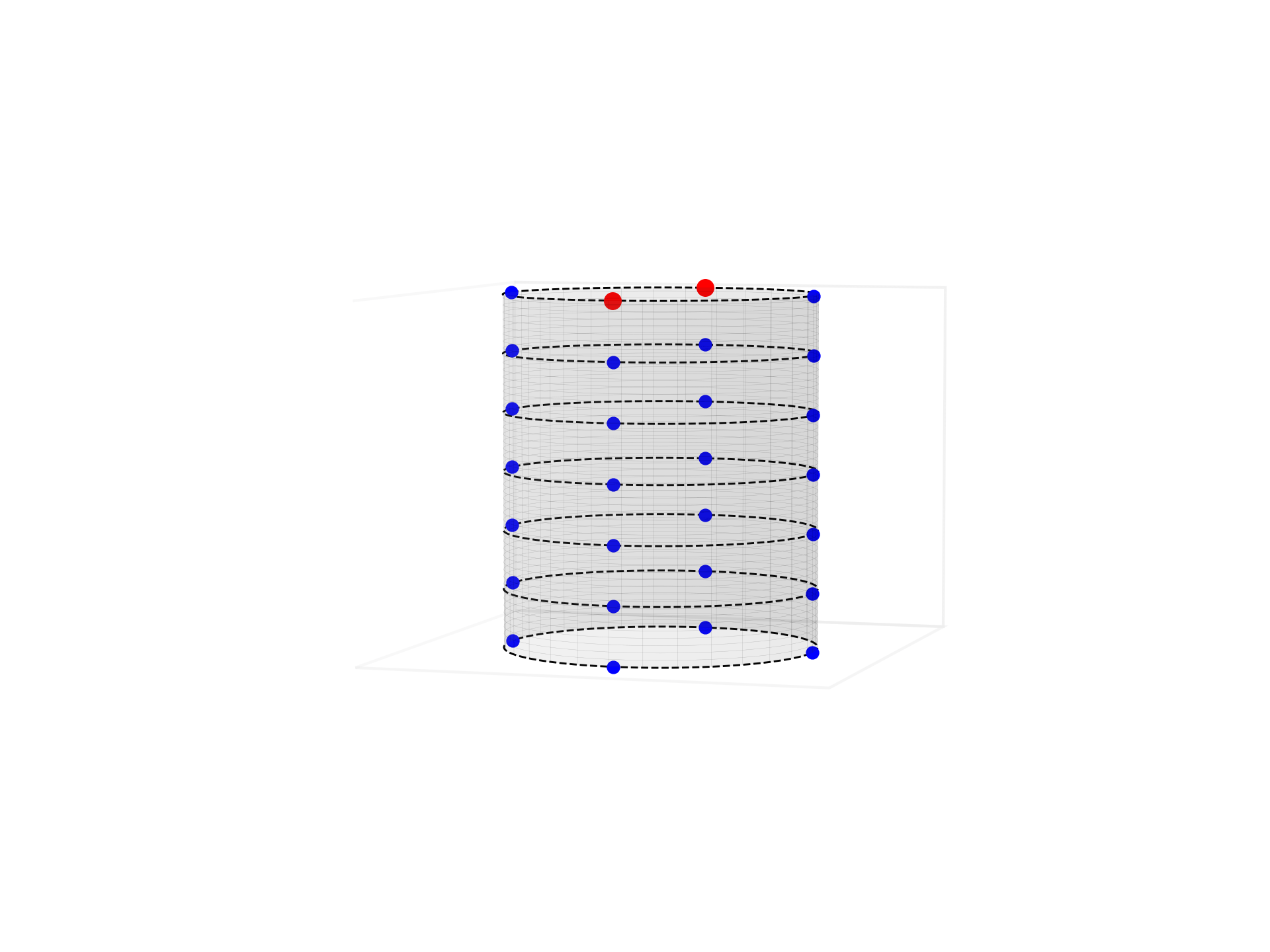}
    \end{minipage}
    \\ \hline
  \end{tabular}
  \captionof{figure}{An illustration for a walk on $\mathbb{Z}_4\times \mathbb{Z}$ driven by $p_1$. $\supp(p_1^{(n)})$ and $\supp(\Theta_{p_1}(n,\cdot))$ are indicated by yellow dots and red dots, respectively, for $n = 1,2,3$.}\label{fig:ElevatorUp}
\end{table}

\item Define $p_2\in\mathcal{M}_2(\mathbb{Z}_4\times\mathbb{Z})$ by
\begin{eqnarray*}
p_2(a,b) = 
\begin{cases}
1/4 & \mbox{if} \quad (a,b) = (\pm 1, 0), (0, \pm 1) \\
0 & \mbox{else}  
\end{cases}
\end{eqnarray*}
for $(a,b) \in \mathbb{Z}_4 \times \mathbb{Z}$. In contrast to $p_2$, the height $b$ of the elevator is now random. Here, a rider makes the following random selections independently at each step: With equal probability, move $\pm1$ within the elevator or press up or down. We have
\begin{eqnarray*}
\widehat{p_2}(\alpha,\beta) = \frac{1}{4} \left(e^{2\pi i\alpha/4} + e^{-2\pi i\alpha/4}+e^{i\beta}+e^{-i\beta}\right)=\frac{1}{2}\left(\cos(2\pi\alpha/4)+\cos(\beta)\right)
\end{eqnarray*}
for $(\alpha,\beta)\in \mathbb{Z}_4\times\mathbb{T}=\widehat{\mathbb{Z}_4\times\mathbb{Z}}$. In this case, $\Omega(p_2) = \{(0,0),(2,\pi)\}\leq\mathbb{Z}_4\times \mathbb{T}$ and from this we easily compute
\begin{eqnarray*}
\Theta_{p_2}(n,a,b) = 1+(-1)^ne^{-i\pi a}e^{-i\pi b}=1 + (-1)^{n-a-b},
\end{eqnarray*}
for $n\in\mathbb{N}$ and $(a,b)\in \mathbb{Z}_4\times\mathbb{Z}$. It is easy to see that $G_p=\langle (1,1), (-1,1)\rangle$ so that $d=\rank(G_p)=1$. Thus, we fall into Case 1 of Theorem \ref{thm:MainLLT} which gives us the local limit theorem
\begin{equation*}
p_2^{(n)}(a,b)=\frac{\Theta_{p_2}(n,a,b)}{4\sqrt{n\pi}}e^{-b^2/n}+o(1/\sqrt{n})
\end{equation*}
uniformly for $(a,b)\in\mathbb{Z}_4\times\mathbb{Z}$ as $n\to\infty$. Here, we have made use of the fact that $\varphi_*(p_2)(b)=\sum_{a\in\ A}p_2(a,b)=1/2$ when $b=\pm 1$ and $0$ otherwise (in accordance with \eqref{eq:MainLLT*3}). This result and relationship between the supports of $p_2^{(n)}$ and $\Theta_{p_2}$ (as guaranteed by Proposition \ref{prop:ThetaCapturesSupport}) are illustrated in Figure \ref{fig:DiffusiveElevator}.

\begin{table}[!h]
  \centering
  \begin{tabular}{ | c | c | c | c | }
    \hline
     & $n=1$ & $n=2$ & $n=3$ \\ \hline
     \rotatebox[origin=c]{90}{$\supp(p_2^{(n)})$}
     &
    \begin{minipage}{.25\textwidth}
      \includegraphics[trim = 3cm 3cm 3cm 3cm, clip, width=\linewidth]{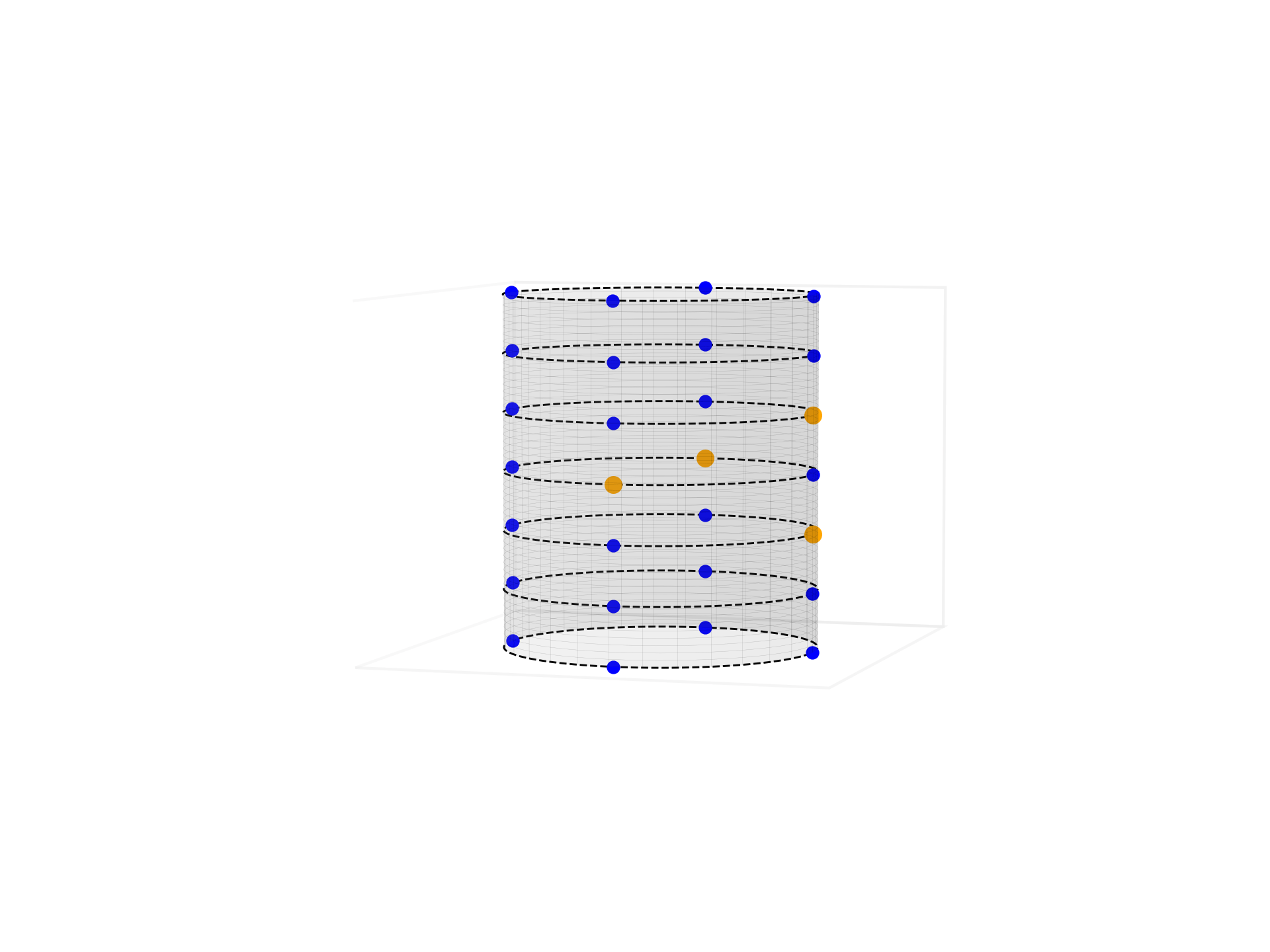}
    \end{minipage}
    &
    \begin{minipage}{.25\textwidth}
      \includegraphics[trim = 3cm 3cm 3cm 3cm, clip, width=\linewidth]{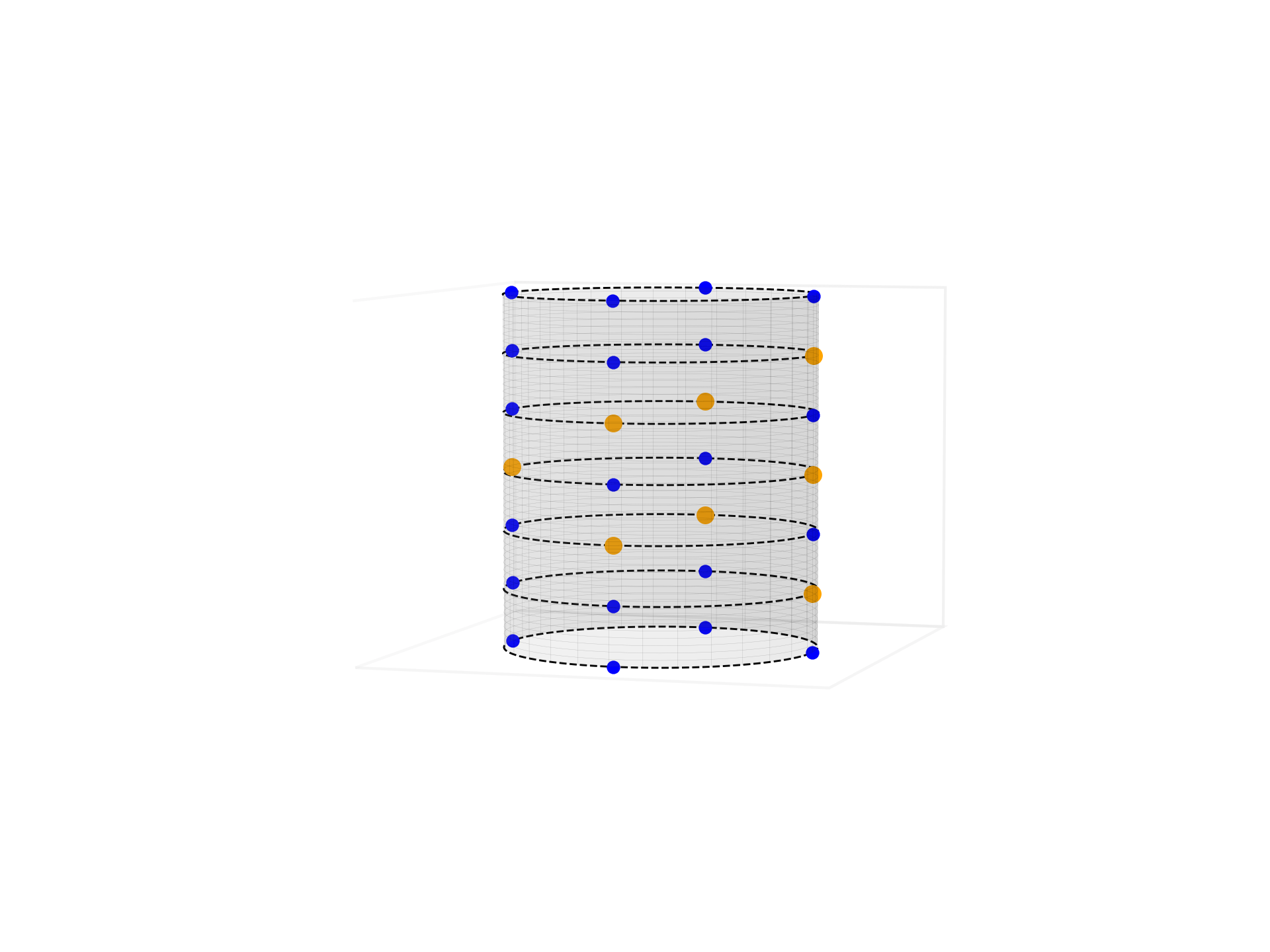}
    \end{minipage}
    & 
      \begin{minipage}{.25\textwidth}
      \includegraphics[trim = 3cm 3cm 3cm 3cm, clip, width=\linewidth]{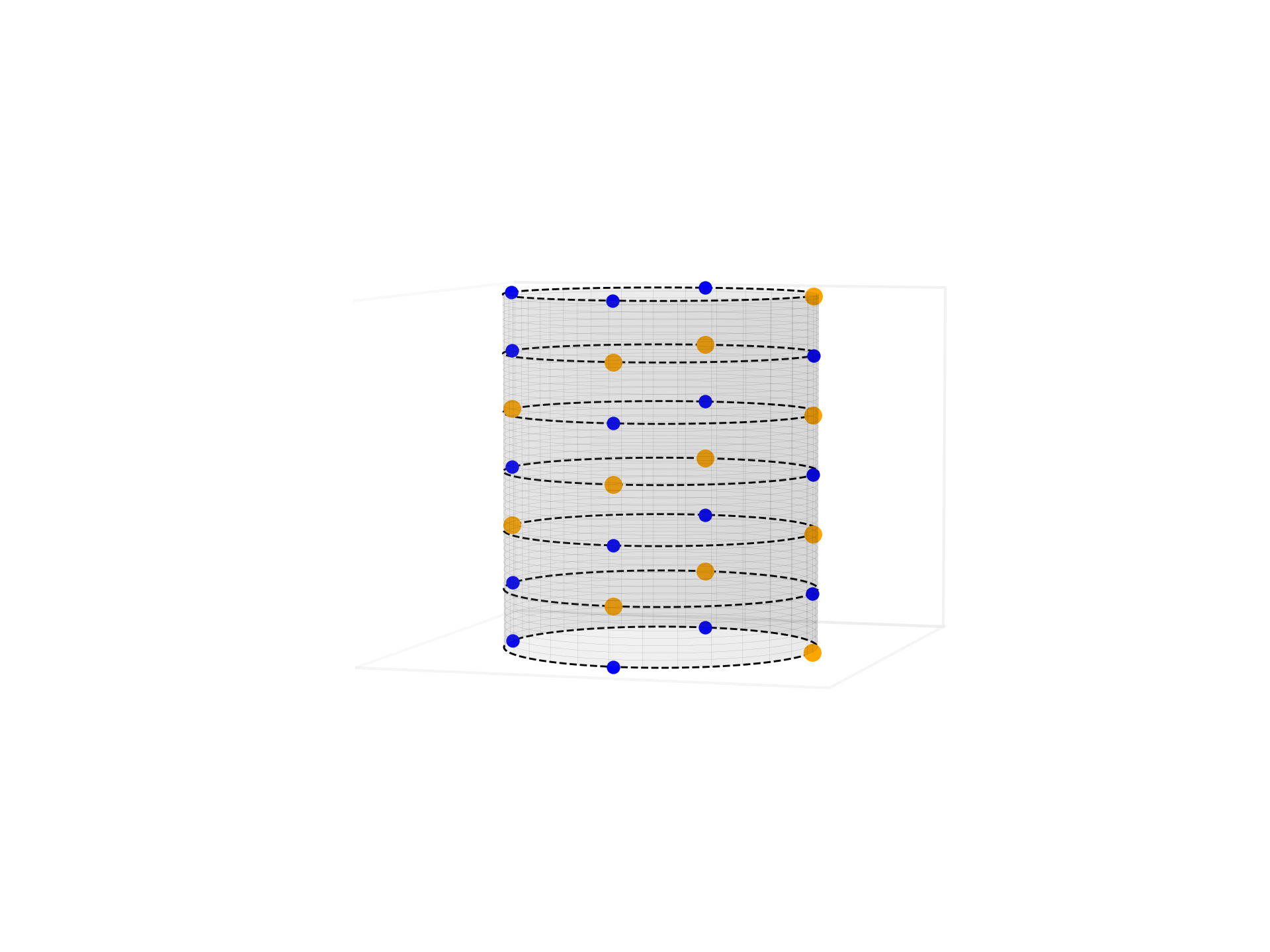}
    \end{minipage}
    \\ \hline
    \rotatebox[origin=c]{90} {$\supp(\Theta_{p_2}(n,\cdot))$}
     &
    \begin{minipage}{.25\textwidth}
      \includegraphics[trim = 3cm 3cm 3cm 3cm, clip, width=\linewidth]{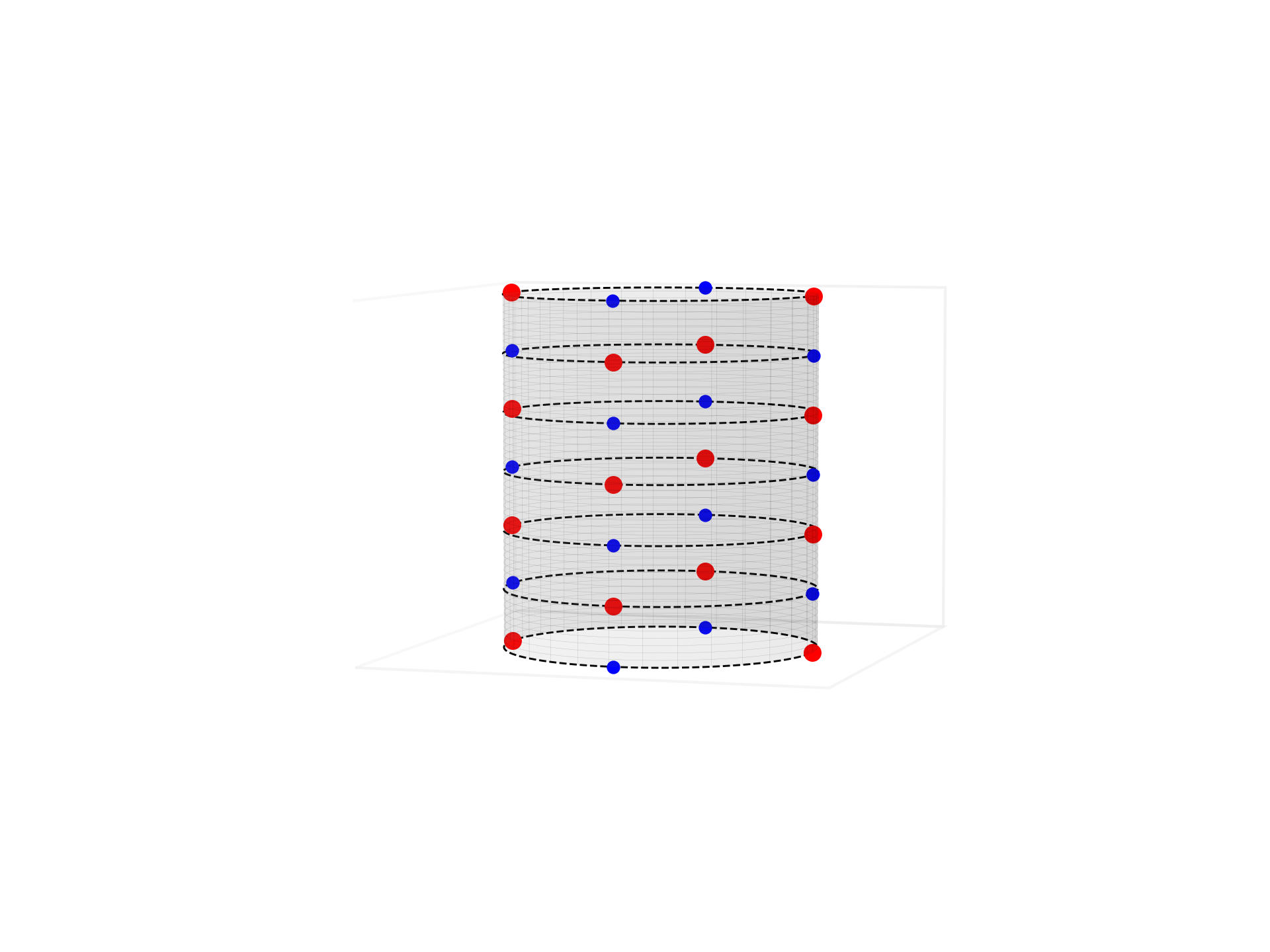}
    \end{minipage}
    &
    \begin{minipage}{.25\textwidth}
      \includegraphics[trim = 3cm 3cm 3cm 3cm, clip, width=\linewidth]{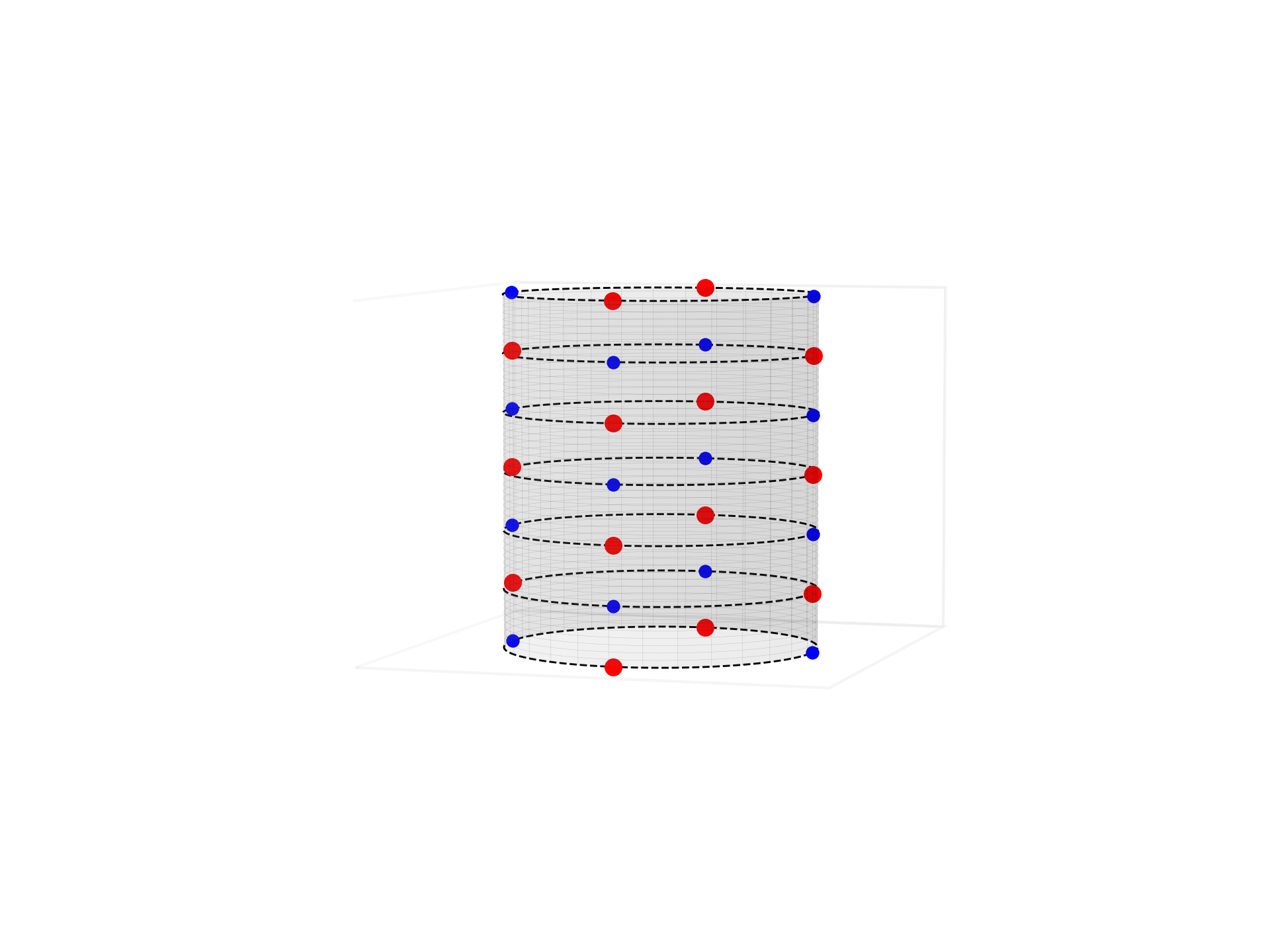}
    \end{minipage}
    & 
      \begin{minipage}{.25\textwidth}
      \includegraphics[trim = 3cm 3cm 3cm 3cm, clip, width=\linewidth]{belev_Theta_n_odd-min.png}
    \end{minipage}
    \\ \hline
  \end{tabular}
  \captionof{figure}{An illustration for the walk on $\mathbb{Z}_4\times\mathbb{Z}$ driven by $p_2$. $\supp(p_2^{(n)})$ and $\supp(\Theta_{p_6}(n,\cdot))$ are illustrated by yellow dots and red dots, respectively, for $n = 1,2,3$.}\label{fig:DiffusiveElevator}
\end{table}
\end{enumerate}
\end{example}

\begin{example}[A one/two dimensional walk on $\mathbb{Z}^2$ -- from \cite{Spitzer}]
Consider a walk on $\mathbb{Z}^2$ driven by 
\begin{eqnarray*}
p(x,y) = 
\begin{cases}
1/2 & \mbox{if} \quad (x,y)=(1,0),(0,1) \\
0 & \mbox{else} \\
\end{cases}
\end{eqnarray*}
for $x,y \in \mathbb{Z}$. This walk is not aperiodic, symmetric, or irreducible and does not meet the hypotheses of the local limit theorems of \cite{LawlerLimic2010,Woess2000,Spitzer}. In fact, F. Spitzer discusses this walk in \cite{Spitzer}, at one point calling it a two-dimensional aperiodic walk while pointing out its one-dimensional return asymptotic of $n^{-1/2}$. While this walk is posed in two dimensions, it is not genuinely two dimensional and so it is not captured by the results of \cite{RSC17}. We have 
\begin{equation*}\widehat{p}(\eta,\zeta)=(e^{i\eta}+e^{i\zeta})/2
\end{equation*}
and consequently
\begin{equation*}
\Omega(p)=\{(\eta,\zeta)\in\mathbb{T}^2:\eta=\zeta\}=\{(t,t):t\in\mathbb{T}\}.
\end{equation*}
It is easy to see that $d\omega_p(t)=dt/2\pi$ is a Haar measure on $\Omega(p)$ and, as we will confirm, it has the correct normalization. With this, \eqref{eq:OmegaDef} gives
\begin{equation}\label{eq:SpitzerExample1}
\Theta_p(n,x,y)=\frac{1}{2\pi}\int_{-\pi}^\pi \widehat{p}(t,t)^n e^{-ixt}e^{-iyt}\,dt=\frac{1}{2\pi}\int_{-\pi}^\pi e^{i(n-x-y)t}\,dt=\delta_{n,x+y}
\end{equation}
for $n\in\mathbb{N}$ and $(x,y)\in\mathbb{Z}^2$. In view of Proposition \ref{prop:ThetaCapturesSupport}, our random walk can visit $(x,y)$ at step $n$ only if $n=x+y$ and this, in particular, forces coordinates to have the same parity at even steps and opposite parity at odd steps. 

While writing down a local limit theorem directly is not too difficult (one can simply specify an appropriate epimorphism $\varphi:\mathbb{Z}^2\to\mathbb{Z}$), for illustrative purposes, we shall follow the method of proof of Theorem \ref{thm:MainLLT} thereby constructing an isomorphism $T:\mathbb{Z}^2\to\mathbb{Z}\times\mathbb{Z}$ from which we will obtain the local limit theorem \eqref{eq:MainLLTInfinite} with $\varphi=\Proj_{1}\circ T$. First, we take $T$ defined by
\begin{equation*}
T\begin{pmatrix} x \\ y \end{pmatrix}=\begin{pmatrix} x \\ x+y \end{pmatrix}
\end{equation*}
for $x,y\in\mathbb{Z}$. We remark that $T=\Phi$ is precisely the automorphism given by Lemma \ref{lem:reduce_many} and, for $S=\supp(p)$, it gives $T(\supp(p))=\Phi(S)=\{0,1\}\times \{1\}\subseteq \mathbb{Z}\times\{1\}$ where $\Proj_{\mathbb{Z}}(\Phi(S))=\Proj_{1}(\Phi(S))=\{0,1\}\subseteq\mathbb{Z}$ is a (genuinely) one dimensional set.  Aligning ourselves with the notation of the proof of Theorem \ref{thm:MainLLT} (and Lemma \ref{lem:LLTforH}), we have $q(b,c)=T_*(p)(b,c)=1/2$ when $(b,c)=(0,1),(1,1)$ and zero otherwise (here, all mentions of $a$ and $A$ are suppressed because $\mathbb{Z}^2$ is torsion free). With this, 
\begin{equation*}
\widehat{q}(\beta,\gamma)=\frac{1}{2}\left(1+e^{i\beta}\right) e^{i\gamma}
\end{equation*}
so that $\Omega(q)=\{0\}\times\mathbb{T}$, $\omega_q(\beta,\gamma)=d\#(\beta)\times d\gamma/2\pi$ and $q_{AB}=q_B=q(\cdot,1)$ has $\Omega(q_{AB})=\{0\}$. Consequently
\begin{equation*}
\Theta_q(n,b,c)=\sum_{\beta\in \{0\}}\frac{1}{2\pi}\int_{\mathbb{T}}\widehat{q}(\beta,\gamma)^n e^{-ib\beta}e^{-ic\gamma}\,d\gamma=\delta_{n,c}
\end{equation*}
and this has the prescribed normalization since $\Theta_q(0,0,0)=1=\abs{\Omega(q_{AB})}$. Upon noting that that $\mu=\E[q_B]=1/2$ and $\Gamma=\operatorname{Var}(q_B)=1/4$, we obtain
\begin{equation*}
q^{(n)}(b,c)=\Theta_q(n,b,c)K_{q_B}^n(b-n/2)+o(n^{-1/2})\\
=\delta_{n,c}\sqrt{\frac{2}{\pi n}}\exp\left(-\frac{2}{n}\left(b-\frac{n}{2}\right)^2\right)+o(n^{-1/2})
\end{equation*}
uniformly for $(b,c)\in\mathbb{Z}^2$ as $n\to\infty$ by virtue of Lemma \ref{lem:LLTforH}. Continuing to follow the proof of Theorem \ref{thm:MainLLT}, we have
\begin{equation*}\Theta_p(n,x,y)=\Theta_q(n,T(x))=\Theta_q(n,x,x+y)=\delta_{n,x+y}
\end{equation*}
for $n\in\mathbb{N}$ and $(x,y)\in\mathbb{Z}$; this confirms our computation in \eqref{eq:SpitzerExample1} (and simultaneously our choice of normalization of $d\omega_p(t)$). By this method, we also obtain the local limit theorem
\begin{equation*}
p^{(n)}(x,y)=q^{(n)}(T(x,y))=\delta_{n,x+y}\sqrt{\frac{2}{\pi n}}\exp\left(-\frac{2}{n}\left(x-\frac{n}{2}\right)^2\right)+o(n^{-1/2})
\end{equation*}
which holds uniformly for $(x,y)\in\mathbb{Z}^2$ as $n\to\infty$. We encourage the reader to verify directly that this is precisely the local limit theorem \eqref{eq:MainLLTInfinite} taking $\varphi(x,y)=(\Proj_{1}\circ T)(x,y)=x$. Finally, in view of Remark \ref{rmk:LLTIndicatorError} and noting that $n=x+y$ for every $(x,y)\in\supp(p^{(n)})$, we rewrite the above as
\begin{equation}\label{eq:SpitzerExample2}
p^{(n)}(x,y)=\delta_{n,x+y}\sqrt{\frac{2}{\pi n}}\exp\left(-\frac{(x-y)^2}{2n}\right)+o(\delta_{n,x+y}n^{-1/2})
\end{equation}
uniformly for $(x,y)\in\mathbb{Z}^2$ as $n\to\infty$. What's nice about this final expression is that it makes clear that our convolution powers are approximated by a one-dimensional Gaussian attractor centered on the line $x=y$ and propagating into the first quadrant along the wave front with $x+y=n$. This approximation is illustrated in Figure \ref{fig:SpitzerExample}.
 \begin{figure}[h!]
\begin{center}
\includegraphics[trim = 5.1cm 2.5cm 4.3cm 3.9cm, clip, width=.7\linewidth]{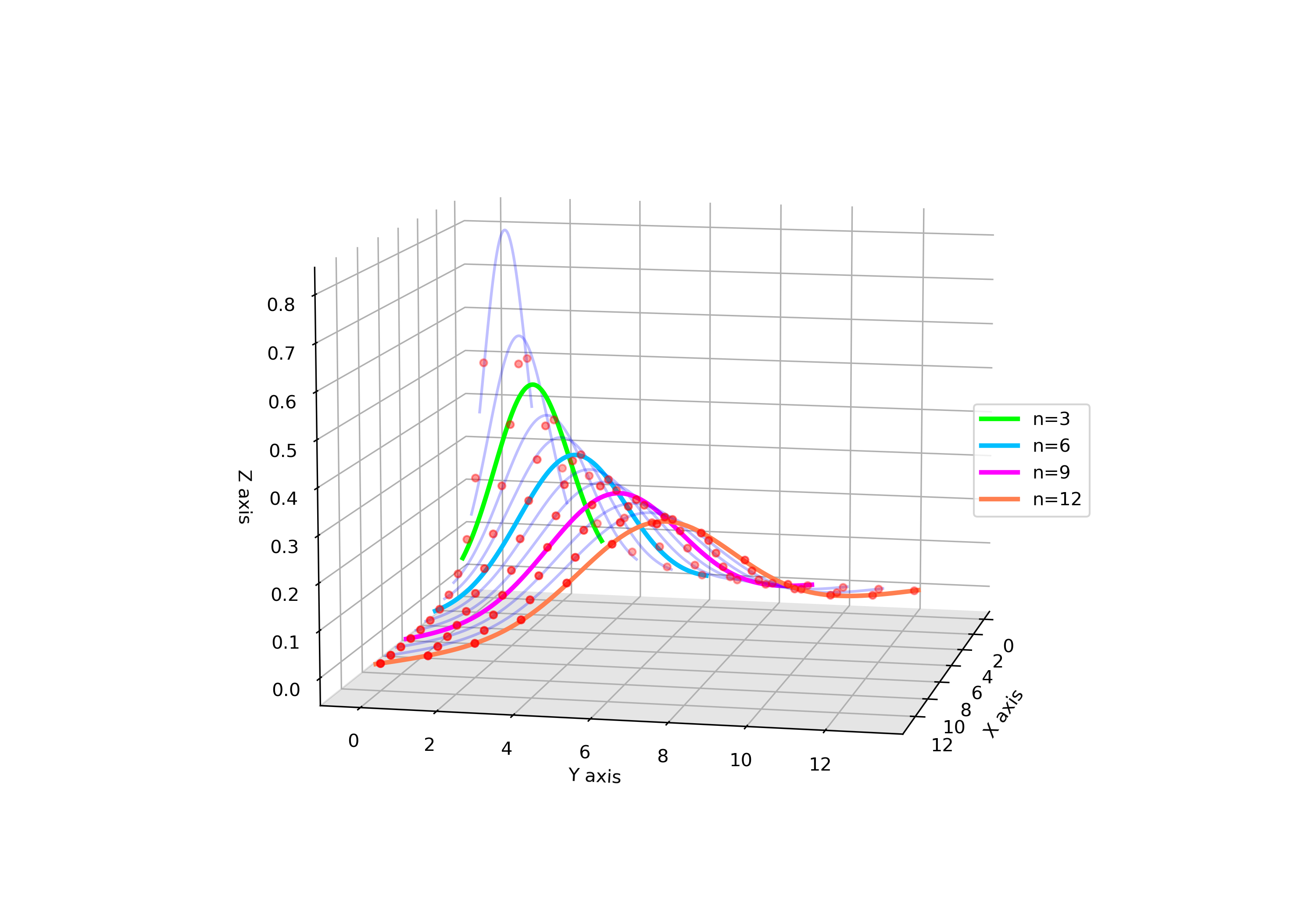}
\caption{The one-dimensional walk on $\mathbb{Z}^2$ driven by $p(0,1)=p(1,0)=1/2$. Values of $p^{(n)}(x,y)$ are illustrated by red dots and the attractors in \eqref{eq:SpitzerExample2} illustrated by solid lines  for $1\leq n\leq 12$.}\label{fig:SpitzerExample}
\end{center}
\end{figure}
\end{example}

\section*{Acknowledgements:} We thank Leo Livshits for suggesting the proof of Lemma \ref{isom_matrix}, a key step in the ``twisting space" Lemma \ref{lem:reduce_many}, and Fernando Gouv\^{e}a for many helpful discussions.

\appendix
\section{Cosets and Periodic Classes Coincide}\label{sec:PeriodicClasses}

In this section, we prove the identity \eqref{eq:CosetIden} used in the proof of Proposition \ref{prop:IrreducAndPerS}. This is the following lemma.
\begin{lemma}
Let $p\in\mathcal{M}(G)$ drive an irreducible walk on $G$ of period $s\in\mathbb{N}_+$, let $G_p$ be the subgroup of $G$ defined by \eqref{eq:GpDef}, and, for each $k\in\mathbb{N}$, define
\begin{equation*}
C_k=\{y\in G:p^{(ns+k)}(y)>0:n\in\mathbb{N}\}.
\end{equation*}
Then, for any $x_0\in \supp(p)$ and $k\in \mathbb{N}$,
\begin{equation*}
G_p+kx_0=C_k.
\end{equation*}
\end{lemma}

\begin{proof} We fix $x_0\in\supp(p)$ and treat first the case that $k=0$. To this end, we shall first prove the $C_0$ is a subgroup of $G$. Since $p$ is irreducible and periodic of period $s$, we immediately find that $0\in C_0$ and so $C_0$ is non-empty. For $x\in C_0$, we can choose $l\in\mathbb{N}_+$ for which $p^{(l)}(-x)>0$ since $p$ is irreducible. With this $l$ and $n\in\mathbb{N}$ for which $p^{(ns)}(x)>0$, we see that
\begin{equation*}
p^{(l+ns)}(0)\geq p^{(l)}(-x)p^{(ns)}(x)>0
\end{equation*}
and hence $s\vert (l+ns)$ since $p$ is periodic of period $s$. Consequently, $s\vert l$ showing that $-x\in C_0$ whenever $x\in C_0$. Given $x,y\in C_0$, let $n$ and $m$ be natural numbers for which $p^{(ns)}(x)>0$ and $p^{(ms)}(y)>0$. Then, $p^{((n+m)s)}(x+y)\geq p^{(ns)}(x)p^{(ms)}(y)>0$ so that $x+y\in C_0$. We have shown that $C_0$ is a subgroup of $G$. 

As $G_p=\langle \supp(p)-x_0\rangle$ in view of \eqref{eq:GpDef} and Lemma \ref{lem:equiv_gen}, to see that $G_p\subseteq C_0$, it suffices to show that $\supp(p)-x_0\subseteq C_0$. To this end, let $y=x-x_0\in \supp(p)-x_0$ and, using the irreducibility of $p$, select $m$ for which $p^{(m)}(-x_0)>0$ so that
\begin{equation*}
p^{(m+1)}(0)\geq p^{(m)}(0-x_0)p(x_0)>0.
\end{equation*}
By the $s$-periodicity of $p$, there is some natural number $n$ for which $ns=m+1$ and therefore
\begin{equation*}
p^{(ns)}(y)=p^{(m+1)}(x-x_0)\geq p^{(m)}(-x_0)p(x)>0
\end{equation*}
since $x\in \supp(p)$. Thus $y\in C_0$ and we have shown that $G_p\subseteq C_0$. 

To conclude that $G_p=C_0$, it remains to prove that $C_0\subseteq G_p$.  To this end, let $y\in C_0$ and select $n$ such that $p^{(ns)}(y)>0$. Since the random walk has reached $y$ at step $ns$, there must be some chain formed by elements in $p$'s support that reaches $y$ at time $ns$. In other words, there are elements $x_1,x_2,\dots, x_m\in\supp(p)$ and natural numbers $n_1,n_2,\dots,n_m$ for which $ns=n_1+n_2+\cdots+n_m$ and $y=n_1x_1+n_2x_2+\cdots+n_nx_n$.
With this, observe
\begin{eqnarray*}
y&=&n_1(x_1-x_0)+n_2(x_2-x_0)+\cdots+n_M(x_m-x_0)+(n_1+n_2+\cdots+n_m)x_0\\
&=&n_1(x_1-x_0)+n_2(x_2-x_0)+\cdots+n_m(x_m-x_0)+n s x_0.
\end{eqnarray*}
Recalling from Lemma \ref{lem:PeriodicReturnsToSupport}, $nsx_0\in G_p$ and so that above gives $y\in G_p$ and so $G_p=C_0$.

With the $k=0$ case complete, it is now easy to prove that $G_p+kx_0=C_k$ for $1\leq k\leq s-1$. Indeed, for a fixed $y=x+kx_0\in G_p+kx_0=C_0+kx_0$, let $n$ be a natural number for which $p^{(ns)}(x)>0$ and observe that
\begin{equation*}
p^{(ns+k)}(y)\geq p^{(ns)}(x)p^{(k)}(kx_0)\geq p^{(ns)}(x)p(x_0)^k>0.
\end{equation*}
Thus, $y\in C_k$ and we have shown that $G_p+kx_0\subseteq C_k$. To prove the reverse inclusion, let $y\in C_k$ and take $n$ for which $p^{(ns+k)}(y)>0$. Writing $(n+1)s=ns+k+(s-k)$ we have
\begin{equation*}
p^{((n+1)s)}(y+(s-k)x_0)\geq p^{(ns+k)}(y)p^{(s-k)}((s-k)x_0)>0.
\end{equation*}
Consequently, $y+(s-k)x_0\in C_0=G_p$. Since $sx_0\in G_p$, $-sx_0\in G_p$ and so $y-kx_0=y+(s-k)x_0-sx_0\in G_p$ whence $y\in G_p+kx_0.$ 

We have proven that $G_p+kx_0=C_k$ for $k=0,1,\dots,s-1$. The full result, for $k\in \mathbb{N}$, now follows immediately upon inspecting the definition of $C_k$ and recalling that $sx_0\in G_p$ by virtue of Lemma \ref{lem:PeriodicReturnsToSupport}.
\end{proof}

\section{Twisting Space Lemmas -- Proof of Lemma \ref{lem:reduce_many}}\label{app:reduce}

In this section, we prove Lemma \ref{lem:reduce_many} which says that, if a set $S\in\mathbb{Z}^k$ is $d$ dimensional in the sense of Definition \ref{def:dimension}, then we can find an automorphism of $\mathbb{Z}^k$ that aligns $S$ with coordinates so that it really ``fills" up the first $d$ coordinates while being constant on the remaining $k-d$ coordinates. To this end, we begin with a basic result. 

\begin{lemma}{\label{isom_matrix}}
For an integer $k\geq 1$, let $a_1, \dots, a_k$ be integers, at least one of which is non-zero. Then there exists $M \in \Gl_k(\mathbb{Z})$ whose bottom row is $(a_1,\dots,a_k)$ and has
    \begin{equation*}\det M = \gcd(a_1,\dots,a_k).
    \end{equation*}
\end{lemma}
\begin{proof}
The statement is clear when $k=1$. In the case that $k\geq 2$, we describe an algorithm that produces the matrix $M$. To this end, let $A$ be an arbitrary $k\times k$ matrix whose last row is $(a_1,a_2,\dots,a_k)$. With this, consider
    \begin{eqnarray*}
        T = \left ( \frac{A}{I}\right ) =  \left (\begin{array}{cccc}
            \vdots & \vdots & \dots & \vdots \\
            \vdots & \vdots & \dots & \vdots \\
            \vdots & \vdots & \ddots & \vdots \\
            a_1 & a_2 & \dots & a_k \\
            \hline
            1 & 0 &\dots & 0 \\
            0 & 1 &\dots & 0 \\
            \vdots & \vdots & \ddots & \vdots \\
            0 & 0 & \dots & 1
        \end{array}  \right )=: \begin{pmatrix}
            | & | & \dots & | \\
            C_1 & C_2 &  \dots & C_k \\
            | & | & \dots & |
        \end{pmatrix}
.
    \end{eqnarray*}
We will perform the following procedure to $C_1$ and $C_2$, and the rest of the argument follows inductively. If the last entry of either $C_1$ or $C_2$ is 0, swap the two columns (if necessary) to make sure that $C_1$ has last entry 0 and proceed to the next paragraph. Otherwise, apply the Euclidean algorithm to find integers $r_1,s_1$ such that
    \begin{eqnarray*}
        a_1 = s_1a_2 + r_1,
    \end{eqnarray*}
    where $0 \leq r_1 < \abs{a_2}$. Update $C_1$ to be $C_1 - s_1C_2$. If $r_1 > 0$, there are integers $r_2, s_2$ such that
    \begin{eqnarray*}
        a_2 = s_2r_1 + r_2,
    \end{eqnarray*}
    where $0 \leq r_2 < r_1$. Using our new $C_1$, update $C_2$ to be $C_2 - s_2C_1$. We can repeat this process until the $k$th entry of either $C_1$ or $C_2$ becomes 0, and we automatically know that the $k$-th entry of the other column among the two is $g_1 = \gcd(a_1,a_2)$. If $C_2 \neq g_1$, exchange $C_1$ and $C_2$. Now, the last entry of (updated) $C_2$  is $g_1$.

    We now repeat the step in the above paragraph with $C_2$ and $C_3$ so that the $k$-th entry of $C_3$ becomes $g_2=\gcd(\gcd(a_1,a_2),a_3)=\gcd(a_1,a_2,a_3)$. Inductively, the same procedure is applied to $C_3$ and $C_4$, and so on. The process terminates after the step is performed on $C_{k-1}$ and $C_k$, which results in the last entry of $C_k$ being\footnote{Here we use the convention that $\gcd(0,n)=n$ for $n = 0,1,2,\dots$ } $g=g_{k-1} = \gcd(a_1,\dots,a_k)$. 
    This algorithm results in a matrix of the form
    \begin{eqnarray*}
 \left ( \frac{A'}{Q}  \right ) = \left (\begin{array}{cccc}
            \vdots & \vdots & \dots & \vdots \\
            \vdots & \vdots & \dots & \vdots \\
            \vdots & \vdots & \ddots & \vdots \\
            b_1 & b_2 & \dots & b_k \\
            \hline
            Q_{11} & Q_{12} &\dots & Q_{1k} \\
            Q_{21} & Q_{22} &\dots & Q_{2k} \\
            \vdots & \vdots & \ddots & \vdots \\
            Q_{k1} & Q_{k2} & \dots & Q_{kk}
        \end{array}  \right ),
    \end{eqnarray*}
    where $b_1 = \dots = b_{k-1} = 0$, $b_k =g$, and $q:=\det Q = (-1)^j$, where $j$ is the total number of columns swaps we have performed in the above procedure. It is now easy to see that our desired matrix is gotten by setting $M=M'Q^{-1}$ where
    \begin{equation*}
        M' = \begin{pmatrix}
        q & 0 & \dots & 0 \\
         0 & 1 & \dots & 0 \\
         \vdots & \vdots & \ddots & \vdots \\
         b_1 & b_2 & \dots & b_k
     \end{pmatrix}.
     \end{equation*}
\end{proof}

\begin{lemma}{\label{lem:reduce_one}}
   For a positive integer $k$, let $S\subseteq\mathbb{Z}^k$. If $S$ is contained in some $k-1$ dimensional affine subspace of $\mathbb{R}^k$, then there exists $\Phi \in \Aut (\mathbb{Z}^k)$ such that
    \begin{eqnarray*}
        \Phi (S) \subseteq \mathbb{Z}^{k-1} \times \{w\}
    \end{eqnarray*}
    for some $w \in \mathbb{Z}$.
\end{lemma}
\begin{proof}
    It suffices to find a matrix $\Phi \in \Aut(\mathbb{Z}^k)$ such that, for some fixed $w\in\mathbb{Z}$,
\begin{equation*}
\Pi_k(\Phi x)=w
\end{equation*}
for all $x\in S$; here, $\Pi_k$ denotes the $k$th-coordinate projection. When $S$ is empty or $k=1$, the statement is immediate. Otherwise, let $b_1,\dots, b_l\in \mathbb{Z}^k$  (for $1\leq l<k$) be a basis of the subspace generated by $S-x_0$ for some fixed $x_0\in S$. We claim that there is $a=(a_1,a_2,\dots,a_k)\in\mathbb{Z}^k$ whose entries are relatively prime and for which
    \begin{equation}\label{eq:OrtAutLem}
        a\cdot b_j = 0
    \end{equation}
    for every $j=1,\dots,l$. To see this, using our hypothesis that $l<k$, the standard algorithm to produce orthogonal subspaces via Gaussian elimination gives $a\in\mathbb{Q}^k$ for which \eqref{eq:OrtAutLem} holds. If this algorithm yields $a\in\mathbb{Z}^k$, we can factor out the greatest common divisor of entries, if necessary, to produced our desired $a$. Otherwise, let $n$ be the smallest natural number for which $na\in\mathbb{Z}^k$ and observe that the entries of $na$ must be relatively prime for otherwise a smaller $n$ could have been gotten by factoring out the greatest common divisor. With this $n$ in hand, our desired element is produced by replacing $a$ by $na$.
    
    Since each $x\in S$ can be expressed in the form $x=(x-x_0)+x_0=\beta_1b_1+\beta_2b_2+\cdots \beta_lb_l+x_0$, we have
    \begin{equation*}
    \Phi x=\beta_1\Phi b_1+\beta_2\Phi b_2+\cdots+\beta_l\Phi b _l+\Phi x_0
    \end{equation*}
    for any $k \times k$ matrix $\Phi$. Thus, if the last row of $\Phi$ is $a=(a_1,a_2,\dots,a_k)$, we immediately find that
    \begin{equation*}
    \Pi_k(\Phi x)=\Pi_k(\Phi x_0)=:w
    \end{equation*}
    for all $x\in S$ by virtue of \eqref{eq:OrtAutLem}. With this observation, we appeal to Lemma \ref{isom_matrix} to produce $\Phi\in\Gl(\mathbb{Z}^k)$ whose last row is $a=(a_1,a_2,\dots,a_k)$ and has $\det(\Phi)=\gcd(a_1,a_2,\dots,a_k)$. Since the entries of $a$ are relatively prime, $\det(\Phi)=1$ which ensures that $\Phi^{-1}$ has integer entries. Consequently, $\Phi\in\Aut(\mathbb{Z}^k)$ and our proof is complete.
\end{proof}


\begin{proof}[Proof of Lemma \ref{lem:reduce_many}]
If $d=k-1$, an appeal to Lemma \ref{lem:reduce_many} gives $\Phi=\Phi_1\in\Aut(\mathbb{Z}^k)$ for which
\begin{equation*}
\Phi(S)\subseteq \mathbb{Z}^{k-1}\times \{w_1\}
\end{equation*}
for $w_1\in\mathbb{Z}$. To verify that $\Proj_{\mathbb{Z}^{k-1}}(\Phi(S))$ is $d=k-1$ dimensional, observe that
\begin{equation}\label{eq:reduce_many1}
\Phi(S-x_0)=\Phi(S)-\Phi(x_0)\subseteq \mathbb{Z}^{k-1}\times \{w_1-w_1\}=\mathbb{Z}^{k-1}\times\{0\}
\end{equation}
for $x_0\in S$. Since $S-x_0$ is $d=k-1$ dimensional and $\Phi$ is an automorphism, it follows that  the image of $S-x_0$ under $\Phi$ must be $d=k-1$ dimensional. In view of \eqref{eq:reduce_many1},  the only way for this to happen is for the first $k-1$ components of (the elements in) $\Phi(S-x_0)$ to span $\mathbb{Z}^{k-1}$, a property which is clearly equivalent to the stated assertion.

 In the case that $d=k-2$, we appeal to Lemma \ref{lem:reduce_one} to produce $\Phi_1\in\Aut(\mathbb{Z}^k)$ for which
\begin{equation*}
\Phi_1(S)\subseteq \mathbb{Z}^{k-1}\times\{w_1\}
\end{equation*}
for $w_1\in\mathbb{Z}$. Since $d=k-2$, $\Proj_{\mathbb{Z}^{k-1}}(\Phi_1(S))$ is supported in a proper affine subspace of $\mathbb{Z}^{k-1}$. We may therefore apply Lemma \ref{lem:reduce_one} again to produce $\widetilde{\Phi_2}\in\Aut(\mathbb{Z}^{k-1})$ having
\begin{equation*}
\widetilde{\Phi_2}(\Proj_{\mathbb{Z}^{k-1}}(\Phi_1(S)))\subseteq\mathbb{Z}^{k-2}\times \{w_2\}
\end{equation*}
for $w_2\in\mathbb{Z}$. With this, we define
\begin{equation*}
\Phi_2=\begin{pmatrix}
\widetilde{\Phi_2} & 0\\
0 & 1
\end{pmatrix}\in \Aut(\mathbb{Z}^k)
\end{equation*}
and $\Phi=\Phi_2\Phi_1$. Clearly, $\Phi\in \Aut(\mathbb{Z}^k)$ and
\begin{equation*}
\Phi(S)\subseteq \Phi_2(\Proj_{\mathbb{Z}^{k-1}}(\Phi_1(S))\times\{w_1\})=\widetilde{\Phi_2}(\Proj_{\mathbb{Z}^{k-1}}(\Phi_1(S))\times \{w_1\}\subseteq \mathbb{Z}^{k-2}\times \{w_2\}\times\{w_1\}
\end{equation*}
and therefore
\begin{equation*}
\Phi(S)\subseteq \mathbb{Z}^{d}\times\{w\}
\end{equation*}
for $d=k-2$, and $w=(w_1,w_2)\in\mathbb{Z}^2$. As in the previous case, our hypothesis that $S$ is $d=k-2$ dimensional, the above inclusion, and the property that $\Phi$ is an automorphism guarantee that the first $d=k-2$ entries of $\Phi(S)-\Phi(x_0)$ for $x_0\in S$ must contain enough elements to span $\mathbb{Z}^{d}$. Consequently, $\Proj_{\mathbb{Z}^d}(\Phi(S))$ is $d$ dimensional.

For the remaining $d$'s, we are able to continue this process inductively by continued application of Lemma \ref{lem:reduce_one}. In general, $\Phi\in\Aut(\mathbb{Z}^k)$ is produced by setting $\Phi=\Phi_{k-d}\Phi_{k-d-1}\cdots\Phi_1$ where each $\Phi_j\in\Aut(\mathbb{Z}^k)$ has the form
\begin{equation*}
\Phi_j=\begin{pmatrix} \widetilde{\Phi_j} & 0\\
0 & I_{j}
\end{pmatrix}
\end{equation*}
where $\widetilde{\Phi_j}\in\Aut(\mathbb{Z}^{k-j+1})$ is produced using Lemma \ref{lem:reduce_one} applied to the set 
\begin{equation*}
\Proj_{\mathbb{Z}_{k-j+1}}(\Phi_{j}\Phi_{j-1}\dots \Phi_1(S))
\end{equation*}
and $I_j$ denotes the $(j-1)\times (j-1)$ identity matrix. These matrices progressively preserve final entries (labeled by $w_j$) by which we find that\
\begin{equation*}
\Phi(S)\subseteq \mathbb{Z}^d\times \{w\}
\end{equation*}
where $w=(w_{k-d},w_{k-d-1},\dots,w_1)$. Finally, exactly as described in the previous cases, we find $\Proj_{\mathbb{Z}^d}(\Phi(S))$ to be $d$ dimensional.
\end{proof}

\end{document}